\documentclass[11pt, leqno]{amsbook}
\usepackage{graphicx}
\usepackage{amsfonts,delarray,amssymb,amsmath,amsthm,a4,a4wide, appendix}
\usepackage{latexsym}
\usepackage{epsfig}
\usepackage{color}

\usepackage[margin=1in]{geometry} 
\newtheorem{thm}{Theorem}[section]
\newtheorem{cor}[thm]{Corollary}
\newtheorem{lem}[thm]{Lemma}
\newtheorem{conj}[thm]{Conjecture}
\newtheorem{exam}[thm]{Example}
\newtheorem{quest}[thm]{Question}
\newtheorem{prop}[thm]{Proposition}
\theoremstyle{definition}
\newtheorem{defn}[thm]{Definition}
\newtheorem{rem}[thm]{Remark}
\newtheorem{prob}[thm]{Problem}
\numberwithin{equation}{section}
\newcommand{\norm}[1]{\left\Vert#1\right\Vert}
\newcommand{\abs}[1]{\left\vert#1\right\vert}

\newcommand{\R}{\mathbb R}

\newcommand{\e}{\varepsilon}
\newcommand{\eps}{\varepsilon}
\newcommand{\ov}{\overline}
\newcommand{\p}{\partial}

\newcommand{\comment}[1]{}
\def\h{\hspace*{.24in}}

\newenvironment{myindentpar}[1]%
{\begin{list}{}%
         {\setlength{\leftmargin}{#1}}%
         \item[]%
}
{\end{list}}
\makeatletter
\def\l@subsection{\@tocline{2}{0pt}{2.5pc}{5pc}{}}
\makeatother

\begin{document}

\title[The second boundary value problem and linearized Monge-Amp\`ere equation ]{ \Huge The second boundary value problem of the prescribed affine mean curvature equation and
 related linearized Monge-Amp\`ere equation}
\author{Nam Q. Le}
\address{Department of Mathematics, Indiana University,
Bloomington, 831 E 3rd St, IN 47405, USA.}
\email{nqle@indiana.edu}

\begin{abstract}

These lecture notes are concerned with the solvability of 
the second boundary value problem of the prescribed affine mean curvature equation and 
 related regularity theory of the Monge-Amp\`ere and  linearized Monge-Amp\`ere equations.

 The prescribed affine mean curvature equation is a fully nonlinear, fourth order, geometric partial differential equation of the following form
 $$\sum_{i, j=1}^n U^{ij}\frac{\partial^2}{\partial {x_i}\partial{x_j}}\left[(\det D^2 u)^{-\frac{n+1}{n+2}}\right]=f$$
 where $(U^{ij})$ is the cofactor matrix of the Hessian matrix $D^2 u$ of a locally uniformly convex function $u$. 
 Its variant is related to the problem of finding K\"ahler metrics of constant scalar curvature in complex geometry. 
 
We first introduce the background of the prescribed affine mean curvature equation which can be viewed as a coupled system of Monge-Amp\`ere and  linearized Monge-Amp\`ere equations. 
Then we state key open problems and present the solution of the  second boundary value problem that prescribes the boundary values of the solution $u$ and its Hessian determinant 
$\det D^2 u$. Its proof uses important tools from the boundary regularity theory
 of the Monge-Amp\`ere and linearized Monge-Amp\`ere equations. Next, we present the regularity theory of the linearized Monge-Amp\`ere equations
initiated by Caffarelli and Guti\'errez. 
We discuss the background and provide the proof of Caffarelli-Guti\'errez's
fundamental interior Harnack inequality and Holder estimates for solutions to the linearized Monge-Amp\`ere equation. The corresponding
global H\"older estimates are also proved. 
These proofs use geometric properties, including the engulfing property, of sublevel sets of solutions to the Monge-Amp\`ere equation. 
Finally, we present the proofs of these geometric properties of solutions to the Monge-Amp\`ere equation, culminating in the proof of Caffarelli's celebrated $C^{1,\alpha}$ regularity theorem.
\end{abstract}

\maketitle
\vglue 2cm
\noindent
\thanks
\vglue 0.4cm
\noindent
\tableofcontents

\section*[Notation]{Notation}
We collect here several standard notations used in the lecture notes.\\
$\bullet$ Partial differentiations:
$$\displaystyle\p_i =\frac{\p}{\p x_i}, \p_{ij} =\frac{\p^2}{\p x_i\p x_j}.$$
$\bullet$ Convex function: A function $u:\Omega\subset \R^n\rightarrow\R$ is convex if for all $0\leq t\leq 1$, and any $x, y\in\Omega$ such that $tx+ (1-t)y\in\Omega$ we have $$u(
tx + (1-t)y)\leq tu(x) + (1-t) u(y).$$
$\bullet$ Subgradient: For a convex function $u$, we use $\nabla u(x)$ to denote a subgradient of the graph of $u$ at $(x, u(x))$, that is, for all $y$
in the domain of $u$, we have $$u(y)\geq u(x) + \nabla u(x) \cdot (y-x).$$ 
$\bullet$ Gradient vector:
$$D u = (\frac{\partial u}{\partial x_1}, \cdots, \frac{\partial u}{\partial x_n} ) = (u_1, \cdots, u_n).$$
$\bullet$ Hessian matrix:
$$D^2 u =(u_{ij})= \left(\frac{\p^2 u}{\p x_i \p x_j}\right)_{1\leq i, j\leq n}.$$
$\bullet$ $I_n$ is the identity $n\times n$ matrix.\\
$\bullet$ $B_r(a)$ denotes the ball in $\R^n$ with center $a$ and radius $r$.\\ 
$\bullet$ Euclidean norm: If $x=(x_1,\cdots, x_n)\in \R^n$ then $$|x|= \left(\sum_{i=1}^n x_i^2\right)^{\frac{1}{2}}.$$
$\bullet$ Dot product: the dot product of $x=(x_1,\cdots, x_n), y= (y_1,\cdots, y_n) \in\R^n$ is $$\displaystyle x\cdot y= \sum_{i=1}^n x_i y_i.$$
$\bullet$ $\omega_n$ is the volume of the unit ball in $\R^n$.\\
$\bullet$ $|\Omega|$ denotes the Lebesgue measure of a measurable set $\Omega\subset\R^n$.\\
$\bullet$ Repeated indices are summed. For example
$$a^{ij} u_{ik}=\sum_{i} a^{ij} u_{ik}.$$\\
$\bullet$ Compact inclusion: If $A\subset B\subset\R^n$ and $\overline{A}\subset B$, then we write $A\subset\subset B$.\\
$\bullet$ Lebesgue space: $L^{p}(\Omega)$ is the Banach space consisting of all measurable functions $f$ on $\Omega$ that are $p$-integrable. The norm of $f$ is defined by
$$\|f\|_{L^{p}(\Omega)}=\left( \int_\Omega |f(x)|^p dx\right)^{\frac{1}{p}}.$$
$\bullet$ Sobolev space: 
\begin{multline*}W^{k, p}(\Omega)=
\Bigl\{u\in L^p(\Omega), D^{\alpha} u\in L^p(\Omega)~\text{for all multi-indices } \alpha= (\alpha_1,\cdots,\alpha_n)~\\ \text{with length~}|\alpha|=\alpha_1+\cdots+\alpha_n\leq k
\Bigr\}.
\end{multline*}
$\bullet$ H\"older space: $C^{\alpha}(\overline{\Omega})$ ($0<\alpha\leq 1$) consists of continuous functions $u$ that are uniformly H\"older continuous with exponent $\alpha$
in $\bar \Omega$. The $C^{\alpha}(\overline{\Omega})$ norm of $u$ is
$$\displaystyle \|u\|_{C^{\alpha}(\overline \Omega)}:= \sup_{x\in\overline\Omega}|u(x)| + \sup_{x\neq y\in\overline\Omega}\frac{|u(x)-u(y)|}{|x-y|^{\alpha}}.$$
$\bullet$ Higher order H\"older space: $C^{k,\alpha}(\overline{\Omega})$ consists of $C^{k} (\overline{\Omega})$ functions whose $k$-th order partial derivatives are 
uniformly H\"older continuous with exponent $\alpha$
in $\bar \Omega$.\\
$\bullet$ $\overline E:$ the closure of a set $E$.\\  
$\bullet$ $\p E:$ the boundary of a set $E$.\\   
$\bullet$ $\text{diam} (E):$ the diameter of a bounded set $E$.\\ 
$\bullet$ $\text{dist} (\cdot, E):$ the distance function from a closed set $E$.\\ 
$\bullet$ $A\geq B$ for symmetric $n\times n$ matrices $A$ and $B$: if the eigenvalues of $A-B$ are nonnegative.\\
$\bullet$ $\text{trace} (M):$  the trace of a matrix  $M$.\\
$\bullet$ $\|M\|:$ the Hilbert-Schmidt norm of a symmetric $n\times n$ matrice $M$: $\|M\|^2=\text{trace}(M^T M)$.\\
\newpage
\section*[Introduction]{Introduction}
These lecture notes, consisting of three parts, are concerned with 
the second boundary value problem of the prescribed affine mean curvature equation and 
 related regularity theory of the linearized Monge-Amp\`ere equation. 
 
 In Part 1, we discuss the affine maximal surface equation and its associated boundary value problems. 
 The affine maximal surface equation is a fully nonlinear, fourth order, geometric partial differential equation of the following form
 $$\sum_{i, j=1}^n U^{ij}\frac{\partial^2}{\partial {x_i}\partial{x_j}}\left[(\det D^2 u)^{-\frac{n+1}{n+2}}\right]=0$$
 where the unknown is a locally uniformly convex function $u$ and $(U^{ij})_{1\leq i, j\leq n}$ is the cofactor matrix of the Hessian matrix $D^2 u$. It arises naturally in affine 
 differential geometry while its analogues appear in the problem of finding K\"ahler metrics of constant scalar curvature in complex geometry which witnesses intensive interest in 
 recent years. The left hand side of the above equation is a constant multiple of the affine mean curvature of the graph of $u$.
 We will consider a more general equation, called
the prescribed affine mean curvature equation which prescribes the affine mean curvature of the graph of 
a locally uniformly convex function $u$ defined on a bounded, smooth, strictly convex domain. 
 As far as boundary value problems are concerned, so far
 only the second boundary value problem has been more or less well understood
 in all dimensions. We present the proof of the solvability of this second boundary value problem in our main result of Part 1, Theorem \ref{mainthm}. 
 Its proof uses tools from the boundary regularity theory
 of the Monge-Amp\`ere and linearized Monge-Amp\`ere equations. 
 
 Part 2 of the notes will be devoted to the regularity theory of the linearized Monge-Amp\`ere equations
initiated by Caffarelli and Guti\'errez. These equations are of the form
$$\sum_{i, j=1}^n U^{ij}\frac{\partial^2}{\partial {x_i}\partial{x_j}} v=g$$
where $(U^{ij})_{1\leq i, j\leq n}$ is the cofactor matrix of the Hessian matrix $D^2 u$ of a locally uniformly convex function $u$ solving the Monge-Amp\`ere equation
$$\lambda\leq \det D^2 u\leq\Lambda$$
where $\lambda$ and $\Lambda$ are two positive constants. 
Caffarelli and Guti\'errez's theory has already played a crucial role in 
Trudinger and Wang's resolution of Chern's conjecture in affine geometry concerning affine maximal hypersurfaces in $\R^3$ and in Donaldson's interior estimates for Abreu's equation 
in his resolution, using the continuity method, of the constant scalar curvature problem for toric surfaces. It was also used by Caffarelli and Silvestre in one of their
pioneering papers on nonlocal equations to prove several regularity results for problems involving the fractional powers of 
the Laplacian or other integro-differential equations.
We will present the background and proof of Caffarelli-Guti\'errez's
fundamental Harnack inequality for nonnegative solution to the homogeneous equation linearized Monge-Amp\`ere equation, Theorem \ref{CGthm}. 
We will also give the proof of  another main result of Part 2, Theorem \ref{global-h}, which  is concerned with global H\"older estimates for the inhomogeneous linearized Monge-Amp\`ere equation, extending Caffarelli-Guti\'errez's interior estimates
to the boundary. Together with Trudinger and Wang's global $C^{2,\alpha}$ estimates for the Monge-Amp\`ere equation, Theorem \ref{global-h} plays a crucial role
in  the proof of Theorem \ref{mainthm}.

In Part 3, we present from scratch  basic and well-known facts regarding geometric properties of solutions to the Monge-Amp\`ere equations used in Part 2.  A very important concept in this part is the notion of sections
of a convex function. They are sublevel sets of a convex function after subtracting a supporting hyperplane. Their role in the regularity 
theory of linearized Monge-Amp\`ere equation is the same as that of balls in the regularity theory of linear, uniformly elliptic equations. Sections of solutions to the Monge-Amp\`ere equation can have degenerate geometry
but they share many crucial properties as Euclidean balls such as precise volume estimate and engulfing property. We restrict ourselves to developing tools in the Monge-Amp\`ere
equation to prove these remarkable properties of sections. Central topics in this Part 3 include Caffarelli's celebrated $C^{1,\alpha}$ regularity of strictly convex solutions
and the engulfing properties of sections.

The notes
are self-contained, except for Caffarelli's striking interior $C^{2,\alpha}$ estimates for the Monge-Amp\`ere equation,
Trudinger and Wang's important global $C^{2,\alpha}$ estimates for the Monge-Amp\`ere equation 
when the Monge-Amp\`ere measure is only assumed to be globally $C^{\alpha}$ and Savin's deep Localization theorem at the boundary for the Monge-Amp\`ere equation 
with bounded right hand side.

These lecture notes expand and update materials covered in seminars at Columbia, Kansas State, Rutgers, and the following mini-courses and lecture series:
\begin{myindentpar}{1cm}
1. ``The linearized Monge-Amp\`ere equation and its geometric
applications'' given at the Workshop on Geometric Analysis and Nonlinear PDEs at  Rutgers University, from May 1-5,  2013.\\
2. ``From a classical theorem of J\"{o}rgens, Calabi, and Pogorelov to the affine Bernstein problem'' given at the
Vietnam Institute for Advanced Study in Mathematics, Hanoi, Vietnam from July 01-August 31, 2013.\\
3. ``The linearized Monge-Amp\`ere equation and its geometric
applications'' given at the Institute of Mathematics, Vietnam Academy of Science and Technology, Hanoi from October to November 2013.\\
4. ``The Monge-Amp\`ere equation and its geometric
applications'' given in a Special topics course at Indiana University, Bloomington during the Spring semseter in 2016.
\end{myindentpar}
I would like to thank Ovidiu Savin, Diego Maldonado, Natasa Sesum, Longzhi Lin, Ng\^o Bao Ch\^au, L\^e Tu\^an Hoa, and Nguy\^en Minh Tr\'i and their institutions for the kind invitations and hospitality.  
My special thanks go to Ovidiu Savin for his enlightening insights and many interesting discussions on the linearized Monge-Amp\`ere equations, and their boundary regularity 
in particular. Oversights or inaccuracies, if any, in these notes, are mine.
\vglue 0.5cm

{\bf Acknowledgements.} The author would like to thank the anonymous referees for their careful reading of the manuscript and for their insightful comments that improve the exposition.
The research of the author was supported in part by the National Science Foundation under grant DMS-1500400.

\part{The affine Bernstein and boundary value problems}
\section{The affine Bernstein and boundary value problems}
In this section, we introduce the affine Bernstein and boundary value problems for affine maximal graphs. The solutions or partial solutions to these problems bring us to the realm
of the regularity theory of the linearized Monge-Amp\`ere equation and the second boundary value problem for affine maximal graphs.
\subsection{Minimal graph}
We first start by briefly recalling the Bernstein and boundary value problems for minimal surfaces. They serve as historical perspectives for the results, problems
and conjectures concerning affine maximal surfaces. 

Suppose $u$ is a real-valued function defined on a bounded domain $\Omega\subset\R^n$. Then the
area of the graph of $u$ on $\Omega$ is given by the formula:
$$A(u,\Omega)=\int_{\Omega} \sqrt{1+\abs{D u}^2}.$$
This is a convex functional.
The graph is called minimal if $u$  is a critical point of $A$ under local perturbations, that is, for all smooth functions $\varphi\in C^{\infty}_0(\Omega)$, we have
$$\displaystyle\frac{d}{dt}\mid_{t=0} A(u+ t\varphi, \Omega)=0.$$
This is equivalent to
$$\int_{\Omega}\frac{Du\cdot D\varphi}{\sqrt{1+ |Du|^2}}=0,~\text{for all} ~ \varphi\in C^{\infty}_0(\Omega),$$
or
$u$ satisfies the Euler-Lagrange equation of the area functional $A(\cdot,\Omega)$:
\begin{equation}
\label{MSE}
\sum_{i=1}^{n}\frac{\partial}{\partial_{x_i}} \left(\frac{u_i}{\sqrt{1+\abs{D u}^2}}\right)=0.
\end{equation}
From the Partial Differential Equations (PDE) viewpoint, two of the main problems concerning (\ref{MSE}) are the Bernstein and Dirichlet problems. The first problem stemmed 
from a geometric result of Bernstein \cite{Berns} around 1915-1917 which says that
an entire, two dimensional, minimal graph must be a hyperplane. Hence, solutions in $\R^2$ of (\ref{MSE}) are affine functions.

\begin{myindentpar}{1cm}
{\bf 1.} The Bernstein problem asks if global solutions (defined on the whole space $\R^n$) of (\ref{MSE}) are affine functions, that is, if $u(x)= a + b\cdot x$ where $a\in\R$
and $b\in\R^n$.\\
{\bf 2.} The Dirichlet problem seeks a minimal graph with given values on the boundary of a domain $\Omega$, that is, to solve the following boundary value problem for minimal surfaces:
\begin{equation}
  \left\{ 
  \begin{alignedat}{2}\sum_{i=1}^{n} \frac{\p}{\p x_i} \left(\frac{u_i}{\sqrt{1+ |Du|^2}}\right)~& =0 ~&&\text{in} ~\Omega, \\\
u ~&=\varphi~&&\text{on}~\p \Omega.
\end{alignedat}
\right.
\label{MSBVP}
\end{equation}
\end{myindentpar}
The search for a complete answer to the Bernstein problem has an interesting history. As mentioned above,
Bernstein gave an affirmative answer for $n=2$ in 1915-1917. 
New proof of Bernstein's theorem was given by Fleming \cite{Fl} in 1962. The combined effort of 
 De Giorgi \cite{DG2} (for $n=3$ in 1965), Almgren \cite{Alm} (for $n=4$ in 1966), and  Simons \cite{Si} (for $n\leq 7$ in 1968) settled the Bernstein problem in the affirmative for all $n\leq 7$.
 The Bernstein problem has a negative answer for all $n\geq 8$ by a counterexample of Bombieri, De Giorgi, and Giusti \cite{BDGG} in 1969.
 
The Dirichlet problem was completely solved by Jenkins-Serrin \cite{JS} in 1968. They proved the following surprising 
and beautiful result: 
\begin{thm} (\cite{JS})
\label{JSthm} Let $\Omega$ be a bounded $C^{2,\alpha}$ domain in $\R^n$ where $0<\alpha<1$.
A smooth solution for (\ref{MSBVP}) exists for arbitrary $\varphi\in C^{2,\alpha}(\overline{\Omega})$ if and only if $\p\Omega$ has everywhere non-negative mean curvature.
\end{thm}
\subsection{Affine maximal graph}
In affine geometry, the affine area of the graph of a smooth, convex $u$ defined on $\Omega$ is
\begin{equation}\mathcal{A}(u,\Omega) =\int_{\Omega}[\det D^2 u(x)]^{\frac{1}{n+ 2}}dx.
 \label{calAfunc}
\end{equation}
We digress for a moment to briefly comment on the geometric meaning of the affine area and its affine invariance; see Calabi \cite{Calabi} for more information.

On the graph $\mathcal{M}=\{(x, u(x)| x\in\Omega)\}$ of $u$, we define the affine metric $g=\left(g_{ij}\right)$, where
$$g_{ij}= \frac{u_{ij}}{(\det D^2 u)^{\frac{1}{n+2}}}.$$
Due to the identity
$$[\det D^2 u]^{\frac{1}{n+2}}= [\det (g_{ij})]^{1/2},$$
the integrand $[\det D^2 u]^{\frac{1}{n+2}} dx$ of the affine area functional $\mathcal{A}$ is the volume form $\sqrt{\det (g_{ij})} dx$ of $\mathcal{M}$ induced by the metric $g$.

The metric $g$ and the area $\mathcal{A}$ are invariant under
unimodular affine transformations on $\R^{n+1}$, that is, linear transformations in $\R^{n+1}$ preserving Euclidean volume and orientation.
For reader's convenience, we verify in Lemma \ref{A_inv} the above mentioned invariant property of $\mathcal{A}$.

The affine area functional $\mathcal{A}$ is concave (see Lemma \ref{concavelem}), i.e., 
$$\mathcal{A}(t u + (1-t)v,\Omega)\geq t \mathcal{A}(u,\Omega) + (1-t)\mathcal{A}(v,\Omega), 0\leq t\leq 1.$$
Critical points of $\mathcal{A}$ are maximizers under local perturbations. Locally uniformly convex maximizers
satisfy the Euler-Lagrange equation
\begin{equation}\sum_{i, j=1}^{n}\frac{\partial^2}{\partial x_i \partial x_j}(U^{ij} w)=0, w = [\det D^2 u]^{-\frac{n+1}{n +2}}
\label{AMSE}
\end{equation}
where $U= (U^{ij})$ denotes the matrix of cofactors of the Hessian matrix
$D^2 u := (\frac{\partial^2 u}{\partial x_i\partial x_j}) = (u_{ij}).$ When $u$ is locally uniformly convex, 
$$U = (\det D^2 u) (D^2 u)^{-1}.$$
See Lemma \ref{ELlem} for a brief derivation of the Euler-Lagrange equation (\ref{AMSE}). 

The fourth order equation (\ref{AMSE}) can also be viewed as a system of two second order partial differential equations. 
Regarded as a linear elliptic equation for $w$, it is a linearized Monge-Amp\`ere equation, since the coefficient matrix
$U$ comes from linearization of the Monge-Amp\`ere operator: $$U=\frac{\p\det D^2 u}{\p u_{ij}}.$$
The other equation in the system is the Monge-Amp\`ere equation for $u$:
$$\det D^2 u = w ^{-\frac{n+2}{n+1}}.$$ 
Since the matrix $U$ is divergence free (see Lemma \ref{divfreeU}), that is
$\displaystyle \sum_{j=1}^{n}\partial_{j}U^{ij}=0$ for all $i$,
we can rewrite (\ref{AMSE}) in the following form 
\begin{equation}H_{\mathcal{A}}[u]:= -\frac{1}{n+1}\sum_{i, j=1}^{n}U^{ij}w_{ij}=0.
 \label{HAu}
\end{equation}
The quantity  $H_{\mathcal{A}}[u]$ represents the affine mean curvature of the graph of $u$ \cite{NS, SiAff}. As a result, (\ref{AMSE}) is called the {\it affine maximal surface equation}
\cite{TW00}. The graph of the function $u$ satisfying (\ref{AMSE}) is then called the affine maximal graph. 

In 1977, Chern \cite{Chern} made the following conjecture: 
\begin{conj}[Chern's conjecture]
An affine maximal graph of a smooth, locally uniformly convex function on the 2-dimensional Euclidean space, $\R^2$,
must be an elliptic paraboloid. 
\end{conj}
This conjecture is known as the Bernstein problem for affine maximal hypersurfaces in $\R^3$.  We can also consider a more general version
of Chern's conjecture on $\R^n$.
The Bernstein problem for affine maximal hypersurfaces was also proposed by Calabi \cite{Calabi}.
\subsection{The affine Bernstein problem}
In PDE language, the Bernstein problem for affine maximal hypersurfaces, which we call {\it the affine Bernstein problem} for short, is equivalent to the following question.
\begin{quest} Suppose $u$ is a smooth, locally uniformly convex solution in $\R^n$ $(n\geq 1)$ of
$$\sum_{i, j=1}^{n}U^{ij} w_{ij}=0, ~w=[\det D^2 u]^{-\frac{n+1}{n +2}},~\text{and}~U= (U^{ij})= (\det D^2 u)(D^2u)^{-1}.$$
Is $u$ an elliptic quadratic polynomial?
\end{quest}
Here we call a quadratic polynomial
$$P(x)= c + b\cdot x + \sum_{i, j=1}^n\frac{1}{2} a_{ij} x_i x_j~(c\in\R,~b\in\R^n,~ a_{ij}= a_{ji}\in\R)$$
{\it elliptic} if its Hessian matrix $(a_{ij})_{1\leq i, j\leq n}$ is positive definite.

It is easy to see that, in $1D$, the affine Bernstein problem has an affirmative answer. Indeed, when $n=1$, $U=1$ in (\ref{AMSE}) and hence 
the equation (\ref{AMSE}) implies that $w$ is an affine function, that is, $w(x)= a+ bx$ where $a,b\in\R$. Using the positivity of $w$, we 
conclude that $w$ is a positive constant. Thus $u$ must be a quadratic polynomial with positive opening.

In 2000, Trudinger and Wang \cite{TW00} proved that the affine Bernstein problem has an affirmative answer in dimensions $n=2$ 
and thus settling Chern's conjecture. They also showed that a corresponding result holds in higher dimensions provided that a uniform, {\it strict convexity} condition on the solution
$u$ holds. However, they produced a (non-smooth) counterexample for $n\geq 10.$ The convex function $u(x)= \sqrt{|x'|^9 + x_{10}^2}$, where
$x'= (x_1, \cdots, x_9)$, satisfies (\ref{AMSE}) in $\R^{10}$ and is not differentiable at the origin. 

For reader's convenience, we provide in Appendix B a heuristic explanation
of this non-smooth example using simple symmetry and scaling arguments. This explanation is reminiscent of Pogorelov's singular solution
of the form $u(x', x_n)=|x'|^{2-2/n} f(x_n)$ to the Monge-Amp\`ere equation $\det D^2 u=1$; see \cite{Pogo}.

Trudinger and Wang \cite{TW3} made the following conjecture.
\begin{conj} If $n\leq 9$ then any smooth, locally uniformly convex solution $u$ in $\R^n$ of
$$\sum_{i, j=1}^{n}U^{ij} w_{ij}=0, ~w=[\det D^2 u]^{-\frac{n+1}{n +2}},~\text{and}~U= (U^{ij})= (\det D^2 u)(D^2u)^{-1}$$
is an elliptic quadratic polynomial.
If $n\geq 10$, then there is a smooth, locally uniformly convex solution $u$ in $\R^n$ of the above equation which is not an elliptic quadratic polynomial.
\end{conj}
The affine maximal surface equation is fourth-order; while the minimal surface equation is second-order. In terms of orders and the 
dimensions of potential counterexamples, the numerologies in the affine Bernstein problem are greater than those in the Bernstein problem by exactly 2, as can be see from $4 = 2 + 2$ and 
$10 = 8 + 2.$
\vglue 0.2cm
The key tool in Trudinger and Wang's resolution of Chern's conjecture is the theory of interior Harnack and 
H\"older estimates for linearized Monge-Amp\`ere equation initiated by Caffarelli 
and Guti\'errez \cite{CG}. We will discuss this theory in Section \ref{LMA_sec}. To get a flavor of this theory, we recall that the classical regularity theory of linear, 
uniformly elliptic equations of second order (in non-divergence form)
studies the equations of the form
$$\sum _{i, j=1}^{n}a^{ij}(x) \frac{\p^2 v}{\p x_i\p x_j}(x)=0$$
where the eigenvalues of the coefficient matrix $A(x)= (a^{ij}(x))$ are bounded between two positive constants $\lambda$, and $\Lambda.$ The 
linearized Monge-Amp\`ere theory studies the same
equation but with the bounds on the determinant of $A$, or equivalently the product of its eigenvalues, as the only {\it quantitative} assumption.
The theory of Caffarelli-Guti\'errez provides interior H\"older estimates similar to those of the classical theory provided that $A$ is matrix of cofactors of some convex function. 
\subsection{Connection with the constant scalar curvature problem}
We can consider a more general version of the affine area functional
$$\mathcal{A}_{\theta}(u,\Omega) = \int_{\Omega} \frac{[\det D^2 u]^{\theta}-1}{\theta}.$$
In the limit of $\theta\rightarrow 0$, using $\frac{t^{\theta}-1}{\theta}\rightarrow \text{log}~t$ for $t>0$, we obtain
the following functional 
$$\mathcal{A}_{0}(u,\Omega)=\int_{\Omega} \text{log} \det D^2 u.$$
This functional is the principal  part of the Mabuchi functional in complex geometry \cite{D1, Mabu}
$$M(u) = \int_{\Omega} -\text{log} \det D^2 u + \int_{\partial \Omega} u d\sigma - \int_{\Omega} f u dx.$$
Here $\sigma$ is some positive Radon measure supported on the boundary $\p\Omega$.
In the special case of
$$f ~\text{is a positive constant A},~\Omega~\text{is a polytope}~ P, $$
$$\sigma~
\text{is a measure on}~\p P~\text
{which is a multiple of the standard
Lebesgue measure on each face},$$
the existence of critical points of $M$ with certain boundary conditions implies the existence of a constant scalar curvature metric. Let us be a bit more precise here.\\
Critical points of $M$ satisfy the Abreu's equation \cite{Ab}
$$-\sum_{i, j=1}^{n} U^{ij} \partial_{ij}\left((\det D^2 u)^{-1}\right)=f.$$
Suppose that the polytope $P$ is defined by the linear inequalities $l_k(x)>c_k$ ($k=1,\cdots, m$) where $l_k$'s are linear functions and $c_k$'s are constants. 
Let $\delta_k(x)=l_k(x)- c_k$. We assume that the measure $\sigma$ and the positive constant $A$ satisfy the {\it stability condition}
$$\int_{\partial \Omega} u d\sigma - \int_{\Omega}  Au dx>0~\text{for all u convex but not affine}.$$
The constant scalar curvature problem for toric varieties is reduced to solving the Abreu's equation
\begin{equation}-\sum_{i, j=1}^{n} U^{ij} \partial_{ij}\left((\det D^2 u)^{-1}\right)=A
 \label{Abeq}
\end{equation}
with the Guillemin boundary condition
\begin{equation}u(x)-\sum_{k=1}^m \delta_k(x) \text{log } \delta_k(x)\in C^{\infty}(\overline{P}).
 \label{Gull}
\end{equation}
The problem (\ref{Abeq})-(\ref{Gull}) was solved by Donaldson in $n=2$ by an ingenious combination of geometric and PDE methods in a series of papers \cite{D1, D2, D3, D4}. It 
is completely open when $n>2$.
Rubin \cite{Rubin} established, by PDE methods,
the existence and boundary regularity away from the corners for equations of the type (\ref{Abeq})-(\ref{Gull}) in dimensions $n=2$.
\subsection{The first boundary value problem } The first boundary value problem for the affine maximal surface equation 
amounts to finding an affine maximal surface given by the graph of a convex function $u$ over 
$\Omega$ whose boundary value and gradient are given by those of $\phi$. Here $\phi\in C^2(\overline{\Omega})$ is a strictly convex function. 
The first boundary value problem for the affine maximal surface equation reads:
\begin{equation}
\left\{
 \begin{alignedat}{2}
   U^{ij}w_{ij} ~& = 0 ~&&\text{in} ~ \Omega, \\\
w ~&= (\det D^2 u)^{-\frac{n+1}{n+2}}~&&\text{in}~\Omega,\\\
u~&=\phi ~&&\text{on}~\partial \Omega,\\\
Du~&=D\phi ~&&\text{on}~\partial \Omega.
 \end{alignedat} 
  \right.
\label{FBV}
  \end{equation}
The solvability of (\ref{FBV}) is a major open problem. Recalling Jenkins and Serrin's solution of the boundary value problem for minimal surfaces in Theorem \ref{JSthm},
one might wonder if a similar affine invariant phenomenon occurs for (\ref{FBV}).

If we replace the last condition in (\ref{FBV}) with a more relaxed condition by requiring that the 
image of $\Omega$ under the mapping $Du$ is contained in that of $\overline{\Omega}$ under $D\phi$, that is, $Du(\Omega)\subset D\phi(\overline{\Omega})$, then we have a 
unique {\it weak solution v} for (\ref{FBV}) as proved by Trudinger and Wang in
\cite{TW05}. This solution $v$ is obtained as the unique maximizer of the affine area function $\mathcal{A}(\cdot,\Omega)$ (defined in (\ref{calAfunc})) in the set $\bar{S}[\phi,\Omega]$ consisting of convex functions
$v$ which satisfy $v=\phi$ on $\p\Omega$ and $Dv(\Omega)\subset D\phi(\overline{\Omega})$. Trudinger and Wang proved in \cite{TW05} (for $n=2$)
and \cite{TW4} (for all n) that $v$ is strictly convex in the interior of $\Omega$
and hence smooth. Two interesting open problems are:
\begin{myindentpar}{1cm}
 (1) The regularity of $v$ near the boundary of $\Omega$.\\
 (2) Does $v$ satisfy $Dv=D\phi$ on $\p\Omega$?
\end{myindentpar}
Motivated by Theorem \ref{JSthm}, Trudinger and Wang made the following conjecture:
\begin{conj} (\cite{TW05})
 The first boundary value problem (\ref{FBV}) has a unique smooth solution $u$ if the smooth, strictly convex function $\phi$ is affine mean convex, 
 that is, the affine mean curvature of the graph of $\phi$, $H_{\mathcal{A}}[\phi]$ as defined in 
 (\ref{HAu}), is positive.
\end{conj}

The crucial ingredient in establishing the interior regularity result for the weak solution $v$ to (\ref{FBV}) was the 
second boundary value problem of the prescribed affine mean curvature equation where the right hand side of the first equation of (\ref{FBV})
is replaced by a function $f$. This is the topic of the next Section \ref{SBV_sec}.
\subsection{The second boundary value problem of the prescribed affine mean curvature equation. } 
\label{SBV_sec}
Let $\Omega\subset\R^{n}$ be a bounded, smooth and strictly convex domain. We consider 
the prescribed affine mean curvature equation which prescribes the affine mean curvature of the graph of 
a locally uniformly convex function $u$ defined on 
$\Omega$. It can be written in the form
\begin{equation}
 \sum_{i, j=1}^{n}U^{ij}w_{ij}~ =f ~\text{in} ~\Omega, ~
 w~= (\det D^2 u)^{-\frac{n+1}{n+2}}~\text{in}~ \Omega
 \label{PAMCE}
\end{equation}
where as usual, $$U = (U^{ij})=(\det D^{2} u) (D^{2} u)^{-1}$$ is the matrix of cofactors of the 
Hessian matrix $D^{2}u= (u_{ij})$.

The second boundary value problem for (\ref{PAMCE}) prescribes the values of $u$ and its Hessian determinant $\det D^2 u$ on the boundary. We are thus led to 
the following
 fully nonlinear, fourth order, geometric partial differential equation:
\begin{equation}
  \left\{ 
  \begin{alignedat}{2}\sum_{i, j=1}^{n}U^{ij}w_{ij}~& =f ~&&\text{in} ~\Omega, \\\
 w~&= (\det D^2 u)^{-\frac{n+1}{n+2}}~&&\text{in}~ \Omega,\\\
u ~&=\varphi~&&\text{on}~\p \Omega,\\\
w ~&= \psi~&&\text{on}~\p \Omega.
\end{alignedat}
\right.
\label{AMCE1}
\end{equation}
It was 
introduced by Trudinger and Wang \cite{TW05} in 2005 in their investigation of the affine Plateau problem in affine geometry.  
Existence and regularity of solutions to (\ref{AMCE1}) are the key in studying the first boundary value problem for affine maximal surface equation.

 More generally, let
$G: (0,\infty) \rightarrow \mathbb{R}$ be a smooth, strictly increasing and strictly concave function on $(0, \infty)$. We consider
a fourth order, fully nonlinear, 
equation of the form
\begin{equation}
\label{AMCE}
L[u]:= U^{ij}w_{ij} =f,\quad w = G'(\det D^{2} u)\quad \text{in}~\Omega,
\end{equation}
where $u$ is a locally uniformly convex function in $\overline{\Omega}$.

The second boundary value problem for (\ref{AMCE}) is then 
\begin{equation}
\label{SBV}
u=\varphi,~~~ w=\psi~~~\text{on}~\p\Omega.
\end{equation}

The problem (\ref{AMCE1}) is a special case of (\ref{AMCE})-(\ref{SBV}) when we take $G(d)=\frac{d^{\theta}-1}{\theta}$ and $\theta=\frac{1}{n +2}$.
In the limiting case $\theta=0$ of $\frac{d^{\theta}-1}{\theta}$, we take $G(d)=\log d$ and (\ref{AMCE}) is then known as Abreu's equation in the 
context of existence of K\"ahler metric of constant scalar curvature \cite{Ab, CHLS, CLS, D1, D2, D3, D4, FS, Zhou, ZZ}. 

Observe that
(\ref{AMCE}) is the Euler-Lagrange equation, with respect to compactly supported perturbations, of 
the functional
\begin{equation}\label{functional}
J[u]= \int_{\Omega}  G(\det D^2 u)dx -\int_{\Omega}  u fdx,
\end{equation}
defined over strictly convex functions $u$ on $\Omega$. For simplicity, we call $L[u]$ in (\ref{AMCE}), where $G$ is a general concave function, the {\it generalized affine mean curvature} 
of the graph of $u$.

For a general concave function $G$, Donaldson \cite{D5} investigated local solutions of (\ref{AMCE}) with $f\equiv 0$ while Savin and the author \cite{LS}
studied regularity of (\ref{AMCE}) with Dirichlet and Neumann boundary conditions on $w$. In \cite{LS}, we considered (\ref{AMCE})
as an Euler-Lagrange equation of a Monge-Amp\`ere functional $E$ motivated by the Mabuchi functional in complex geometry:
$$E(u) =\int_{\Omega} -G(\det D^{2} u ) \, dx + \int_{\p\Omega} u d\sigma -\int_{\Omega} u dA.$$
Convex minimizers of $E$ satisfy a system of the form
\begin{equation}\label{EL-intro}
 \left\{
 \begin{alignedat}{2}
   G'(\det D^{2} u) ~&=v \h~&&\text{in} ~\Omega, \\\
 U^{ij} v_{ij}&= -dA \h~&&\text{in}~ \Omega,\\\
v &=0\h~&&\text{on}~\p \Omega,\\\
U^{\nu \nu} v_{\nu} &=- \sigma~&&\text{on}~\p \Omega,
 \end{alignedat}
 \right.
\end{equation}
where $U^{\nu\nu} = \det D^{2}_{x^{'}} u$ with $x' \perp \nu$ denoting the tangential directions along $\p \Omega$. 
A minimizer $u$ solves a fourth order elliptic equation with two nonstandard boundary conditions involving the second and third order derivatives of $u$. In \cite{LS} we apply the boundary H\"older gradient 
estimates established in \cite{LS1} and show that $u\in C^{2,\alpha}(\overline{\Omega})$ in dimensions $n=2$ under suitable conditions on the function $G$ and the measures $dA$ and $ d\sigma.$

It is an interesting problem, both geometrically and analytically, to study the solvability of the  fourth order, fully nonlinear equation
(\ref{AMCE})-(\ref{SBV}).
We recall the classical Schauder \cite{Sch} and Calderon-Zygmund \cite{CZ} theories of second order, linear, uniformly elliptic equations. A basic result in the Schauder theory is 
the following global $C^{2,\alpha}$ estimates. 
\begin{thm}(\cite[Theorems 6. 6 and 6.14]{GT})
\label{Schauthm}
 Let $\Omega$ be a $C^{2,\alpha}$ domain in $\R^n$, where $0<\alpha<1$. Let us consider the operator
 $$Lu=\sum_{i, j=1}^n a^{ij}(x) u_{ij}(x)$$
 where the coefficient matrix $(a^{ij})$ satisfies, for positive constants $\lambda,\Lambda$,
 $$\lambda I_n\leq (a^{ij})~\text{and } \|a^{ij}\|_{C^{\alpha}(\overline{\Omega})}\leq \Lambda.$$
 Then if $f\in C^{\alpha}(\overline{\Omega})$ and $\varphi\in C^{2, \alpha}(\overline{\Omega})$, the Dirichlet problem
 \begin{equation*}
 \left\{
 \begin{alignedat}{2}
   Lu ~&=f \h~&&\text{in} ~\Omega, \\\
 u&= \varphi \h~&&\text{on}~ \p\Omega,
 \end{alignedat}
 \right.
\end{equation*}
has a unique solution $u\in C^{2,\alpha}(\overline{\Omega})$ with the estimate
\begin{equation*}
 \|u\|_{C^{2, \alpha}(\overline{\Omega})} \leq C(\|u\|_{L^{\infty}(\Omega)} + \|f\|_{C^{\alpha}(\overline{\Omega})} + \|\varphi\|_{
 C^{2, \alpha}(\overline{\Omega})})
\end{equation*}
where $C$ depends on $n,\lambda,\Lambda,\overline{\Omega}$ and $\alpha$. 
\end{thm}
We next state a basic result in the Calderon-Zygmund theory concerning global $W^{2,p}$ estimates. 
\begin{thm}(\cite[Theorems 9.13 and 9.15]{GT})
 Let $\Omega$ be a $C^{1,1}$ domain in $\R^n$, and let us consider the operator
 $$Lu=\sum_{i, j=1}^n a^{ij}(x) u_{ij}(x)$$
 where the coefficient matrix $(a^{ij})$ satisfies, for positive constants $\lambda,\Lambda$,
 $$a^{ij}\in C^0(\overline{\Omega}), \lambda I_n\leq (a^{ij})\leq \Lambda I_n.$$
 Then if $f\in L^p(\Omega)$ and $\varphi\in W^{2,p}(\Omega)$, with $1<p<\infty$, the Dirichlet problem
 \begin{equation*}
 \left\{
 \begin{alignedat}{2}
   Lu ~&=f \h~&&\text{in} ~\Omega, \\\
 u&= \varphi \h~&&\text{on}~ \p\Omega,
 \end{alignedat}
 \right.
\end{equation*}
has a unique solution $u\in W^{2,p}(\Omega)$ with the estimate
\begin{equation*}
 \|u\|_{W^{2,p}(\Omega)} \leq C(\|u\|_{L^p(\Omega)} + \|f\|_{L^p(\Omega)} + \|\varphi\|_{W^{2, p}(\Omega)})
\end{equation*}
where $C$ depends on $n, p, \lambda,\Lambda,\overline{\Omega}$ and the moduli of continuity of the coefficients $a^{ij}$ in $\overline{\Omega}$. 
\label{CZthm}
\end{thm}

In dimensions $n\geq 3$, the continuity of the coefficient matrix in Theorem \ref{CZthm} is essential. Indeed, it is shown in \cite{U} and \cite{PT} that 
if the continuity is dropped in the above theorem,  then the $W^{2,p}$ estimate is false for $p\geq 1$.

Motivated by Theorems \ref{Schauthm} and \ref{CZthm}, we are led naturally to the following:
\begin{prob} Suppose the boundary data $\varphi$ and $\psi$ are smooth. Investigate the solvability of 
$C^{4,\alpha}(\overline{\Omega})$ solutions to (\ref{AMCE})-(\ref{SBV}) when $f$ is H\"older continuous and $W^{4,p}(\Omega)$ solutions when $f$ is less regular. 
\end{prob}
Note that the case of dimension $n=1$ is very easy to deal with and is by now completely settled (see also \cite{CW}). Thus we assume throughout
that $n\geq 2$. Let us recall previous results on this problem in chronological order.

Regarding $C^{4,\alpha}(\overline{\Omega})$ solutions: Trudinger-Wang \cite{TW08} solved this problem 
when $f\in C^{\alpha}(\overline\Omega)$, $f\leq 0$, $G(d)=\frac{d^{\theta}-1}{\theta}$ and $\theta\in (0, 1/n)$ and very recently, Chau-Weinkove \cite{CW} completely 
removed the sign condition 
on $f$ in this case.

Regarding $W^{4,p}(\Omega)$ solutions: For the case $G(d)=\frac{d^{\theta}-1}{\theta}$ and $\theta\in (0, 1/n)$,
the previous works of Trudinger-Wang \cite{TW08}, Chau-Weinkove \cite{CW}
and the author \cite{Le} solved this global problem in $W^{4,p}$ under some restrictions on the sign or integrability of the affine mean curvature.

In a recent paper \cite{Le16}, we remove these restrictions and obtain global $W^{4,p}$ solution and $W^{4, p}$ estimates
to the second boundary value problem
when 
the affine mean curvature belongs to $L^p$ with $p$ greater than the dimension $n\geq 2$.  Our analysis also covers the case of Abreu's equation. 
\subsection{Solvability of the second boundary value problem}
 From now on, we assume that $G: (0,\infty) \rightarrow \mathbb{R}$ is a  smooth strictly concave function on $(0,\infty)$ whose derivative $w(d)=G'(d)$ is strictly 
positive. We introduce the following set of conditions:
\begin{myindentpar}{2cm}
 (A1) $\displaystyle{w' + (1-\frac{1}{n}) \frac{w}{d} \le 0}$. \\
(A2) $G(d) - dG'(d)\rightarrow \infty$ when $d\rightarrow\infty$.\\
(A3) $\displaystyle{d^{1-1/n} w \rightarrow \infty}$ as $d \rightarrow 0$.
\end{myindentpar}

Our main result, Theorem \ref{mainthm}, asserts the solvability of (\ref{AMCE})-(\ref{SBV}) in $W^{4,p}(\Omega)$ when $f\in L^p(\Omega)$ with $p>n$ and when (A1)--(A3) are satisfied.
\begin{thm}(\cite[Theorem 1.1]{Le16}) \label{mainthm}
 Assume that (A1)--(A3) are satisfied.  
 \begin{myindentpar}{1cm}
 (i) Fix $p>n$. Let $\Omega$ be a bounded, uniformly convex domain in $\R^n$ with $\partial \Omega \in C^{3, 1}$. 
Suppose $f \in L^{p}(\Omega)$, $\varphi \in W^{4,p}(\Omega)$ and $\psi \in W^{2,p}(\Omega)$ with $\inf_{\Omega}\psi>0$.    
Then there exists  a unique uniformly convex solution $u \in W^{4, p}(\Omega)$ to the second boundary value problem (\ref{AMCE})-(\ref{SBV}).\\
(ii) Let $\Omega$ be a bounded, uniformly convex 
domain in $\R^{n}$ with $\partial \Omega \in C^{4, \alpha}$ for some $\alpha \in (0,1)$. Suppose $f \in C^{\alpha}(\overline{\Omega})$, $\varphi \in C^{4, \alpha}(\overline{\Omega})$, $\psi \in C^{2, \alpha}(\overline{\Omega})$ and 
$\inf_{\Omega}\psi>0$. Then there exists  a unique uniformly convex solution $u \in C^{4, \alpha}(\overline\Omega)$ to the second boundary value problem (\ref{AMCE})-(\ref{SBV}).
\end{myindentpar}
\end{thm}

It is quite remarkable that the integrability condition of the generalized affine mean curvature $L[u]$ in Theorem \ref{mainthm} does not depend on the concave function $G$. In the special case of
$G(d)=\frac{d^{\theta}-1}{\theta}$ with $\theta=\frac{1}{n+2}$, Theorem \ref{mainthm} tells us that we can prescribe the $L^p$ 
affine mean curvature (for any finite $p>n$) of the graph of a uniformly convex function with smooth Dirichlet boundary conditions on the function and its Hessian determinant.

\begin{rem} \label{ABrem} 
 Functions $G$ satisfying (A1)--(A3) include $$G(d) = \frac{d^{\theta}-1}{\theta}~ \text{where}~ 0\leq \theta<1/n,~\text{and }~G(d) =\frac{\log d}{\log \log (d+ e^{e^{4n}})}.$$
 Thus, Theorem \ref{mainthm} also covers the case of functions $G$ below 
the threshold of Abreu's equation where $G(d)=\log d$.
\end{rem}
\newpage
 \section{Existence of solution to the second boundary value problem }
 \subsection{Existence of solution via degree theory and a priori estimates}
By using the Leray-Schauder degree theory as in Trudinger-Wang \cite{TW05}, Theorem \ref{mainthm} follows from the following global a priori $W^{4,p}$ and $C^{2,\alpha}$ estimates for solutions 
of (\ref{AMCE})-(\ref{SBV}).
\begin{thm} 
\label{keythm} (\cite[Theorem 1.2]{Le16}) Assume that (A1)--(A3) are satisfied.
\begin{myindentpar}{1cm}
(i) Let $p> n$ and let $\Omega$ be a bounded, uniformly convex 
domain in $\R^{n}$ with $\p\Omega\in C^{3,1}$. Suppose $\varphi \in W^{4,p}(\Omega), \psi\in W^{2,p}(\Omega)$, $\inf_{\Omega}\psi>0$ and $f\in L^{p}(\Omega)$. Then, 
for any uniformly convex solution $u\in C^{4}(\overline{\Omega})$ of (\ref{AMCE})-(\ref{SBV}), we have the estimates
\begin{equation}
\label{global-est}
\norm{u}_{W^{4,p}(\Omega)}\leq C,~\text{and}~\det D^2 u\geq C^{-1}
\end{equation}
where $C$ depends 
on $n, p, G, \Omega$, $\norm{f}_{L^{p}(\Omega)}$, $\norm{\varphi}_{W^{4,p}(\Omega)}, \norm{\psi}_{W^{2,p}(\Omega)}$, and $\inf_{\Omega} \psi.$\\
(ii) Let $\Omega$ be a bounded, uniformly convex 
domain in $\R^{n}$ with $\partial \Omega \in C^{4, \alpha}$ for some $\alpha \in (0,1)$. Suppose $f \in C^{\alpha}(\overline{\Omega})$, $\varphi \in C^{4, \alpha}(\overline{\Omega})$, $\psi \in C^{2, \alpha}(\overline{\Omega})$ and 
$\inf_{\Omega}\psi>0$. Then, 
for any uniformly convex solution $u\in C^{4}(\overline{\Omega})$ of (\ref{AMCE})-(\ref{SBV}), we have the estimates
\begin{equation}\label{global-est2}\| u \|_{C^{4, \alpha}(\overline{\Omega})} \le C,
~\text{and}~\det D^2 u\geq C^{-1}\end{equation}
where $C$ depends 
on $n, \alpha, G, \Omega$, $\| f\|_{C^{\alpha}(\ov{\Omega})}$, $\| \varphi\|_{C^{4, \alpha}(\overline{\Omega})}$, $\| \psi\|_{C^{2, \alpha}(\ov{\Omega})}$ and $\inf_{\partial \Omega}\psi$.
\end{myindentpar}
\end{thm}

We briefly comment on the roles of conditions (A1)--(A3) in Theorem \ref{keythm}. (A1) guarantees the concavity of the 
functional $J$ defined in (\ref{functional}); (A3) gives the upper bound for $w$ while (A2) gives the upper bound for the dual $w^{*}$ of $w$ via the Legendre transform. As will be seen 
later in Lemma \ref{Legeq}, $w^{*}= G(\det D^2 u)- (\det D^2 u) G'(\det D^2 u).$ 

In \cite{Le16}, we noted that
the global $W^{4,p}$ estimate (\ref{global-est}) fails when
$f$ has integrability less than the dimension $n$ or the coercivity condition (A2) is not satisfied. 
However, it turns out that solutions 
to the second boundary value problem (\ref{AMCE})-(\ref{SBV}) 
are well-behaved near the boundary even if the coercivity condition (A2) fails. This is an interesting phenomenon: things can only go wrong in the interior in the boundary value problem
for the fourth order problem  (\ref{AMCE})-(\ref{SBV}).

Assuming Theorem \ref{keythm}, we can now complete the proof of Theorem \ref{mainthm} using the Leray-Schauder degree theory argument of Trudinger-Wang \cite{TW05, TW08}.  
Here we follow the presentation in Chau and Weinkove \cite{CW}.
\begin{proof}[Proof of Theorem \ref{mainthm}]
Let $\Omega, \varphi, \psi, f, p$ be as in the first part of Theorem \ref{mainthm}. Since $p>n$, and $\varphi\in W^{4, p}(\Omega)$, we have $\varphi\in C^3(\overline{\Omega})$
by the Sobolev embedding theorem.

Fix $\alpha \in (0,1)$. For  a large constant $R>1$ to be determined, define a bounded set $D(R)$ in $C^{\alpha}(\ov{\Omega})$ as follows:
$$D(R) = \{ v \in C^{\alpha}(\ov{\Omega}) \ | \ v \ge R^{-1}, \ \| v\|_{C^{\alpha}(\ov{\Omega})} \le R\}.$$
Next, let $\Theta: (0,\infty) \rightarrow (0,\infty)$ be the inverse function of $G': (0,\infty) \rightarrow (0,\infty).$  

For $t \in [0,1]$, we will define an operator $\Phi_t : D(R) \rightarrow C^{\alpha}(\ov{\Omega})$ as follows.
Given $w \in D(R)$, define $u \in C^{2, \alpha}(\ov{\Omega})$ to be the unique strictly convex solution to
\begin{equation}\label{LS1}
 \left\{
 \begin{alignedat}{2}
   \det D^2 u ~&=\Theta(w) \h~&&\text{in} ~\Omega, \\\
 u&= \varphi \h~&&\text{on}~ \p\Omega.
 \end{alignedat}
 \right.
\end{equation}
The existence of $u$ follows from the boundary regularity result of the Monge-Amp\`ere equation established by Trudinger and Wang \cite{TW08}; see Theorem \ref{TWC2} below.
Next, let $w_t \in W^{2,p}(\Omega)$  be the unique solution to the equation
\begin{equation}\label{LS2}
 \left\{
 \begin{alignedat}{2}
   U^{ij} (w_t)_{ij} ~&=tf \h~&&\text{in} ~\Omega, \\\
 w_t&= t\psi + (1-t) \h~&&\text{on}~ \p\Omega.
 \end{alignedat}
 \right.
\end{equation}
Because $p>n$, $w_t$ lies in  $C^{ \alpha}(\ov{\Omega})$.  We define $\Phi_t$ to be the map sending $w$ to $w_t$.

We note that:
\begin{enumerate}
\item[(i)] $\Phi_0(D(R)) = \{ 1\} $, and in particular, $\Phi_0$ has a unique fixed point.
\item[(ii)] The map $[0,1] \times D(R) \rightarrow C^{\alpha}(\ov{\Omega})$ given by $(t,w) \mapsto \Phi_t(w)$ is continuous.  
\item[(iii)] $\Phi_t$ is compact for each $t \in [0,1]$.
\item[(iv)] For every $t\in [0,1]$, if $w \in D(R)$ is a fixed point of $\Phi_t$ then $w \notin \partial D(R)$.
\end{enumerate}
Indeed, part (iii) follows from the standard \emph{a priori} estimates for the two separate equations (\ref{LS1}) and (\ref{LS2}).  For part (iv), let $w>0$ be a fixed point of $\Phi_t$.  
Then $w \in W^{2,p}(\Omega)$ and hence $u \in W^{4,p}(\Omega)$.  Next we apply Theorem \ref{keythm} to obtain $w>R^{-1}$ and $\| w \|_{C^{\alpha}(\ov{\Omega})} < R$ for some $R$ sufficiently large and depending only on the initial data.

Then the Leray-Schauder degree of $\Phi_t$ is well-defined for each $t$ and is constant on $[0,1]$ (see \cite[Theorem 2.2.4]{OCC}, for example).  $\Phi_0$ has a fixed point and hence $\Phi_1$ 
must also have a fixed point $w$, giving rise to a solution $u\in W^{4,p}(\Omega)$ of  the second boundary value problem (\ref{AMCE})-(\ref{SBV}).

In the second case of 
Theorem \ref{mainthm}, by similar arguments, $u$ will lie in $C^{4, \alpha}(\ov{\Omega})$.
Note that the solution is uniformly convex since $\det D^2u \ge C^{-1}>0$. 
\end{proof}
\subsection{Several boundary regularity results for Monge-Amp\`ere and linearized Monge-Amp\`ere equations }
In the above proof of Theorem \ref{mainthm} and that of Theorem \ref{keythm}, we use the following
global $C^{2,\alpha}$ estimates for the Monge-Amp\`ere equation established by Trudinger and Wang \cite{TW08} 
when the Monge-Amp\`ere measure is only assumed to be globally $C^{\alpha}$.
\begin{thm} ( \cite[Theorem 1.1]{TW08}) 
\label{TWC2}
Let $\Omega$ be a uniformly convex domain in
$\R^n$, with boundary $\p\Omega\in C^3$, $\phi\in C^3(\overline{\Omega})$ and $f\in
C^\alpha(\overline{\Omega})$, for some $\alpha\in (0, 1)$, satisfying $\inf
f>0$. Then the Dirichlet problem 
\begin{equation*}
 \left\{
 \begin{alignedat}{2}
   \det D^2 u ~&=f \h~&&\text{in} ~\Omega, \\\
 u&= \phi \h~&&\text{on}~ \p\Omega,
 \end{alignedat}
 \right.
\end{equation*}
has a unique strictly convex solution $u\in C^{2,\alpha}(\overline{\Omega})$. This solution satisfies the estimate
$$\|u\|_{C^{2, \alpha}(\overline{\Omega})}\le C$$
where $C$ is a constant depending on $n, \alpha$, $\inf f$,
$\|f\|_{C^\alpha(\overline{\Omega})}$, $\partial\Omega$ and $\phi$.
\end{thm}
In the proof of Theorem \ref{keythm}, we will use two sets of H\"older estimates. The first is the following global H\"older estimates for the linearized Monge-Amp\`ere equation.
\begin{thm} (\cite[Theorem 1.4]{Le})
\label{global-h}
Let $\Omega$ be a bounded, uniformly convex domain in $\R^{n}$ with $\p\Omega\in C^{3}$. Let $u: \overline{\Omega}\rightarrow \R$, $u\in C^{0,1}(\overline{\Omega})\cap C^{2}(\Omega)$  be a convex function satisfying
$$0<\lambda\leq \det D^{2}u\leq \Lambda<\infty,~
\text{and}~u\mid_{\p\Omega}\in C^{3}.$$
Denote by $U=(U^{ij})$ the cofactor matrix of $D^2 u$.
 Let $v\in C(\overline{\Omega})\cap W^{2, n}_{loc}(\Omega)$ be the solution to the linearized Monge-Amp\`ere equation
 \begin{equation*}
 \left\{
 \begin{alignedat}{2}
   U^{ij} v_{ij} ~&=g \h~&&\text{in} ~\Omega, \\\
 v&= \varphi \h~&&\text{on}~ \p\Omega,
 \end{alignedat}
 \right.
\end{equation*}
 where $\varphi\in C^{\alpha}(\p\Omega)$ for some $\alpha\in (0, 1)$ and $g\in L^n(\Omega)$. Then, $v\in C^{\beta}(\overline{\Omega})$ 
with the estimate
$$\|v\|_{C^{\beta}(\overline{\Omega})}\leq C\left(\|\varphi\|_{C^{\alpha}(\p\Omega)} + \|g\|_{L^{n}(\Omega)}\right)$$
where $\beta$ depends only on $\lambda, \Lambda, n, \alpha$, and $C$ depends only on $\lambda, \Lambda, n, \alpha$, $diam (\Omega)$, $\|u\|_{C^3(\p\Omega)}$, $\|\p\Omega\|_{C^3}$ and the uniform convexity of $\Omega.$ 
\end{thm}
The second set of H\"older estimates  is concerned with 
boundary H\"older continuity for solutions to non-uniformly elliptic, linear equations without lower order terms where we have 
lower bound on the determinant of the coefficient matrix. 
\begin{prop}(\cite[Proposition 2.1]{Le})\label{global-holder}
Assume that $\Omega$ is a uniformly convex domain in $\R^n$. 
Let $v\in C(\overline{\Omega})\cap W^{2, n}_{loc}(\Omega)$ be the 
solution to the equation 
\begin{equation*}
 \left\{
 \begin{alignedat}{2}
   a^{ij} v_{ij} ~&=g \h~&&\text{in} ~\Omega, \\\
 v&= \varphi \h~&&\text{on}~ \p\Omega.
 \end{alignedat}
 \right.
\end{equation*}
Here,
$\varphi\in C^{\alpha}(\p\Omega)$ for some $\alpha\in (0, 1)$, $g\in L^n(\Omega)$, and
the matrix $(a^{ij})$ is assumed to be measurable, positive definite and satisfies $\det (a_{ij})\geq \lambda.$ 
Then, there exist $\delta, C$ depending only on $\lambda, n, \alpha$, $diam (\Omega)$, and the uniform convexity of $\Omega$ so that, for any $x_{0}\in\partial\Omega$, we have
$$|v(x)-v(x_{0})|\leq C|x-x_{0}|^{\frac{\alpha}{\alpha +2}}\left(\|\varphi\|_{C^{\alpha}(\p\Omega)} + \|g\|_{L^{n}(\Omega)}\right)~\text{for all}~ x\in \Omega\cap B_{\delta}(x_{0}). $$
\end{prop}
\newpage
\section{Proof of global $W^{4,p}$ and $C^{4,\alpha}$ estimates}
In this section, we give the proof of Theorem \ref{keythm}, following the presentation in \cite{Le16}. We focus on the 
 global $W^{4,p}$ estimates.  The global $C^{4,\alpha}$ estimates then easily follow.
 
Let $p> n$ and let $\Omega$ be a bounded, uniformly convex 
domain in $\R^{n}$ with $\p\Omega\in C^{3,1}$. Assume $\varphi \in W^{4,p}(\Omega), \psi\in W^{2,p}(\Omega)$, $\inf_{\Omega}\psi>0$ and $f\in L^{p}(\Omega)$. 
Suppose a uniformly convex solution $u\in C^{4}(\overline{\Omega})$ solves (\ref{AMCE})-(\ref{SBV}).

We denote by $C, C', C_1, C_2, c, c_1$, etc,  {\it universal constants} that may change from line to line. Unless stated otherwise, they depend only 
on $n, p, G, \Omega$, $\norm{f}_{L^{p}(\Omega)}$, $\norm{\varphi}_{W^{4,p}(\Omega)}, \norm{\psi}_{W^{2,p}(\Omega)}$, and $\inf_{\Omega} \psi.$
We briefly explain the structure of the proof. 

The key step to global $W^{4,p}$ estimates for (\ref{AMCE})-(\ref{SBV}) consists in showing that the Hessian determinant $\det D^2 u$ is bounded between two positive universal constants $C_1 $ and $C_2$. Once this is done, the proof can be easily completed by using
two global regularity results in Theorems \ref{TWC2} and \ref{global-h}.

The proof of a uniform lower bound for $\det D^2 u$ is quite easy. It is just an application of the Aleksandrov-Bakelman-Pucci (ABP) maximum principle. The most difficult part 
of the proof is to get a uniform upper bound on $\det D^2u$.

Our key insight to prove a uniform upper bound for $\det D^2 u$ is to apply the ABP estimate to the  dual equation of (\ref{AMCE}) via the Legendre transform.
For this, we need the coercivity condition (A2) and a global gradient bound for $u$; see Lemma \ref{boundd}.

The proof of a global gradient bound for $u$ is more involved. 
First, we prove the global a priori bound on $u$ in Lemma \ref{udet}, assuming only (A1), by testing against smooth 
concave functions $\hat u$ and convex functions $\tilde u$ having generalized affine 
mean curvature $L[\tilde u]$  bounded in $L^1$ (Lemma \ref{geoc}). By (A3), we 
have a uniform lower bound for 
$\det D^2 u$ (Lemma \ref{upwlem}). 
Next, by using boundary H\"older estimates for second-order equations with lower bound on the determinant of the coefficient matrix in Proposition \ref{global-holder} to $U^{ij}w_{ij}=f$, 
we obtain a uniform bound for $\det D^2 u$ near the boundary. This, together the global bound on $u$, 
allows us to construct barriers using the strict convexity of $\p\Omega$ to obtain the global gradient bound for $u$;
see Lemma \ref{Dubound}.

\subsection{Test functions}
Our basic geometric construction is the following:
\begin{lem}\label{geoc} There exist a convex function $\tilde u\in W^{4,p}(\Omega)$ and a concave function $\hat u\in W^{4,p}(\Omega)$
 with the following properties:
 \begin{myindentpar}{1cm}
  (i) $\tilde u=\hat u=\varphi$ on $\p\Omega$,\\
  (ii) $$
\| \tilde{u} \|_{C^{3}(\ov{\Omega})} + \| \hat{u} \|_{C^{3}(\ov{\Omega})} +
\|\tilde u\|_{W^{4, p}(\Omega)} + \|\hat u\|_{W^{4, p}(\Omega)}\leq C,\quad \textrm{and } \det D^2\tilde{u} \ge C^{-1}>0,
$$
(iii) letting $\tilde{w}=G'(\det D^2 \tilde{u})$, and denoting by $(\tilde{U}^{ij})$ the 
cofactor matrix of $(\tilde{u}_{ij})$, then  the generalized affine mean curvature of the graph of $\tilde u$ is uniformly bounded in $L^p$, that is $$\left\|\tilde U^{ij}\tilde w_{ij}\right\|_{L^p(\Omega)}\leq C,$$
 \end{myindentpar}
where $C$ depends only on $n$, $p$, $\Omega$, 
$G$, and $\| \varphi\|_{W^{4, p}(\Omega)}$.
 
\end{lem}

\begin{proof} 
Let $\rho$ be a strictly convex defining function of $\Omega$, that is 
$\Omega:=\{x\in \R^n: \rho(x)<0\}$, $\rho=0$ on $\p\Omega$ and $D\rho\neq 0$ on $\p\Omega$.
Then $$D^2 \rho\geq \eta I_n~ \text{and} ~\rho\geq -\eta^{-1} ~\text{in}~ \Omega$$ for some $\eta>0$ depending only on $\Omega$. 
Consider the following functions
$$\tilde u(x) =\varphi(x) + \mu (e^{\rho}-1), \hat{u}= \varphi(x) - \mu (e^{\rho}-1).$$
Then $\tilde u, \hat u \in W^{4,p}(\Omega)$. 
From
$$D^2 (e^{\rho}-1)=e^{\rho}(D^2\rho + D\rho \otimes D\rho)\geq e^{-\eta^{-1}}\eta I_n,$$
we find that for a fixed but sufficiently large $\mu$ (depending only on $n, p, \Omega$ and $\|\varphi\|_{W^{4, p}(\Omega)}$), $\tilde u$ is convex while $\hat u$ is concave and; moreover, recalling $p>n$, (i) and (ii) are satisfied. 
From (ii), the 
smoothness of $G$ and 
$$\tilde w_{ij}= G'''(\det D^2\tilde u) \tilde U^{kl}\tilde U^{rs} \tilde u_{kli}\tilde u_{rsj} + G''(\det D^2\tilde u) \tilde U^{kl}\tilde u_{klij}
+ G''(\det D^2\tilde u) \tilde U^{kl}_j \tilde u_{kli},$$
we easily obtain (iii).
\end{proof}
\subsection{$L^1$ bound and lower bound on the Hessian determinant }
The following lemma gives a uniform $L^1$ bound on $\det D^2 u$ and as a consequence, a uniform bound on $u$. 
\begin{lem}\label{udet} Assuming (A1), we have 

$$(i)~~\int_{\Omega} \det D^2 u \le C,~\text{and}~
(ii)~ \sup_{\Omega} |u| \le C,$$
where $C$ depends only on $n$, $p$, $\Omega$, 
$G$, $\| f\|_{L^1(\Omega)}$, $\| \varphi\|_{W^{4, p}(\Omega)}$, $\| \psi\|_{L^{\infty}(\Omega)}$ and $\inf_{\partial \Omega}\psi$.

\end{lem}

\begin{proof}[Proof of Lemma \ref{udet}]  
Let $\tilde u$ be as in Lemma \ref{geoc}. Set $\tilde f= \tilde U^{ij}\tilde w_{ij}$.
The assumption (A1) implies that the function $\tilde G(d):= G(d^n)$ is concave because
$$\tilde G^{''}(d)= n^2 d^{n-2} \left[w'(d^n) d^n + (1-\frac{1}{n})w(d^n)\right]\leq 0.$$
Using this, $G'>0$, and the concavity of the map $M\longmapsto (\det M)^{1/n}$ in the space of symmetric matrices $M\geq 0$ (see Lemma \ref{concavelem}), we obtain
\begin{eqnarray*}
 \tilde G((\det D^2 \tilde u)^{1/n}) -\tilde G((\det D^2  u)^{1/n})&\leq& \tilde G^{'}((\det D^2 u)^{1/n})((\det D^2 \tilde u)^{1/n}-(\det D^2 u)^{1/n})\\
 &\leq & \tilde G^{'}((\det D^2 u)^{1/n})\frac{1}{n} (\det D^2 u)^{1/n-1} U^{ij} (\tilde u-u)_{ij} .
\end{eqnarray*}
Since $\tilde G^{'}((\det D^2 u)^{1/n}) = n G^{'}(\det D^2 u) (\det D^2 u)^{\frac{n-1}{n}}$, we rewrite the above inequalities as
$$G(\det D^2 \tilde u)- G(\det D^2 u)\leq wU^{ij}(\tilde u- u)_{ij}.$$
Similarly,
$$G(\det D^2 u)- G(\det D^2 \tilde u)\leq \tilde w \tilde U^{ij}(u- \tilde u)_{ij}.$$
Adding, integrating by parts twice and using the fact that $(U^{ij})$ is divergence free, we obtain
\begin{eqnarray*}
 0&\leq& \int_{\Omega} wU^{ij}(\tilde u- u)_{ij} + \tilde w \tilde U^{ij}(u- \tilde u)_{ij}\\
 &=& \int_{\partial \Omega} \psi U^{ij} (\tilde{u}_j - u_j) \nu_i + \int_{\Omega} f(\tilde{u}-u) + 
 \int_{\partial \Omega} \tilde w \tilde{U}^{ij} (u_j - \tilde{u}_j) \nu_i + \int_{\Omega} \tilde f (u-\tilde{u}) . 
\end{eqnarray*}
Here $\nu= (\nu_{1},\cdots,\nu_n)$ is the unit outer normal vector field on $\p\Omega$. It follows that
\begin{eqnarray} \label{key1}
\int_{\Omega} (f- \tilde f)u + \int_{\partial \Omega} \left(\psi U^{ij} (u_j -\tilde{u}_j ) \nu_i  +
 \tilde w \tilde{U}^{ij} (\tilde{u}_j - u_j) \nu_i\right) &\le& \int_{\Omega} (f-\tilde f) \tilde u\nonumber\\&\leq& 
(\|f\|_{L^1(\Omega)} + \|\tilde f\|_{L^1(\Omega)})\|\tilde u\|_{L^{\infty}(\Omega)}\leq 
C.
\end{eqnarray}
Let us analyze the boundary terms in (\ref{key1}). Since $u-\tilde u=0$ on $\p\Omega$, we have $(u-\tilde u)_j= (u-\tilde u)_{\nu} \nu_j$, and hence
$$U^{ij}(u-\tilde u)_j \nu_i= U^{ij}\nu_j \nu_i (u-\tilde u)_{\nu} = U^{\nu\nu}(u-\tilde u)_{\nu}\equiv (\det D^2_{x'} u)(u-\tilde u)_{\nu},$$
with $x'\perp \nu$ denoting the tangential directions along $\p\Omega$. Therefore, 

\begin{equation} \label{nunu}
U^{ij} (u_j - \tilde{u}_j) \nu_i = U^{\nu\nu} (u_{\nu} - \tilde{u}_{\nu}), \quad \tilde{U}^{ij} (\tilde{u}_j - u_j) \nu_i = \tilde{U}^{\nu\nu} (\tilde{u}_{\nu} - u_{\nu}).
\end{equation}

On the other hand, from $u-\varphi=0$ on $\p\Omega$, we have, with respect to a principal coordinate system at any point $y\in\p\Omega$ (see, e.g., 
\cite[formula (14.95) in \S 14.6]{GT})
$$D_{ij}(u-\varphi)= (u-\varphi)_{\nu}\kappa_i\delta_{ij}, i, j=1, \cdots, n-1,$$
where $\kappa_1,\cdots,\kappa_{n-1}$ denote the principal curvatures of $\p\Omega$ at $y$. Let $K=\kappa_1\cdots\kappa_{n-1}$ be the Gauss curvature of $\partial \Omega$ at
$y\in\p\Omega$. Then, at any $y\in\Omega$, by noting that 
\begin{equation}\label{Gauss1}U^{\nu\nu}=\det D^2_{x'} u =\det (D_{ij}u)_{1\leq i, j\leq n-1}
\end{equation} and taking the determinants of
\begin{equation}\label{Gauss2}D_{ij} u = u_{\nu}\kappa_i\delta_{ij}-\varphi_{\nu}\kappa_i\delta_{ij} + D_{ij}\varphi,
\end{equation}
we obtain, with $u_{\nu}^+ =\max (0, u_{\nu})$,
\begin{equation} \label{Gaussc}
U^{\nu\nu}  = K (u_{\nu})^{n-1} + E, \quad \textrm{where } |E| \le C (1+ |u_{\nu}|^{n-2}) \leq C(1 + (u_{\nu}^+)^{n-2}).
\end{equation}
In the last inequality of (\ref{Gaussc}), we used the following fact which is due to the convexity of $u$:
\begin{equation}
 \label{lowerunu} u_{\nu}\geq -\|D\varphi\|_{L^{\infty}(\Omega)}.
\end{equation}

Now, let $\hat u$ be as in Lemma \ref{geoc}. Integrating by parts twice, and using (\ref{nunu}), we find that
\begin{equation}\label{hateq}\int_{\Omega} U^{ij} (u-\hat u)_{ij}= \int_{\p\Omega} U^{ij} (u-\hat u)_{i}\nu_j= \int_{\p\Omega} U^{\nu\nu} (u_{\nu}-\hat u_{\nu}).
\end{equation}
By Lemma \ref{geoc}, $\hat u_{\nu}$ is bounded by a universal constant. 
The concavity of $\hat u$ gives $U^{ij} \hat u_{ij}\leq 0$. Thus,
using $U^{ij} u_{ij}= n\det D^2 u$, we obtain from (\ref{Gaussc})-(\ref{hateq}) the following estimates
\begin{equation}\int_{\Omega} \det D^2 u \leq  \int_{\p\Omega} U^{\nu\nu} (u_{\nu}-\hat u_{\nu}) \leq
C + C
\int_{\p\Omega}  (u_{\nu}^+)^n .
 \label{l1det}
\end{equation}
The Aleksandrov's maximum principle (see Lemma \ref{ABPmax2}) then gives

\begin{equation}\label{umax}\|u\|_{L^{\infty}(\Omega)} \leq \|\varphi\|_{L^{\infty}(\p\Omega)} + C(n) \text{diam}(\Omega)\left(\int_{\Omega} \det D^2 u\right)^{1/n} \leq C + C\left(
\int_{\p\Omega}  (u_{\nu}^+)^n \right)^{1/n}.
\end{equation}

By Lemma \ref{geoc}, $\tilde{u}_{\nu}, \tilde w$, $\tilde{U}^{\nu\nu}$ and $\|\tilde f\|_{L^1(\Omega)}$ are uniformly bounded. Taking (\ref{key1})-(\ref{Gaussc}) 
and (\ref{umax}) into account, we obtain
\begin{eqnarray*} 
 \int_{\partial \Omega} K \psi (u_{\nu}^+)^n  &\le& C + C \int_{\partial \Omega} (u_{\nu}^+)^{n-1} - \int_{\Omega} (f - \tilde f)u\\
 &\leq& C + C \int_{\partial \Omega} (u_{\nu}^+)^{n-1} + C\left(
\int_{\p\Omega}  (u_{\nu}^+)^n \right)^{1/n}.
\end{eqnarray*} 
From H\"older inequality, $n\geq 2$ and the fact that $K\psi$ has a positive lower bound, we easily obtain
$$\int_{\partial \Omega} (u_{\nu}^+)^n \le C,$$
from which the claimed uniform bound for $u$ in (ii) follows by (\ref{umax}).
Recalling (\ref{l1det}), we obtain 
the desired bound for the $L^1$ norm of $\det D^2 u$ stated in (i).
\end{proof}

We prove the uniqueness part of Theorem \ref{mainthm} in the following lemma.
\begin{lem}
The problem (\ref{AMCE})-(\ref{SBV}) has at most one strictly convex solution $u\in W^{4,p}(\Omega)$.
\end{lem}
\begin{proof} Suppose $u$ and $\tilde u$ are two solutions. We use the same notation as in the proof of Lemma \ref{udet}. Then, using concavity of 
the functional $J$, we obtain as in (\ref{key1}) the estimate
\begin{eqnarray} 
\label{Gauss3}
 0\geq \int_{\partial \Omega} \psi (U^{ij}-\tilde{U}^{ij}) (u_j -\tilde{u}_j ) \nu_i=\int_{\partial\Omega} \psi (U^{\nu\nu}-\tilde U^{\nu\nu})(u_{\nu}-\tilde u_{\nu}).
\end{eqnarray}
It is clear from (\ref{Gauss1}) and (\ref{Gauss2}) that if $u_\nu>\tilde u_{\nu}$ then $U^{\nu\nu}>\tilde U^{\nu\nu}$. Therefore (\ref{Gauss3}) and (\ref{key1}) are now
actually equalities and 
$u_\nu=\tilde u_\nu$ on $\p\Omega$. Thus $Du=D\tilde u$ on $\p\Omega$.
Using the concavity of $J$ in the derivation of (\ref{key1}), we obtain $\det D^2 u=\det D^2 \tilde u$ in $\Omega$. Hence $u=\tilde u$ on $\overline{\Omega}$.
\end{proof}
The next lemma gives a uniform lower bound on the Hessian determinant $\det D^2 u$. 
\begin{lem} \label{upwlem} Assume (A3) is satisfied. 
Then, there exists a constant $C>0$ depending 
only on $n, p, G, \Omega$, $\norm{f}_{L^{n}(\Omega)}$ and $\norm{\psi}_{W^{2,p}(\Omega)}$ 
such that 
\begin{equation*}
w\leq C,~\text{and}~ \det D^2 u\geq C^{-1}.
\end{equation*}
\end{lem}
\begin{proof} Let $d:=\det D^2 u$. Then $\det U = d^{n-1}.$ 
We apply the ABP maximum principle in Theorem \ref{ABPmax} to $U^{ij}w_{ij}=f$ in $\Omega$ with
$w=\psi$ on $\p\Omega$ to find that
\begin{equation}\label{Aleksandrov}\sup_{\Omega} w \leq \sup_{\partial \Omega} \psi +C \left\|\frac{f}{d^{(n-1)/n}}\right\|_{L^n(\Omega)}
\leq \sup_{\partial \Omega} \psi +C \left\| f\right \|_{L^n(\Omega)} \sup_{\Omega} (d^{(1-n)/n}),\\
\end{equation}
where $C$ depends only on $n$ and $\Omega$.  The desired upper bound on $w$ follows from \eqref{Aleksandrov} and  assumption (A3) on $G$.  The lower bound for $\det D^2 u=d$ then follows immediately.
\end{proof}
\subsection{Gradient bound}
Now, we prove a key gradient bound for $u$.
\begin{lem}\label{Dubound}
 Assume (A1), and (A3) are satisfied. Then, there exists a constant $C>0$ depending 
only on $n, p, G, \Omega$, $\norm{f}_{L^{n}(\Omega)}$, $\|\varphi\|_{W^{4, p}(\Omega)}$, $\norm{\psi}_{W^{2,p}(\Omega)}$ and $\inf_{\p\Omega}\psi$
such that 
 $$\sup_{\Omega}|Du|\leq C.$$ 
\end{lem}
\begin{proof} [Proof of Lemma \ref{Dubound}]
Let $\nu$ be the unit outer normal vector field on $\p\Omega$.
The crucial point in the proof is to prove an upper bound for $u_{\nu}$. 

By Lemma \ref{upwlem}, we have
a lower bound for the Hessian determinant $\det D^2 u\geq C_1$ 
where $C_1$ depends
only on $n, p, G, \Omega$, $\norm{f}_{L^{n}(\Omega)}$ and $\norm{\psi}_{W^{2,p}(\Omega)}$. Because $p>n$, $\psi$ is clearly H\"older continuous in $\overline{\Omega}$. Since $\det U\geq C_1^{n-1}$, applying  Proposition \ref{global-holder} to $U^{ij}w_{ij}= f$ in $\Omega$, we find that
$w$ is H\"older continuous at the boundary. 

Note that (A1) implies  $(w(d)d^{1-1/n})' \le 0$ and therefore
$w(d)d^{1-1/n} \le C$ for $ d\geq 1.$
Since $w=\psi\geq \inf_{\p\Omega} \psi>0$ on 
$\p\Omega$, it follows from the boundary H\"older continuity of $w$ that $w$ is uniformly bounded from below while $\det D^2 u$  is uniformly bounded from above in a 
neighborhood $\Omega_{\delta}:= \{x\in\Omega: \text{dist}(x,\p\Omega)
\leq \delta\}$ of the boundary. Here $\delta$ is a universal constant, depending only on $n, p, G,\Omega$, $\inf_{\p\Omega} \psi$, $\|f\|_{L^n(\Omega)}$, and $\norm{\psi}_{W^{2,p}(\Omega)}$.

Let $\rho$ be a strictly convex defining function of $\Omega$, that is 
$\Omega:=\{x\in \R^n: \rho(x)<0\}$, $\rho=0$ on $\p\Omega$ and $D\rho\neq 0$ on $\p\Omega$. Then
$$D^2\rho\geq \eta I_n~\text{and}~\rho\geq -\eta^{-1}~\text{in}~\Omega$$
where $\eta>0$ depends only on $\Omega$. We easily find that,
for large $\mu$, the function
$$v(x) =\varphi(x) + \mu (e^{\rho}-1)$$
is a lower bound for $u$ in $\Omega_{\delta}$.

Indeed, there 
exists a universal $C_2>0$ such that $\rho\leq -C_2$ on $\p\Omega_{\delta}
\cap \Omega$.
Since $$ D^2 v = D^2\varphi + \mu e^{\rho}(D^2\rho + D\rho \otimes D\rho)\geq D^2\varphi + \mu\eta e^{-\eta^{-1}} I_n~\text{in}~ \Omega_{\delta},$$ $v=u$ on $\p\Omega$ while
$$v\leq \|\varphi\|_{L^{\infty}(\Omega)} + \mu (e^{-C_2}-1)~ \text{on}~ \p\Omega_{\delta}
\cap \Omega,$$ we find that for $\mu$ universally large,
$$\det D^2 v\geq \det D^2 u~\text{in} ~\Omega_{\delta}$$
and $u\geq v$ on $\p\Omega_{\delta}$ by the global bound on $u$ in Lemma \ref{udet}. Hence $u\geq v$ in $\Omega_{\delta}$ by the comparison principle (see 
Lemma \ref{comp-prin}). 

From $u=v$ on $\p\Omega$, we deduce that $u_{\nu}\leq v_{\nu}$ and this gives a uniform upper bound for $u_{\nu}$. By convexity, 
$$u_{\nu}(x) \ge  -\|D\varphi\|_{L^{\infty}(\Omega)} \quad \textrm{for all } x\in \partial \Omega. $$ Because $u=\varphi$ on $\p\Omega$, the tangential 
derivatives of $u$ on $\p\Omega$ are those of $\varphi$. 
Thus $Du$ is uniformly bounded on $\p\Omega$. Again, by convexity, we find that $Du$ is bounded in $\Omega$ by a universal constant as stated in the lemma.
\end{proof}

\subsection{Legendre transform and upper bound on Hessian determinant}

To prove a uniform upper bound for $\det D^2 u$, we use the Legendre transform:
$$y= Du(x), u^*(y) = x\cdot y - u(x) (=\sup_{z\in\Omega} \left(y\cdot z-u(z))\right).$$
The Legendre transform $u^{*}$ of $u$ is defined in $\Omega^*:=Du(\Omega)$. $u^*$ is a uniformly convex, $C^4$ smooth function in
$\Omega^*$. Furthermore the Legendre transform of $u^*$ is $u$
itself. From $y=Du(x)$ we have $x=D u^*(y)$ and
$D^2 u(x)= \left(D^2 u^*(y)\right)^{-1}. $

The Legendre transform $u^{*}$ satisfies a dual equation to (\ref{AMCE}) as stated in the following lemma.
\begin{lem}\label{Legeq}
The Legendre transform $u^*$ satisfies the equation
$${U^*}^{ij} w^*_{ij}=-f(Du^*)\,\det D^2 u^*,$$
where $({U^*}^{ij})$ is the cofactor matrix of $D^2 u^*$ and $$w^*=G\left((\det
D^2 u^*)^{-1}\right)- (\det
D^2 u^*)^{-1} G' \left((\det
D^2 u^*)^{-1}\right).$$
\end{lem}

This lemma was previously observed by Trudinger-Wang \cite{TW00} (in the proof of Lemma 3.2 there) and Zhou \cite{Zhou} (before the proof of Lemma 3.2 there). The idea is to observe that
$u^*$ is a critical point of the dual functional $J^*$ of $J$ under local perturbations and this gives the conclusion of Lemma \ref{Legeq}. We give here a direct proof of Lemma \ref{Legeq}.
\begin{proof}[Proof of Lemma \ref{Legeq}]
For simplicity, let
$d= \det D^2 u$ and $d^*= \det D^2 u^*.$ Then $d= {d^{*}}^{-1}$.
We denote by $(u^{ij})$ and $({u^*}^{ij})$ the inverses of the Hessian matrices $D^2 u= (u_{ij})= (\frac{\p^2 u}{\p x_i\p x_j})$ and $D^2 u^*=(u^*_{ij})=(\frac{\p^2 u^*}{\p y_i\p y_j})$.
Note that $w= G' (d)= G' ({d^{*}}^{-1}).$ Thus
$$w_{j} =\frac{\p w}{\p x_j} =\frac{\p G' ({d^*}^{-1})}{\p y_k} \frac{\p y_k}{\p x_j} = \left[\frac{\p }{\p y_k} G' ({d^*}^{-1})\right]u_{kj}=
\left[\frac{\p }{\p y_k} G' ({d^*}^{-1})\right]{u^*}^{kj}.$$
Clearly,

$${d^{*}}^{-1}\frac{\p }{\p y_k} G' ({d^*}^{-1}) = -\frac{\p}{\p y_k} \left[G ({d^*}^{-1})- {d^*}^{-1}G' ({d^*}^{-1})\right ]= - w^{*}_k,$$
from which it follows that
$w_j = - w^{*}_k (U^{*})^{kj}.$
Similarly,
$w_{ij}= \frac{\p }{\p y_l} w_j {u^*}^{li}.$
Hence, using
$$U^{ij} = (\det D^2 u) u^{ij}= (d^{*})^{-1} u^{*}_{ij},$$
and the fact that $U^{*}= ({U^*}^{ij})$ is divergence-free (see Lemma \ref{divfreeU}), we find from (\ref{AMCE}) that
\begin{eqnarray*}f(Du^{*})= U^{ij} w_{ij}=(d^{*})^{-1} u^{*}_{ij} {u^*}^{li} \frac{\p }{\p y_l} w_j  
= -(d^{*})^{-1}  \frac{\p }{\p y_j}\left\{ w^{*}_k
{U^*}^{kj}\right\} = -(d^{*})^{-1} {U^{*}}^{kj} w^{*}_{kj}.
\end{eqnarray*}
Thus, the lemma is proved.
\end{proof}

We are now ready to prove that the Hessian determinant $\det D^2 u$ is universally bounded away from $0$ and $\infty$.
\begin{lem} \label{boundd} Assume (A1)--(A3) are satisfied. Then, there exists a constant $C>0$ depending 
only on $n, p, G, \Omega$, $\norm{f}_{L^{n}(\Omega)}$, $\|\varphi\|_{W^{4, p}(\Omega)}$, $\norm{\psi}_{W^{2,p}(\Omega)}$ and $\inf_{\p\Omega}\psi$
such that 
 $$C^{-1}\leq \det D^2 u\leq C.$$
\end{lem}
\begin{proof}[Proof of Lemma \ref{boundd}] We use the same notation as in Lemma \ref{Legeq} and its proof. 
By Lemma \ref{Dubound}, $\text{diam}(\Omega^*)$ is bounded by a universal constant $C$. With (A1) and (A3), we can apply Lemma \ref{Legeq} to conclude that
${u^{*}}^{ij} w^{*}_{ij} = -  f(Du^{*}(y))~\text{in}~\Omega^*$
with $$w^*=G(d) - d G'(d) = G({G^{'}}^{-1}(w))- {G^{'}}^{-1}(w) w= G({G^{'}}^{-1}(\psi))- {G^{'}}^{-1}(\psi) \psi~ \text{on}~ \p\Omega^*.$$
Applying the ABP estimate, Theorem \ref{ABPmax}, to $w^{*}$ on $\Omega^{*}$, and then changing of variables $y= Du(x)$ with $dy = \det D^2 u~ dx,$ we obtain
\begin{eqnarray*}
 \|w^{\ast}\|_{L^{\infty}(\Omega^*)} &\leq& \|w^{\ast}\|_{L^{\infty}(\p\Omega^*)} + C_n \text{diam} (\Omega^*) \left\|\frac{f(Du^{*})}{(\det {u^{*}} ^{ij})^{1/n}}\right\|_{L^n(\Omega^*)}\\
 &=& C + C \left(\int_{\Omega^{*}} \frac{|f|^n(Du^{*})}{ (\det D^2 u^*)^{-1}}~ dy\right)^{1/n}\\&=& 
 C + C  \left(\int_{\Omega} \frac{|f|^n(x)}{ \det D^2 u} \det D^2 u~ dx\right)^{1/n}  = C + C \|f\|_{L^n(\Omega)}.
\end{eqnarray*}
Since $w^{*} = G(d)- d G'(d)$ and the coercivity condition (A2) is satisfied, the above estimates give a uniform upper bound for $d=\det D^2 u$. The lower bound for $\det D^2 u$ follows from Lemma \ref{upwlem}.
\end{proof}

With Lemma \ref{boundd}, we can now complete the proof of the global $W^{4,p}$ and $C^{4,\alpha}$ estimates in Theorem \ref{keythm}.

\begin{proof}[Proof of theorem \ref{keythm}]~~\\
(i) By Lemma \ref{boundd}, 
$C^{-1}\leq \det D^2 u \leq C.$
Note that, by (\ref{AMCE}), $w$ is the solution to the linearized Monge-Amp\`ere equation $U^{ij}w_{ij} =f$ with boundary data $w=\psi.$ Because $\psi\in W^{2,p}(\Omega)$ with $p>n$, $\psi$ is clearly H\"older continuous on $\partial\Omega$. 
Thus, by Theorem \ref{global-h}, $w \in C^{\alpha}(\overline{\Omega})$ for some $\alpha>0$ depending on the data of (\ref{AMCE})-(\ref{SBV}). Rewriting the equation for $w$ as
$$\det D^2 u = (G')^{-1}(w),$$
with the right hand side being in $C^{\alpha}(\overline{\Omega})$
and noticing $u=\varphi$ on $\p\Omega$ where $\varphi\in C^{3}(\overline{\Omega})$ because $\varphi\in W^{4,p}(\Omega)$ and $p>n$, we obtain
$u\in C^{2,\alpha}(\overline{\Omega})$ by Theorem \ref{TWC2}. Thus the first equation of (\ref{AMCE}) is a uniformly elliptic, second order partial differential equations in $w$
with $L^p(\Omega)$ right hand side. Hence $w\in W^{2,p}(\Omega)$ and in turn $u\in W^{4, p}(\Omega)$ with desired estimate
\begin{equation*}
\norm{u}_{W^{4,p}(\Omega)}\leq C,
\end{equation*}
where $C$ depends on $n, p, G, \Omega$, $\norm{f}_{L^{p}(\Omega)}, \norm{\varphi}_{W^{4,p}(\Omega)}, \norm{\psi}_{W^{2,p}(\Omega)}$, and $\inf_{\Omega} \psi.$\\
(ii) In this case, we also obtain as in (i) that $u\in C^{2,\alpha}(\overline{\Omega})$. The first equation of (\ref{AMCE}) is now a uniformly elliptic, second order partial differential equations in $w$
with $C^{\alpha}(\overline{\Omega})$ right hand side. Hence $w\in C^{2,\alpha}(\Omega)$ and in turn $u\in C^{4, \alpha}(\overline{\Omega})$ with the estimate
\begin{equation*}\| u \|_{C^{4, \alpha}(\overline{\Omega})} \le C,\end{equation*}
where $C$ depends 
on $n, \alpha, G, \Omega$, $\| f\|_{C^{\alpha}(\ov{\Omega})}$, $\| \varphi\|_{C^{4, \alpha}(\overline{\Omega})}$, $\| \psi\|_{C^{2, \alpha}(\ov{\Omega})}$ and $\inf_{\partial \Omega}\psi$.
\end{proof}

\part{The linearized Monge-Amp\`ere equation}
\section{The linearized Monge-Amp\`ere equation and interior regularity of its solution} 
\label{LMA_sec}

\subsection{The linearized Monge-Amp\`ere equation}
The linearized Monge-Amp\`ere equation associated with a $C^2$ and locally uniformly convex potential $u$ defined on some subset of $\R^n$ is of the form
\begin{equation}L_u v:= \sum_{i, j=1}^{n} U^{ij} v_{ij}\equiv \text{trace}(U D^2 v)= g.
\label{LMAEq}
 \end{equation}
Here and throughout, $$U = (U^{ij}) = (\det D^{2} u) (D^{2} u)^{-1}$$ is the matrix of cofactors of the 
Hessian matrix $D^{2}u= (u_{ij})$. The coefficient matrix $U$ of $L_u$ arises from the linearization of the Monge-Amp\`ere operator $\det D^2 u$ because
$$U=\frac{\p (\det D^2 u)}{\p (D^2 u)}.$$
One can also note that $L_u v$ is the coefficient of $t$ in the expansion
$$\det D^2 (u + t v) =\det D^2 u + t~ \text{trace} (U D^2 v) + \cdots+ t^n \det D^2 v.$$
Typically, one assumes that $u$ solves the Monge-Amp\`ere equation
\begin{equation}\label{MAEq}
\det D^2 u = f~\text{for some function f satisfying the bounds }0 <\lambda \leq f \leq \Lambda
\end{equation} 
where $\lambda$ and $\Lambda$ are positive constants.
Given these bounds, $U$ is a positive semi-definite matrix. Hence, $L_{u}$ is a linear elliptic partial differential 
operator, possibly degenerate.

The linearized Monge-Amp\`ere operator $L_u$ captures two of the most important second order equations in PDEs from the simplest 
linear equation to one of the most important nonlinear equations. In fact, in the special case where $u$ is a quadratic polynomial, say $u(x)=\frac{1}{2}|x|^2$, $L_u$ becomes the 
Laplace operator: $\displaystyle L_u=\Delta= \sum_{i=1}^{n}\frac{\p^2}{\p x_i^2}$. On the other hand,
since $L_u u = n\det D^2 u$, the Monge-Amp\`ere equation is a special case of the linearized Monge-Amp\`ere equation. As $U= (U^{ij})$ is divergence-free (see Lemma \ref{divfreeU}), that is, $$
\displaystyle \sum_{i=1}^{n} \p_i U^{ij}=0$$ 
for all $j=1,\cdots, n$, the linearized Monge-Amp\`ere equation can be written in both divergence and double divergence form:
$$L_u v= \sum_{i, j=1}^{n} \p_i (U^{ij} v_{j}) =  \sum_{i, j=1}^{n} \p_{ij}(U^{ij} v).$$

\subsection{Linearized Monge-Amp\`ere equations in contexts}
$L_u$ appears in many contexts:
\begin{myindentpar}{1cm}
(1) Affine maximal surface equation in affine geometry (Chern \cite{Chern}, Trudinger-Wang \cite{TW00, TW05, TW08})
$$U^{ij} w_{ij}=0,~w= (\det D^2 u)^{-\frac{n+1}{n+2}} $$
(2) Abreu's equation (Abreu \cite{Ab}, Donaldson \cite{D1, D2, D3, D4}) in the context of existence 
of K\"ahler metrics of constant scalar curvature
in complex geometry
$$U^{ij}w_{ij}=-1,~ w = (\det D^2 u)^{-1}$$
A more familiar form of the Abreu's equation is
$$\sum_{i,j=1}^{n} \frac{\p^2 u^{ij}}{\p x_i\p x_j}=-1$$
where $(u^{ij}) = (D^2 u)^{-1}$ is the inverse matrix of $D^2 u$.\\
(3) Semigeostrophic equations in fluid mechanics (Brenier \cite{B}, Cullen-Norbury-Purser \cite{CNP}, Loeper \cite{Loe}).\\
(4) Regularity of the polar factorization for time dependent maps (Loeper \cite{Loe05}).
\end{myindentpar}
\subsection{Difficulties and expected regularity}
The classical regularity theory for uniformly elliptic equations with measurable coefficients deals with {\bf divergence} form operators 
$$L=\displaystyle \sum_{i, j=1}^{n}\frac{\p}{\p x_i}\left(a^{ij}\frac{\p }{\p x_j}\right)$$ or {\bf nondivergence}
form operators $$\displaystyle 
L= \sum_{i, j=1}^{n} a^{ij}\p_{ij}$$ with 
positive ellipticity
constants $\lambda$ and $\Lambda$, that is, the eigenvalues of the coefficient matrix $A= (a^{ij})$ are bounded between $\lambda$ and $\Lambda$. 
The important Harnack and H\"older estimates for {\bf divergence} form equations $Lu=0$  were established in the late 50s by De Giorgi-Nash-Moser \cite{DG, Na, Mo}.
The regularity theory in this case is connected with isoperimetric inequality, Sobolev embedding, Moser iteration, heat kernel, BMO (the space of functions of bounded mean 
oscillation).
On the other hand, the Harnack and H\"older estimates for {\bf nondivergence} form equations $Lu=0$ were established 
 only in the late 70s by Krylov-Safonov \cite{KS1, KS2}. The regularity theory is connected with the Aleksandrov-Bakelman-Pucci (ABP) maximum principle coming 
 from the Monge-Amp\`ere equation.

The linearized 
Monge-Amp\`ere theory investigates operators of the form $$\displaystyle L_u= \sum_{i, j=1}^{n} U^{ij}\p_{ij}$$ where it is only known that the {\it product of the eigenvalues} of the 
coefficient matrix $U$ is bounded between
two constants. This comes from (\ref{MAEq}) because $$\lambda^{n-1}\leq \det U =(\det D^2 u)^{n-1}\leq \Lambda^{n-1}.$$ 
Therefore, the linearized 
Monge-Amp\`ere operator $L_u$ is in 
general not uniformly elliptic, i.e., the eigenvalues of $U = (U^{ij})$ are not necessarily bounded away from $0$ and $\infty.$ Moreover, when considered in a bounded convex domain $\Omega$, $L_u$ can be possibly singular near 
the boundary. In other words, the linearized 
Monge-Amp\`ere equation can be both degenerate and singular. The degeneracy and 
singularity of $L_u$ are the main difficulties in establishing regularity results for its solutions.

A natural question is what regularity we can hope for solutions of the linearized Monge-Amp\`ere equation $L_u v=0$ under the structural assumption (\ref{MAEq}). At least on 
a heuristic level, they can be expected to be H\"older continuous. Indeed, 
strictly convex solutions of (\ref{MAEq}),
interpreted in the sense of Aleksandrov for $u$ not $C^2$ as in Definition \ref{Alek_defn}, are $C^{1,\alpha}$ for some $\alpha\in (0, 1)$ depending only 
on $n,\lambda$ and $\Lambda$. This follows from the regularity theory of the Monge-Amp\`ere equation; see Theorems \ref{C1alpha} and \ref{C1alpha2}. By 
differentiating (\ref{MAEq}), we 
see that each partial derivative $u_k=\frac{\p u}{\p x_k}$ ($k=1,\cdots, n$) is a 
solution of the inhomogeneous linearized Monge-Amp\`ere equation $$L_u u_k= f_k.$$
We can expect that the regularity for $v$ is that of $u_k$, which is $C^{\alpha}$, and hence it should be H\"older continuous. The theory of Caffarelli-Guti\'errez confirms this 
expectation.
\subsection{Affine invariance property} 
\label{AIP_sec}
The second order operator $L_u:= U^{ij}\p_{ij}$ is affine invariant, i.e., invariant with respect to linear transformations of 
the independent variable $x$ of the form $x\mapsto Tx$ with $\det T =1$. Indeed, for such $T$,
the rescaled functions
$$\tilde u(x) = u(Tx) ~\text{and}~\tilde v(x) = v (Tx)$$
satisfy the same structural conditions as in (\ref{LMAEq}) and (\ref{MAEq}) because
$$\det D^2 \tilde u(x) =\det D^2 u(Tx)= f(Tx)~\text{and}~L_{\tilde u} \tilde v (x) = L_u v(Tx)= g(Tx).$$
More generally, under the transformations
$$\tilde u(x) = u(Tx),~\tilde v(x)= v(Tx),$$
the equation (\ref{LMAEq})
becomes
$$~L_{\tilde u} \tilde v (x):=\tilde U^{ij} \tilde v_{ij}(x)= (\det T)^2 g(Tx).$$
The last equation follows from standard computation. We have
$$D \tilde u = T^{t} Du;~D^2 \tilde u = T^{t}(D^2 u) T;~D^2 \tilde v = T^{t}(D^2 v) T$$
and
\begin{eqnarray*}\tilde U = (\det D^2 \tilde u)(D^2 \tilde u)^{-1}=(\det T)^{2}(\det D^2 u) T^{-1} (D^2 u)^{-1} (T^{-1})^{t}=
(\det T)^2 T^{-1} U (T^{-1})^{t}. 
\end{eqnarray*}
Therefore, $$L_{\tilde u} \tilde v (x)= \text{trace} (\tilde U D^2\tilde v)= (\det T)^2 \text{trace} (UD^2 v(Tx))
= (\det T)^2 L_uv(Tx)= (\det T)^2 g(Tx).$$

The rest of the section will be devoted to interior regularity for solutions to the linearized Monge-Amp\`ere equation. We start by recalling Krylov-Safonov's Harnack inequality for  linear, uniformly elliptic equations
in non-divergence form.
\subsection{Krylov-Safonov's Harnack inequality}
In 1979, Krylov-Safonov \cite{KS1, KS2} established the Harnack inequality and H\"older estimates for solutions of linear elliptic equations in {\bf non-divergence} form
\begin{equation}Lv:=\sum_{i, j=1}^{n}a^{ij}\frac{\p^2 v}{\p x_i \p x_j}=0
 \label{unieq}
\end{equation}
where the eigenvalues of the coefficient matrix $A= (a^{ij})$ are bounded between two positive constants $\lambda$ and $\Lambda$, that is
\begin{equation}
 \label{unilam}
 \lambda I_n \leq (a^{ij})\leq \Lambda I_n.
\end{equation}
The following theorem is the celebrated result of Krylov-Safonov.
\begin{thm} 
[Krylov-Safonov's Harnack inequality, \cite{KS1, KS2}]
Assume $(a^{ij})$ satisfies (\ref{unilam}). Let $v$ be a
nonnegative solution of (\ref{unieq}) in $\Omega$. Then $v$ satisfies the Harnack inequality on Euclidean balls. More precisely, 
for all $B_{2r}(x_0)\subset\subset\Omega$, we have
\begin{equation}
 \label{HI1}
 \sup_{B_r(x_0)} v\leq C(n, \lambda, \Lambda) \inf_{B_r(x_0)} v. 
\end{equation}
\end{thm}
From the Harnack inequality (\ref{HI1}), we obtain a H\"older estimate

$$\abs{v(x)-v(y)}\leq C\abs{x-y}^{\alpha} \sup_{B_{2r}(x_0)} \abs{v}$$
for $x, y\in B_{r}(x_0)$ where $\alpha$ and $C$ are positive constants depending only on $n,\lambda, \Lambda$.
\begin{rem}~~
\begin{myindentpar}{1cm}
(i) The uniform ellipticity of $A(x)$ is invariant under rigid transformation of the domain, i.e., for any orthogonal matrix $O$, 
the matrix $A(O x)$ is also uniformly elliptic with the same 
ellipticity constants as $A(x)$. \\
(ii) Balls are invariant under orthogonal transformations.\\
(iii) One important fact, but hidden, in the regularity theory of uniformly elliptic equations is that the quadratic polynomials
$$P(x) = a + b\cdot x + \frac{1}{2}\abs{x}^2,~ b\in\R^n,$$
are ``potentials`` for $L$, that is
$$L(P)\approx 1$$
and level surfaces of $P(x)$ are all possible balls of $\R^{n}.$ Moreover, $$\abs{\nabla P(x)-b}\approx 1$$
for $x$ in the ring $B_2(b)\backslash B_1(b).$
\end{myindentpar}
\end{rem}
Krylov-Safonov theory makes crucial use of the ABP estimate which
bounds solution of $Lv=f$ using the boundary values of $v$ and $L^{n}$ norm of the right hand side. In general form, it states as follows; 
see \cite{Al, Ba, Pu} and also \cite[Theorem 9.1]{GT}.
\begin{thm} [ABP maximum principle]\label{ABPmax} Let $(a^{ij})$ be a measurable, positive definite matrix.
 For $u\in C^{2}(\Omega)\cap C^{0}(\overline{\Omega})$, we have
 $$\displaystyle \sup_{\Omega} u\leq \sup_{\partial\Omega} u + \frac{\emph{diam}(\Omega)}{n \omega_{n}^{1/n}}\left\|\frac{a^{ij} u_{ij}}{[\det (a_{ij})]^{1/n}}\right\|_{L^{n}(\Gamma^{+})}$$
 where $\Gamma^{+}$ is the upper contact set $$\Gamma^{+}=\{y\in\Omega| u(x) \leq u(y) + p\cdot (x-y)~\text{for all}~x\in\Omega,~\text{for some}~ p= p(y)\in \R^{n}\}
 .$$
\end{thm}

\subsection{Harnack inequality for the linearized Monge-Amp\`ere equation}
The regularity theory for the linearized Monge-Amp\`ere equation was initiated in the fundamental paper \cite{CG} by Caffarelli and Guti\'errez. They developed an interior Harnack 
inequality theory for nonnegative solutions of the homogeneous
equations $$L_u v=0,$$
where $L_u$ is defined as in (\ref{LMAEq}), in terms of the pinching of the Hessian determinant 
\begin{equation}
\lambda\leq \det D^{2} u\leq \Lambda.
\label{pinch1}
\end{equation}
Their approach is based on that of Krylov and Safonov \cite{KS1, KS2} on the Harnack inequality and
H\"older estimates for linear, uniformly elliptic equations in general form, with sections replacing Euclidean balls. Before stating precisely the Harnack inequality theory
of Caffarelli-Guti\'errez, we would like to see, at least heuristically, what objects are prominent in this theory. 
\begin{rem}~~
\begin{myindentpar}{1cm}
(i) By the affine invariance property of the linearized Monge-Amp\`ere equations (see Section \ref{AIP_sec}), it is not hard to imagine that
good estimates for the linearized Monge-Amp\`ere equations must be formulated on domains that are invariant under affine transformations. Balls are not affine invariant.\\
(ii) Clearly, after an affine transformation, an ellipsoid becomes another ellipsoid.\\
(iii) A very important class of ellipsoid-like objects in the context of the Monge-Amp\`ere equation and the linearized Monge-Amp\`ere equation are 
sections. 
\end{myindentpar}
\label{heur_rem}
\end{rem}
The notion of sections (or cross sections) of convex solutions to the Monge-Amp\`ere equation
 was first introduced and studied by Caffarelli \cite{C1,C2,C3,C4}, and plays an important role in his fundamental interior $W^{2,p}$ estimates \cite{C2}. Sections are defined 
as sublevel sets of convex solutions after subtracting their supporting hyperplanes. 
They have the same role as Euclidean balls have in the classical theory. The section of a convex function $u$ defined on $\overline{\Omega}$ with center $x_0$ in $\overline{\Omega}$ and height $t$ is defined by
$$S_u(x_0, t)=\{x\in\overline{\Omega}: u(x) <u(x_0) + \nabla u(x_0) \cdot (x-x_0) + t\}.$$
After affine transformations, the sections of $u$ become sections of another convex function. 
\begin{exam} A Euclidean ball of radius $r$ is a section with height $r^2/2$ of the quadratic function $|x|^2/2$ whose Hessian determinant is $1$.
For $u(x)= \abs{x}^2/2$, we have
$$S_u(x, h)= B_{\sqrt{2h}}(x)\cap \overline{\Omega}.$$
\end{exam}
An important fact is the convexity of sections. They can be normalized to look like balls (John's lemma, Lemma \ref{John_lem}). Illustrating (i) and (iii) in Remark
\ref{heur_rem}, we can consider the following example.

\begin{exam}Consider the functions $u (x_1, x_2) = \frac{x_1^2}{2\varepsilon} +\frac{\varepsilon}{2} x_2^2$ and $v(x_1, x_2)= 
\frac{x_1^2}{2\varepsilon} -\frac{\varepsilon}{2} x_2^2 + 1$ in $\R^2$ where $\e\in (0, 1)$. Then $\det D^2 u =1$ and
$$U^{ij} v_{ij} = 0.$$
We can compute for $\frac{1}{4}\leq r\leq \frac{1}{2} $ and $\frac{1}{4}\leq t\leq \frac{1}{2}$
\begin{myindentpar}{1cm}
 (i) $$\sup_{B_r(0)} v =\frac{r^2}{2\varepsilon} + 1; \inf_{B_r(0)} v= 1-\frac{\varepsilon}{2}r^2;
 \sup_{B_r(0)} v\geq \frac{1}{32\varepsilon}\inf_{B_r(0)} v.$$
 (ii) $$\sup_{S_u(0, t)} v = t + 1; \inf_{S_u(0, t)} v = 1- t.$$
\end{myindentpar}
The ratio $\sup v/\inf v$ does not depend on the eccentricity of the section $S_u(0, t)$ for the given range of $t$. This ratio becomes unbounded on balls around $0$ when
$\e\rightarrow 0.$
\end{exam}
Now, if $v$ is a nonnegative solution of the linearized Monge-Amp\`ere equation $L_u v=0$ in a section $S_u (x_0, 2h)\subset\subset \Omega$ then 
Caffarelli and Guti\'errez's theorem on the Harnack inequality says that the values of $v$ in the concentric section of half height are comparable with each other. More precisely,
we have the following:
\begin{thm} [Caffarelli-Guti\'errez's Harnack inequality, \cite{CG}]
\label{CGthm}
Assume that the $C^2$ convex function $u$ satisfies the Monge-Amp\`ere equation 
\begin{equation*}
\lambda\leq \det D^{2} u\leq \Lambda~\text{in}~\Omega.
\end{equation*}
Let
$v\in W^{2, n}_{\text{loc}}(\Omega)$ be a nonnegative solution of $$L_u v:= U^{ij} v_{ij}=0$$ in a section $S_u (x_0, 2h)\subset\subset \Omega$. Then
\begin{equation}\sup_{S_{u}(x_0, h)} v\leq C(n, \lambda, \Lambda) \inf_{S_{u}(x_0, h)} v.
 \label{HI2}
\end{equation}
\end{thm}
This theory of Caffarelli and Guti\'errez is an affine invariant version of the classical Harnack inequality for uniformly elliptic equations with measurable coefficients. In fact, 
since the linearized Monge-Amp\`ere  operator $L_u $ can be 
written in both divergence form and non-divergence form,  Caffarelli-Guti\'errez's theorem is the affine invariant analogue of De Giorgi-Nash-Moser's theorem \cite{DG, Na, Mo} and also 
Krylov-Safonov's theorem \cite{KS1, KS2} on H\"older continuity of solutions to uniformly elliptic equations in divergence and nondivergence form, respectively. 
\begin{rem}
 The Harnack estimate (\ref{HI2}) also holds for nonnegative solutions to equations of the form
$$\text{trace}(A(x) U D^2 v)=0$$
with $A$ uniformly elliptic
$$C^{-1}I_n\leq A(x) \leq C I_{n}.$$
Thus, when $u(x) =\frac{1}{2}\abs{x}^2,$ we obtain the Krylov-Safonov's Harnack inequality for uniformly elliptic equations. Therefore, Harnack inequality also works for $$a^{ij} v_{ij}=0$$
with
$$\tilde{\lambda} (D^2 u)^{-1}\leq (a^{ij})\leq \tilde{\Lambda} (D^2 u)^{-1}.$$ In this case, we have a $\text{Hessian}^{-1}$-like elliptic equation.
\end{rem}
The Harnack inequality (\ref{HI2}) implies the geometric decay of the oscillation of the solution on sections with smaller height
and gives the $C^{\alpha}$ estimate for solution. Quantitatively, this says that if $v$ solves $L_u v=0$ in $S_u(x_0, 2)\subset\subset\Omega$ then
$v$ is $C^{\alpha}$ in $S_u(x_0, 1)$ and
$$\|v\|_{C^{\alpha}(S_u(x_0, 1))}\leq C(n, \lambda, \Lambda, S_u(x_0, 2)) \|v\|_{L^{\infty}(S_u(x_0, 2))}.$$
The important point to be emphasized here is that $\alpha$ depends only on $n, \lambda, \Lambda$ and the dependence of $C$ on $S_u(x_0, 2)$ can be actually removed in applications if we use 
affine transformations to transform the convex set $S_u(x_0, 2)$ into a convex set comparable to the unit Euclidean ball. The latter point follows from John's lemma (see Lemma
 \ref{John_lem})
on inscribing ellipsoid of maximal volume of a convex set \cite{John}. In fact, we can obtain interior H\"older estimate for inhomogeneous equations.
\begin{thm}[Interior H\"older estimate]\label{inho_Holder}
 Assume that $\lambda\leq \det D^2 u\leq \Lambda$ in a convex domain $\Omega\subset\R^n$ with $u=0$ on $\p\Omega$ where $B_1 (0)\subset \Omega\subset B_n(0)$.
 Let  $f\in L^n(B_1(0))$ and $v\in W^{2,n}_{loc}(B_1(0))$ be a solution of 
$U^{ij} v_{ij}= f$ in $B_1(0)$. Then there exist constants $\beta\in (0,1)$ and $C>0$ depending only on  $n$, $\lambda$, and $\Lambda$ such that
\[
|v(x) - v(y)|\leq C |x-y|^\beta \Big( \|v\|_{L^\infty(B_1(0))} +\|f\|_{L^n(B_1(0))}\Big)\quad \text{for all } x,y\in B_{\frac12}(0).
\]
\end{thm}
The Harnack inequality (\ref{HI2}) is also true for more general hypotheses on the Monge-Amp\`ere measure $\mu=\det D^2 u$ such as
a suitable doubling property. We say that the Borel measure $\mu$ is {\it doubling with respect to the center of mass} on the sections of $u$ if there exist constants $\beta>1$ and $0<\alpha<1$ such that for all sections 
$S_u (x_0, t)$,
\begin{equation}
 \label{muDC}
 \mu (S_u(x_0, t)) \leq \beta \mu (\alpha S_u (x_0, t)).
\end{equation}
Here $\alpha S_u (x_0, t)$ denotes the $\alpha$-dilation of $S_u(x_0, t)$ with respect to its center of mass $x^{\ast}$ (computed with respect to the Lebesgue measure):
$$\alpha S_u(x_0, t)= \{x^{\ast} + \alpha (x- x^{\ast}): x\in S_u(x_0, t)\}.$$

Maldonado \cite{Maldo}, extending the work of Caffarelli-Guti\'errez,  proved the following Harnack inequality for the linearized Monge-Amp\`ere equation under minimal geometric condition, namely, the doubling condition (\ref{muDC}).
\begin{thm}(\cite{Maldo}) Assume that 
$\det D^2 u=\mu$ satisfies (\ref{muDC}). 
For each compactly supported section $S_u (x, t)\subset\subset\Omega$, and any nonnegative solution $v$ of $L_u v=0$ in $S_u (x, t)$, we have 
$$\sup_{S_u(x,\tau t)} v\leq C \inf_{S_u(x, \tau t)} v$$
for universal $\tau, C$ depending only on $n, \beta$ and $\alpha$.
\label{MHolder_thm}
\end{thm}

For example, the Harnack inequality holds for $\mu$ positive polynomials.
If $u(x_1, x_2)= x_1^4 + x_2^2 $ then $\mu=\det D^2 u= C x_1^2$ is an admissible measure. The Harnack inequality applies to equation of the Grushin-type
\begin{equation}x_1^{-2}v_{11} + v_{22}=0.
 \label{Gru}
\end{equation}
\begin{rem}
Equation of the type (\ref{Gru}) is relevant in non-local equations such as fractional Laplace equation. 
By Caffarelli-Silvestre \cite{CS}, we can relate the fractional Laplacian 
$$(-\Delta )^{s} f(x) = C_{n, s} \int_{\R^{n}}\frac{f(x)-f(\xi)}{|x-\xi|^{n + 2s}} d\xi, $$
where the parameter $s$ is a real number between $0$ and $1$, and $C_{n,s}$ is some normalization constant, with solutions of the following extension problem. For a function $f : \R^n \to \R$, we consider the 
extension $v : \R^n \times [0,\infty) \to \R$ that satisfies the equations
\begin{align*}
v(x,0) = f(x) \label{eq:dirichletboundary},~
\Delta_x v + \frac{a}{y} v_y + v_{yy} = 0. 
\end{align*}
The last equation can also be written as
\begin{equation*}
\text{div} (y^{a} D v) = 0
\end{equation*}
which is clearly the Euler-Lagrange equation for the functional
\begin{equation*}
 J(v) = \int_{y>0} \abs{ D v }^2 y^{a} dX,~ X=(x, y). 
\end{equation*}
We can show that
\[ C (-\Delta)^s f = \lim_{y \to 0^+} -y^a v_y = \frac{1}{1-a} \lim_{y \to 0} \frac{v(x,y) - v(x,0)}{y^{1-a}}\]
for $s = \frac{1-a}{2}$ and some constant $C$ depending on $n$ and $s$, which reduces to the regular normal derivative in the case 
$a=0$.

If we make the change of variables $z = \left( \frac{y}{1-a} \right)^{1-a}$, we obtain a nondivergence form equation of the type (\ref{Gru})
\begin{equation*}
\Delta_x v + z^{\alpha} v_{zz} = 0
\end{equation*}
for $\alpha = \frac{-2a}{1-a}$. Moreover, $y^a v_y = v_z$. Thus, we can show that the following equality holds up to a multiplicative constant
\[ (-\Delta)^s f(x) = -\lim_{y \to 0^+} y^a v_y(x,y) = -v_z(x,0). \]
\end{rem}
\begin{rem}
The Harnack inequality in Theorem \ref{CGthm} has been recently extended to the boundary in \cite{Le_Bdr}.
\end{rem}
\newpage
\section[Interior Harnack and H\"older estimates for linearized Monge-Amp\`ere]{Interior Harnack and H\"older estimates for the linearized Monge-Amp\`ere equation}
In this section, we prove Theorems \ref{CGthm} and \ref{inho_Holder}.
\subsection[Proof of Caffarelli-Guti\'errez's Harnack inequality]{Proof of Caffarelli-Guti\'errez's Harnack inequality}
In this section, we prove Theorem \ref{CGthm} concerning Caffarelli-Guti\'errez's Harnack inequality for the linearized Monge-Amp\`ere equation.

We first briefly outline the proof of the Harnack inequality (\ref{HI2}) in Theorem \ref{CGthm}. Our proof adapts the general scheme in proving Harnack inequality
in Krylov-Safonov \cite{KS1, KS2}, Caffarelli-Cabr\'e \cite{CC}, Caffarelli-Guti\'errez \cite{CG}, Savin \cite{S_perb} and most recently Imbert-Silvestre \cite{IS}.

By using the affine invariant property of the linearized Monge-Amp\`ere equation as explained in Section \ref{AIP_sec}, we can rescale the domain, and the functions
$u$ and $v$. Furthermore, by changing coordinates and subtracting a supporting hyperplane to the graph of $u$ at $(x_0, u(x_0))$, 
we can assume that $x_0=0$, $u(0)=0$, $Du(0)=0$, $h=2$ and that the section $S_4= S_u(0, 4)\subset\subset\Omega$ is normalized, that is
$$B_1 (0) \subset S_4\subset B_n(0).$$
For simplicity, we denote $S_t= S_u(0, t)$.

A constant depending only on $\lambda,\Lambda$ and $n$ is called {\it universal}. We denote universal constants by $c, C, C_1, C_2, K, M,\delta, \cdots,$ etc.  
Their values may change from line to line.

From the engulfing property of sections in Theorem \ref{engulfthm}, we find that if $y\in S_{u}(x, t)$ then
$$S_{u}(x, t)\subset S_{u}(y, \theta_0 t)\subset S_{u}(x, \theta_0^2 t),$$
it suffices to show that if $v\geq 0$ in $S_2$ then $v\leq C(n, \lambda, \Lambda) v(0)$ in $S_1$.
\vglue 0.2cm
The idea of the proof is the following. We show that the distribution function of $v$, $|\{v>t\}\cap S_1|$ decays like $t^{-\varepsilon}$ ($L^{\e}$ estimate). Thus, $v\approx v(0)$ in $
S_1$ except a set of very small measure. If $v(x_0)\gg v(0)$ at some point $x_0$, then by the same method (now applying to $C_1- C_2 v$), we find $v\gg v(0)$ in a set of positive measure
which contradicts the above estimate. To study the distribution function of $v$, we slide generalized paraboloids associated with $u$ of constant opening, $P(x)= - a[u(x)
- u(y)-Du(y)\cdot (x-y)]$,
from below till they touch the graph of $v$ for the first time. These are the points where we use the equation and obtain the lower bound for the measure of the 
touching points. By increasing the opening of the sliding paraboloids, the set of touching points almost covers $S_1$ in measure. \\
There are three main steps in the proof of the $L^{\e}$ estimate.
\begin{myindentpar}{1cm}
{\bf Step 1: } Measure (ABP type) estimate.  
The rough idea is that
$$\text{Measure of contact points}~\geq~ c~\text{Measure of vertices}.$$
This step is not difficult. The reason why it works is the following. In the ABP estimate, we need the lower bound 
on the determinant of the coefficient matrix which is the case here.\\
{\bf Step 2: } Doubling estimate. This step is based on construction of subsolutions.\\
{\bf Step 3: } This step proves the geometric decay of $|\{v>t\}\cap S_1|$. It is based on a covering lemma which is a consequence
of geometric properties of sections. 
\end{myindentpar}
Our measure estimate in Step 1 states as follows.
\begin{lem}[Measure estimate]\label{meas_lem}
 Suppose that $v\geq 0$ is a solution of $L_u v=0$ in a normalized section $S_4$. 
 There are small, universal constants $\delta>0,\alpha>0$ and a large constant $M_1>1$ with the following properties. If 
 $\inf_{S_{\alpha}}v\leq 1$ then 
 $$|\{v>M_1\}\cap S_1|\leq (1-\delta) |S_1|.$$
\end{lem}
The key doubling estimate for Step 2 is the following lemma.
\begin{lem}[Doubling estimate]
\label{double_lem} Suppose that $v\geq 0$ is a solution of $L_u v=0$ in a normalized section $S_4$. Let $\alpha$ be the small constant in Lemma \ref{meas_lem}.
If $v\geq 1$ in $S_\alpha$ then $v\geq c(n, \lambda, \Lambda)$ in $S_{1}.$
\end{lem}
Combining Lemmas \ref{meas_lem} and \ref{double_lem}, and letting $M:= M_1 c(n,\lambda,\Lambda)^{-1}$, we obtain the following result: 
\begin{prop}[Critical density estimate]
 \label{decay_rem}
 Suppose that $v\geq 0$ is a solution of $L_u v=0$ in a normalized section $S_4$. 
 There is a small, universal constant $\delta>0$ and a large constant $M>1$ with the following properties. If 
 $$|\{v>M\}\cap S_1|> (1-\delta) |S_1|$$
 then $v>1$ in $S_1$.
\end{prop}
From the critical density estimate and the growing ink-spots lemma stated in Lemma \ref{inkspots}, we obtain the $L^{\e}$ estimate and completing the proof of Step 3.
\begin{thm}[Decay estimate of the distribution function] \label{decay_thm} Suppose that $v\geq 0$ is a solution of $L_u v=0$ in a normalized section $S_4$ with 
$$\inf_{S_u(0, 1)} v\leq 1.$$
Then there are universal constants $C_1>1$  and $\e\in (0, 1)$ such that for all $t>0$, we have
\[ |\{ v > t \} \cap S_1 | \leq C_1 t^{-\eps}.\]
\end{thm}

\begin{proof}[Proof of Theorem \ref{decay_thm}] Let $\delta\in (0, 1)$ and $M>1$ be the constants in Proposition \ref{decay_rem}. The conclusion of the theorem follows from the following decay estimate
for  $A_k := \{v > M^k\} \cap S_1$:
\[ |A_k| \leq C_2 M^{-\eps k}. \]
Note that $A_k$'s are open sets and $A_k \subset A_1$ for all $k\geq 1$. Recalling $\inf_{S_1} v \leq 1$, by Proposition \ref{decay_rem},
we have 
$$|A_k|\leq |A_1| \leq (1-\delta)|S_1|~\text{for all }k.$$ 

From Proposition \ref{decay_rem}, we find that if  a section $S \subset S_1$ satisfies 
$|S \cap A_{k+1}| > (1-\delta) |S|$, then $S \subset A_{k}$. Using Lemma \ref{inkspots}, we obtain
\[ |A_{k+1}| \leq (1-c\delta) |A_k|, \]
and therefore, by induction, $$|A_k| \leq (1-c\delta)^{k-1} (1-\delta) |S_1| = C_2 M^{-\eps k},$$ where $\eps = -\log(1-c\delta) / \log M$ and 
$C_2=(1-c\delta)^{-1}(1-\delta)|S_1|$.
This finishes the proof.
\end{proof}
\begin{proof}[Proof of Theorem~\ref{CGthm}]
Let $\delta\in (0, 1)$ and $M>1$ be the constants in Proposition \ref{decay_rem} and $\e\in (0, 1)$ be the constant in Theorem \ref{decay_thm}.
By a covering argument, our theorem follows from the following claim.\\
{\bf Claim 1.} $$\sup_{S_u(0,1/2)} v \leq C \inf_{S_u(0, 1/2)} v. $$
This in turns follows from the following claim.\\
{\bf Claim 2.} If $\inf_{S_u(0, 1/2)} v \leq 1 $ then for some universal constant $C$, we have $\sup_{S_u(0, 1/2)} v \leq C.$

Indeed, for each $\tau>0$, the function $$v^{\tau}= \frac{v}{\inf_{S_u(0, 1/2)} v  +\tau  }$$
satisfies
$a^{ij} v^{\tau}_{ij} = 0.$ We apply {\bf Claim 2} to $v^{\tau}$ to obtain
$$\sup_{S_u(0,1/2)} v \leq C \left(\inf_{S_u(0, 1/2)} v  +\tau \right).$$
Sending $\tau\rightarrow 0$, we get the conclusion of {\bf Claim 1.}

It remains to prove {\bf Claim 2.} Let $\beta >0$ be a universal constant to be determined later and let $h_t(x) = t(1-u(x))^{-\beta}$ be
defined in $S_u(0, 1)$. We consider the minimum value of $t$ such that
$h_t \geq v$ in $S_u(0, 1)$. It suffices to show that $t$ is universally bounded by a constant $C$ because we have then
$$\sup_{S_u(0,1/2)} v \leq C \sup_{S_u(0,1/2)} (1- u(x))^{-\beta} \leq 2^{\beta} C.$$
If $t \le 1$, we are
done. Hence, we further assume that $t \ge 1$. 

Since $t$ is chosen to be the \emph{minimum} value such that $h_t \geq
v$, then there must exist some $x_0 \in S_u(0, 1)$ such that $h_t(x_0) =
v(x_0)$. Let $r = (1-u(x_0))/2$.  Let $H_0 := h_t(x_0) = t(2r)^{-\beta}
\ge 1$. By Theorem \ref{pst},
there is a small constant $c$ and large constant $p_1=\mu^{-1}$ such that $S_u(x_0, 2cr^{p_1})\subset S_u(0, 1)$. 
We bound $t$ by estimating the measure of the set $\{v \geq H_0/2\} \cap
S_u(x_0, cr^{p_1})$ from above and below. 

The estimate from above can be done using Theorem
\ref{decay_thm} which then says that
\begin{equation} \label{up_H}  
|\{v>H_0/2\} \cap S_u(x_0, cr^{p_1})| \leq |\{v>H_0/2\} \cap S_u(0, 1)| \leq
CH_0^{-\eps} = C t^{-\eps} (2r)^{\beta \eps}.
\end{equation}

To estimate the measure of $\{v \geq H_0/2\} \cap
S_u(x_0, cr^{p_1})$ from below, we apply Theorem 
\ref{decay_thm} to $C_1-C_2v$ on a small but definite fraction of this section.  Let $\rho$ be the small universal
constant and $\beta$ be a large universal constant such that
\begin{equation}\label{beta_choice}
 M \left( 1-\rho)^{-\beta} - 1 \right)  \leq
 \frac 12,~
\beta  \geq \frac{n}{2\mu\e}.
\end{equation}

Consider the section $S_u(x_0, c_1 r^{p_1}) $ where $c_1\leq c$ is small. We claim that $1- u(x)\geq 2r-2\rho r$ in this section. Indeed, if $x\in S_u(x_0, c_1 r^{p_1}) $
then by Lemma \ref{sec-size}, we have $|x-x_0| \leq C (c_1 r^{p_1})^{\mu} \leq c \rho r$ for small $c_1$ and hence, by the gradient estimate in Lemma \ref{slope-est}
\begin{eqnarray*}
1-u(x) = 2r + u(x_0)- u(x) \geq 2r -(\sup_{S_u(0, 1)}|D u|) |x-x_0| \geq 2r-2\rho r.
\end{eqnarray*}

The maximum of $v$ in the section $S_u(x_0, c_1 r^{\gamma})$ is at most the maximum
of $h_t$ which is not greater than $t(2r-2\rho r)^{-\beta} =
(1-\rho)^{-\beta} H_0$. 
Define the 
following function for $x\in S_u(x_0, c_1 r^{p_1}) $

\[ w(x) = \frac{(1-\rho)^{-\beta} H_0
  -v(x)}{\left((1-\rho)^{-\beta} - 1
  \right) H_0}. \] Note that $w(x_0) = 1$, and $w$ is a non-negative solution of $L_u w=0$ in $S_u(x_0, c_1 r^{p_1})$.
Using Proposition \ref{decay_rem}, we obtain
\[ |\{w \leq M\} \cap S_u(x_0, 1/4 c_1 r^{p_1})| \geq \delta |S_u(x_0, 1/4 c_1 r^{p_1})|. \]

In terms of the original function $v$, this is an estimate of a set
where $v$ is larger than
\[ H_0 \left((1-\rho)^{-\beta} - M \left(
( 1-\rho)^{-\beta} - 1 \right) \right) \geq
\frac{H_0}2, \]
because of the choice of $\rho$ and $\beta$. Thus, we obtain the estimate
\[ |\{v \geq H_0/2\} \cap S_u(x_0, c_1 r^{p_1})| \geq \delta |S_u(x_0, c_1 r^{p_1})|.\]
In view of  \eqref{up_H}, and the volume estimate on sections in Theorem \ref{vol-sec1}, we find
$$C t^{-\eps} (2r)^{\beta \eps}\geq  \delta |S_u(x_0, c_1 r^{p_1})|\geq c(n,\lambda,\Lambda) r^{np_1/2} = c(n,\lambda,\Lambda) r^{\frac{n}{2\mu}}.$$
By the choice of $\beta$ in (\ref{beta_choice}), we find that $t$ is universally bounded. 
\end{proof}

In the proof of Theorem \ref{decay_thm}, we use the following consequence of Vitali's covering
lemma. It is often referred to as the growing ink-spots lemma which was first introduced by Krylov-Safonov \cite{KS2}. The term ''growing ink-spots lemma`` was coined
by E. M. Landis.

\begin{lem}[Growing ink-spots lemma] \label{inkspots}
Suppose that $u$ is a strictly convex solution to the Monge-Amp\`ere equation
$\lambda\leq \det D^2 u\leq\Lambda$ in a bounded and convex set $\Omega\subset\R^n$. Assume that for some $h>0$, 
$S_u(0, 2h)\subset\subset\Omega.$

Let $E \subset F \subset S_u(0, h)$ be two open sets. Assume that for some constant $\delta \in (0,1)$, the
following two assumptions are satisfied.
\begin{itemize}
\item If any section $S_u(x, t) \subset S_u(0, h)$ satisfies $|S_u(x, t) \cap E| > (1-\delta)
  |S_u(x, t)|$, then $S_u(x, t) \subset F$.
\item $|E| \leq (1-\delta) |S_u(0, h)|$. 
\end{itemize}
Then $|E| \leq (1-c\delta) |F|$ for some constant $c$ depending only on $n,\lambda$ and $\Lambda$.
\end{lem}

\begin{proof}
For every $x \in F$, since $F$ is open, there exists some maximal section which is contained in $F$ and contains $x$. 
We choose one of those sections for each $x \in F$ and call it $S_u(x,\bar h(x))$.

If $S_u(x,\bar h(x)) = S_u(0, h)$ for any $x \in F$, then the result of the lemma follows immediately since 
$|E| \leq (1-\delta) |S_u(0, h)|$, so let us assume that it is not the case.

We claim that $|S_u(x,\bar h(x)) \cap E| \leq (1-\delta)|S_u(x,\bar h(x))|$. 
Otherwise, we could find a slightly larger section $\tilde S$ containing $S_u(x,\bar h(x))$ such that $|\tilde S \cap E| > (1-\delta) |\tilde S|$ and 
$\tilde S \not\subset F$, contradicting the first hypothesis.

The family of sections $S_u(x,\bar h(x))$ covers the set $F$. By the Vitali covering Lemma \ref{Vita_cov}, we can select a 
subcollection of non overlapping sections $S_j := S_u(x_j,\bar h(x_j))$ such that $F \subset \bigcup_{j=1}^{\infty} S_u(x_j, K\bar h(x_j))$ for
some universal constant $K$ depending only on $n,\lambda$ and $\Lambda$. The volume estimates in Lemma \ref{vol-sec1} then imply that $$|S_u(x_j, K\bar h(x_j))|\leq
C(n,\lambda,\Lambda) |S_u(x_j,\bar h(x_j))|$$ for each $j$.

By construction, $S_j \subset F$ and $|S_j \cap E| \leq (1-\delta) |S_j|$. Thus, we have that $|S_j \cap (F \setminus E)| \geq \delta |S_j|$. Therefore
\begin{align*}
|F \setminus E| \geq \sum_{j=1}^{\infty} |S_j \cap (F \setminus E)| 
 &\geq \sum_{j=1}^{\infty} \delta |S_j| \\
& \geq \frac{\delta}{C(n,\lambda,\Lambda)} \sum_{j=1}^\infty|S_u(x_j, K\bar h(x_j))|  \geq \frac{\delta}{C(n,\lambda,\Lambda)} |F|.
\end{align*}
Hence $|E|\leq (1-c\delta)|F|$ where $c = C(n,\lambda,\Lambda)^{-1}$.
\end{proof}
\begin{lem}[Vitali covering]\label{Vita_cov} Suppose that $\lambda\leq \det D^2 u\leq\Lambda$ in a bounded on convex set $\Omega\subset\R^n$. Then there exists a universal 
constant $K>1$ depending only on $n,\lambda$ and $\Lambda$
 with the following properties.
\begin{myindentpar}{1cm}
 (i) Let $\mathcal{S}$ be a collection of sections $S^x=S_u(x, h(x))\subset\subset\Omega$. Then there exists a countable subcollection of disjoint sections 
 $\displaystyle\bigcup_{i=1}^\infty S_u(x_i, h(x_i))$ such that
 $$\displaystyle \bigcup_{S^x\in \mathcal{S}} S^x\subset \bigcup_{i=1}^\infty S_u(x_i, Kh(x_i)).$$
 (ii) Let $D$ be a compact set in $\Omega$ and assume that to each $x\in D$ we associate a corresponding section $S_u(x, h(x))\subset\subset\Omega$. Then we can find a 
 finite number of these sections $S_u(x_i, h(x_i)), i=1,\cdots, m,$ such that
$$D \subset \bigcup_{i=1}^m S_u(x_i, h(x_i)),~\text{with}~ S_u(x_i, K^{-1} h(x_i))~\text{disjoint}.$$
\end{myindentpar}
\end{lem}

\begin{proof}[Proof of Lemma \ref{Vita_cov}] We use the following fact for sections compactly inclluded in $\Omega$: There exists a universal constant $K>1$ such that 
if $S_u(x_1, h_1)\cap S_u(x_2, h_2)\neq\emptyset$ and $2h_1\geq h_2$ then $S_u(x_2, h_2)\subset
S_u(x_1, Kh_1)$. The proof of this fact is based on the engulfing property of sections in Theorem \ref{engulfthm}. Suppose that
$x\in S_u(x_1, h_1)\cap S_u(x_2,  h_2)$ and $2h_1\geq h_2$. Then we have $ S_u(x_2, h_2)\subset S_u(x, \theta_0 h_2)
\subset S_u(x, 2\theta_0 h_1)$ and 
$x_1\in S_u(x_1, h_1)
\subset S_u(x, 2\theta_0 h_1)$. Again, by the engulfing property, we have $S_u(x, 2\theta_0 h_1)\subset S_u(x_1, 2\theta^2_0 h_1)$. 
It follows that $S_u(x_2, h_2)\subset S_u(x_1, 2\theta_0^2 h_1)$. The result
follows by choosing $K=2\theta^2_0$.\\
(i) From the volume estimate for sections in Lemma \ref{vol-sec1} and $S_u(x, h(x))\subset\subset \Omega$, we find that
$$H\equiv \sup\{h(x)| S^x\in\mathcal{S}\}\leq C(n,\lambda, \Lambda,\Omega)<\infty.$$
Define $$\mathcal{S}_i\equiv \{S^x\in \mathcal{S}| \frac{H}{2^i}<h(x) \leq \frac{H}{2^{i-1}}\}~(i=1,2,\cdots).$$ We define $\mathcal{F}_i\subset \mathcal{S}_i$ as follows. 
Let $\mathcal{F}_1$ be any maximal disjoint collection of sections in $\mathcal{S}_1$. By the volume estimate in Lemma \ref{vol-sec1}, $\mathcal{F}_1$ is finite. 
Assuming $\mathcal{F}_1,\cdots, \mathcal{F}_{k-1} $ have been selected, we choose $\mathcal{F}_k$
to be any maximal disjoint subcollection of
$$\left\{S\in \mathcal{S}_k| S\cap S^{x}=\emptyset~\text{for all~} S^x\in \bigcup_{j=1}^{k-1}\mathcal{F}_j\right\}.$$ Each $\mathcal{F}_k$ is again a finite set. 

We claim that the 
countable subcollection of disjoint sections $S_u(x_i, h(x_i))$ where $S^{x_i}\in \mathcal{F}:=\bigcup_{k=1}^{\infty} \mathcal{F}_k$ satisfies the conclusion of the lemma. To see this, it 
suffices to show that for any section $S^x\in \mathcal{S}$, there exists a section $ S^y\in \mathcal{F}$ such that $S^x\cap S^y\neq \emptyset$ and $S^x\subset S_u(y, Kh(y))$. 
The proof of this fact is simple. There is an index $j$ such that $S^x\subset \mathcal{S}_j$. By the maximality of $\mathcal{F}_j$, there is a section $S^y\in \bigcup_{k=1}^j 
\mathcal{F}_k$ with $S^x\cap S^y \neq\emptyset$. Because $h(y)> \frac{H}{2^j} $ and $h(x) \leq \frac{H}{2^{j-1}}$, we have $h(x) \leq 2 h(y)$.
By the fact established above, we have $S^x\subset S_u(y, K h(y))$.\\
(ii) We apply (i) to the collection of sections $S_u(x, K^{-1}h(x))$ where $x\in D$. 
Then there exists a countable subcollection of disjoint sections 
 $\left\{S_u(x_i, K^{-1}h(x_i))\right\}_{i=1}^{\infty}$ such that
 $$\displaystyle D \subset\bigcup_{x\in D}S_u(x, K^{-1}h(x))\subset \bigcup_{i=1}^\infty S_u(x_i, h(x_i)).$$

 By the compactness of $D$, we can choose a finite number of sections $S_u(x_i, h(x_i))$ $(i=1,\cdots, m)$ which cover $D$.
 \end{proof}
\begin{proof}[Proof of Lemma \ref{meas_lem}]
 Suppose $v(x_0)\leq 1$ at $x_0\in S_{\alpha}$ where $\alpha\in (0, 1/2)$. Consider the set of vertices $V=S_{\alpha}$. We claim there there is a large constant $a$ (called the opening) such that,
 for each $y\in V$, there is a constant $c_y$ such that
 the generalized paraboloid $ -a [u(x) -D u(y)\cdot (x-y) - u(y)] + c_y$ touches the graph of $v$ from below at some point $x$ (called the contact point) in $S_1$.
 Indeed, for each $y\in V$, we consider the function
 $$P(x) = v(x) + a [u(x) -D u(y)\cdot(x-y) - u(y)]$$
 and look for its minimum points on $\overline{S_1}$. On the boundary $\p S_1$, we have
 $$P\geq a [u(x) -D u(y)\cdot(x-y) - u(y)] \geq a C_1(n,\lambda,\Lambda)$$ by the Aleksandrov maximum principle. 
 At $x_0$, we have
 $$P(x_0)\leq 1 + a [u(x_0) -D u(y)\cdot(x_0-y) - u(y)] \leq 1 + a \alpha\theta_0.$$
 The last inequality follows from the engulfing property. Indeed, we have $x_0, y\in S_{\alpha}$ and hence by the engulfing property in Theorem
 \ref{engulfthm}, $x_0, y\in S_u(0,\alpha)\subset 
 S_u(y,\theta_0 \alpha)$.
 Consequently, $$u(x_0) -Du(y)\cdot(x_0-y) - u(y) \leq \theta_0\alpha.$$
 Thus, we can fix $\alpha>0$ small, universal and $a, M_1$ large such that
 $$M_1= 2 + a \alpha\theta_0 < a C_1.$$
 Therefore, $P$ attains its minimum at a point $x\in S_1$. Furthermore
 $$v(x) \leq P(x_0)< M_1.$$
At the contact point $x\in S_1$, we have
$$D v(x)= a (D u(y)-D u(x))$$
which gives
$$Du(y) = D u(x) +\frac{1}{a}D v(x).$$
We also have 
\begin{equation}
D^2 v(x) \geq -a D^2 u(x).
\label{uvD2}
\end{equation}
Hence
\begin{equation}D^2 u(y) D_x y = D^2 u(x) +\frac{1}{a} D^2 v(x)\geq 0.
 \label{uvD22}
\end{equation}
Now using the equation at only $x$, we find that
$$\text{trace} ((D^2 u)^{-1} D^2 v(x))=0.$$
This together with (\ref{uvD2}) gives
\begin{equation}C(a, n) D^2 u(x)\geq D^2 v(x) \geq -a D^2 u(x).
 \label{uvupdown}
\end{equation}
Here we use the following basic estimates. If $A\geq -B$ and $\text{trace} (B^{-1}A)=0$ then $$CB\geq A\geq -B.$$
Indeed, we can rewrite $$B^{-1/2} A B^{-1/2}\geq -I_n, ~\text{trace}(B^{-1/2} A B^{-1/2})=0.$$
Hence $$B^{-1/2} A B^{-1/2}\leq C(n) I_n.$$
Now, taking the determinant in (\ref{uvD22}) and invoking (\ref{uvupdown}), we obtain
$$\det D^2 u (y) \abs{\det D_x y}= \det (D^2 u(x) +\frac{1}{a} D^2 v(x))\leq C(a, n) \det D^2 u(x).$$
This implies the bound 
$$\abs{\det D_x y }\leq C(a, n,\Lambda,\lambda).$$
Then, by the area formula, the set $E$ of contact points $x$ satisfies
$$|S_{\alpha}|=\abs{V}= \int_{E}\abs{\det D_x y} \leq C(a, n,\Lambda,\lambda)\abs{E}\leq C |\{v<M_1\}\cap S_1|.$$
Using the volume estimate of sections in Lemma \ref{vol-sec1}, we find that $|S_1|\leq C^{\ast} |\{v<M_1\}\cap S_1|$ for some $C^{\ast}>1$ universal. 
The conclusion of the Lemma holds with $\delta= 1/C^{\ast}.$
\end{proof}

\begin{proof}[Proof of Lemma \ref{double_lem}] Recall that $u(0)=0, Du(0)=0$ and $B_1(0)\subset S_u(0, 4)\subset B_n(0)$.
To prove the lemma, it suffices to construct a subsolution $w: S_{2}\backslash S_{\alpha}\longrightarrow \R$, i.e.,
$U^{ij}w_{ij}\geq 0$,
with the following properties
\begin{myindentpar}{1cm}
 (i) $w\leq 0$ on $\partial S_2$\\
 (ii) $w\leq 1$ on $\partial S_\alpha$\\
 (iii) $w\geq c(n,\Lambda,\lambda)$ in $S_{1}\backslash S_{\alpha}.$
\end{myindentpar}
Our first guess is
$$w= C(\alpha, m) (u^{-m}- 2^{-m})$$
where $m$ is large. 

Let $(u^{ij})_{1\leq  i, j\leq n}$ be the inverse matrix $(D^2 u)^{-1}$ of the Hessian matrix $D^2 u$. 
We can compute for $W= u^{-m}-2^{-m}$
\begin{equation}u^{ij}W_{ij} = m u^{-m-2}[(m+1) u^{ij}u_{i}u_{j}-u u^{ij} u_{ij}] = mu^{-m-2}[(m+ 1) u^{ij} u_{i}u_{j}- n u].
 \label{uW1}
\end{equation}
By Lemma \ref{uv_trace} $$u^{ij}u_{i}u_{j} \geq \frac{\abs{D u}^2}{\text{trace} (D^2 u)}.$$
If $x\in S_2\setminus S_{\alpha}$ and $y=0$ then from from the convexity of $u$, we have
$0=u(y)\geq u(x)+ D u(x) \cdot (0-x)$ and 
therefore,
$$\abs{D u(x)}\geq \frac{u(x)}{\abs{x}}\geq \frac{\alpha}{n}\equiv 2c_{n}$$
for some constant $c_n$ depending only on $n,\lambda$ and $\Lambda$.

In order to obtain $u^{ij}W_{ij}\geq 0$  using (\ref{uW1}), we only have trouble when $\|D^2u\|$ is unbounded. But the set of bad points, i.e., 
where $\|D^2 u\|$ is large, is small. Here is how we see this.
Because $S_u(0, 4)$ is normalized, we can deduce from the Aleksandrov maximum principle, Theorem \ref{Alekmp} applied to $u-4$, that
$$\text{dist}(S_u(0, 3), \p S_u(0, 4))\geq c(n,\lambda,\Lambda)$$ for some universal $c(n,\lambda,\Lambda)>0$. By Lemma \ref{slope-est},
$D u$ is bounded on $S_3$. Now let $\nu$ denote the outernormal unit vector field on $\p S_3$. Then, using the convexity of $u$, we have $\|D^2u\|\leq \Delta u$ and thus,
by the divergence theorem,
\begin{equation*}\int_{S_3}\|D^2 u\| \leq \int_{S_3}\Delta u =\int_{\partial S_3} \frac{\partial u}{\p \nu}\leq C(n, \lambda,\Lambda).
\end{equation*}
Therefore, given $\varepsilon >0$ small, the set
$$H_{\varepsilon} = \{x\in S_3\mid \|D^2 u\|\geq \frac{1}{\varepsilon}\}$$
has measure bounded from above by
$$|H_{\e}|\leq C \e.$$

To construct a proper subsolution bypassing the bad points in $H_{\e}$,
we only need to modify $w$ at bad points. Roughly speaking, the modification involves the solution to 
$$\det D^2 u_{\e} = \Lambda\chi_{H_{\varepsilon}},~ u_{\e}=0 ~\text{on}~\partial S_4.$$
Here we use $\chi_E$ to denote the characteristic function of 
the set $E$: $\chi_E(x)=1$ if $x\in E$ and $\chi_E(x)=0$ if otherwise. The problem with this equation is that the solution is not in general smooth while we need two derivatives to construct the subsolution. But this
problem can be fixed, using approximation, as follows. 

We approximate $H_\e$ by an open set $\tilde H_\e$ where $H_\e\subset \tilde H_\e\subset S_4$ and the measure of their difference is small, that is
$$|\tilde H_\e\setminus H_\e|\leq \e.$$
We introduce a smooth function $\varphi$ with the following properties:
$$\varphi=1~\text{in}~H_\e,~\varphi=\e~\text{in}~S_4\setminus \tilde H_\e,~\e\leq \varphi\leq 1~\text{in } S_4.$$
Let $h_{\e}$ be the solution to
$$\det D^2 h_{\e} = \Lambda\varphi,~ h_{\e}=0 ~\text{on}~\partial S_4;$$
see Theorem \ref{muthm}.
By Caffarelli's $C^{2,\alpha}$ estimates \cite{C2}, $h_\e\in C^{2,\alpha}(S_4)$ for all $\alpha\in (0, 1)$.
From the Aleksandrov maximum principle, Theorem \ref{Alekmp}, we have on $S_4$
$$|h_{\e}|\leq C_n \text{diam} (S_4) \left(\int_{S_4} \Lambda \varphi\right)^{1/n}.$$
We need to estimate the above right hand side.
From the definitions of $\tilde H_\e$ and $\varphi$, we can estimate
$$\int_{S_4}\Lambda\varphi= \int_{H_\e}\Lambda + \int_{\tilde H_\e\setminus H_\e}\Lambda\varphi + \int_{S_4\setminus \tilde H_\e}\e
\leq \Lambda|H_\e| + \Lambda |\tilde H_\e\setminus H_\e| + \e C(n,\lambda,\Lambda)\leq C(n,\lambda,\Lambda)\e.$$
It follows that for some universal constant $C_1(n,\lambda,\Lambda)$, 
$$|h_\e|\leq C_1(n,\lambda,\Lambda)\e^{1/n}.$$
By the gradient estimate in Lemma \ref{slope-est}, we have on $S_2$
$$|D h_{\e}(x)|\leq \frac{- h_{\e}(x)}{\text{dist}(S_3, \p S_4)}\leq C_2(n,\lambda,\Lambda)\e^{1/n}.$$
We choose $\e$ small so that 
\begin{equation}C_1(n,\lambda,\Lambda)\e^{1/n} \leq 1/4, ~C_2(n,\lambda,\Lambda)\e^{1/n}\leq c_n.
 \label{epchoice}
\end{equation}
Let $$\tilde V = (u - h_{\e})~ \text{and} ~\tilde W = \tilde V^{-m}- 2^{-m}.$$
Then
\begin{equation}
 \label{gradV} |\tilde V|\leq 3~\text{and}~
 |D \tilde V |\geq c_n~\text{on}~ S_2\setminus S_\alpha;~\alpha \leq \tilde V \leq 1 + 1/4 =5/4~\text{on}~S_1\backslash S_{\alpha}.
\end{equation}
Now, compute as before
$$u^{ij}\tilde W_{ij} = m \tilde V^{-m-2}[(m+1) u^{ij}\tilde V_{i}\tilde V_{j} - \tilde V u^{ij} \tilde V_{ij}]=
m \tilde V^{-m-2}[(m+1) u^{ij}\tilde V_{i}\tilde V_{j} + \tilde V (u^{ij} (h_{\e})_{ij}-n)].$$
We note that, by Lemma \ref{trlem},
$$u^{ij} ({h_\e})_{ij} = \text{trace} ((D^2 u)^{-1}D^2 h_{\e})\geq n (\det (D^2 u)^{-1} \det D^2 h_{\e})^{1/n}\geq n~\text{on}~H_{\e}.$$
It follows that
$$u^{ij}\tilde W_{ij}\geq 0~\text{on}~H_{\e}.$$
On $(S_2\setminus S_{\alpha})\backslash H_{\e}$, we have $\text{trace} (D^2 u)\leq n\e^{-1}$ and from (\ref{gradV})
\begin{eqnarray*}u^{ij}\tilde W_{ij} &\geq& m \tilde V^{-m-2}[(m+1) u^{ij}\tilde V_{i}\tilde V_{j}  -n\tilde V ]\\
 &\geq& m \tilde V^{-m-2}[(m+1) \frac{|D \tilde V|^2}{\text{trace} (D^2 u)}-n\tilde V]\\
 &\geq&  m \tilde V^{-m-2} [(m+1)n^{-1}\e c_n-n\tilde V] \geq 0 
 \end{eqnarray*}
if we choose $m$ large, universal.  Therefore, 
$$u^{ij}\tilde W_{ij} \geq 0~\text{on}~S_2\setminus S_{\alpha}$$
and hence $\tilde W = \tilde V^{-m}- 2^{-m}$ is a subsolution to $u^{ij} v_{ij}\geq 0$ 
on $S_2\setminus S_{\alpha}$.

Finally, by (\ref{gradV}) and $\tilde W\leq 0$ on $\p S_2$, we choose 
a suitable $C(\alpha, n,\lambda,\Lambda)$ so that 
the subsolution of the form
$$\tilde w = C(\alpha, n, \lambda,\lambda) (\tilde V^{-m}-  2 ^{-m})$$
satisfies $\tilde w \leq 1$ on $\p S_{\alpha}$. Now, 
we obtain the desired universal lower bound for $v$ in $S_{1}$ from $v\geq \tilde w$ on $S_1\backslash S_{\alpha}$ and $v\geq 1$ on $S_{\alpha}$.
 \end{proof}

\subsection{Proof of the interior H\"older estimates for the inhomogeneous linearized Monge-Amp\`ere equation} In this section, we prove Theorem \ref{inho_Holder}, following an argument
of Trudinger and Wang \cite{TW3}.

The following lemma is a refined version of the Aleksandrov-Bakelman-Pucci (ABP) maximum principle for convex domains.
\begin{lem}
\label{ABP_refined}
 Assume that $\Omega$ is a bounded, convex domain in $\R^n$. Let $$Lu(x) = \emph{trace} (A(x)D^2 u(x))$$ where $A$ is an $n\times n$ symmetric and 
 positive definite matrix in $\overline{\Omega}$. Then, for 
all $u\in C^2(\Omega)\cap C(\overline{\Omega})$,
$$\max_{\overline{\Omega}} u \leq \max_{\p\Omega} u + C(n) |\Omega|^{1/n} \left\|\frac{Lu}{(\det A)^{1/n}}\right\|_{L^n(\Omega)}.$$
\end{lem}
\begin{proof}
We use the ABP estimate, Theorem \ref{ABPmax}, and John's lemma, Lemma \ref{John_lem}. According to this lemma, there is an affine 
transformation $T(x) =Mx + b$ where $M$ is an $n\times n$ invertible matrix and $b\in \R^n$ such that
\begin{equation}\label{om_norm} B_1(0) \subset T(\Omega)\subset B_n(0).
\end{equation}
For $x\in T(\Omega)$, we define $$v(x) = u(T^{-1}x)~  \text{and}~ \tilde L v = \text{trace}(\tilde A(x) D^2 v(x))$$ where
$\tilde A(x) = M A(T^{-1}x) M^t$. We then compute $D^2 v(x) = (M^{-1})^t D^2 u(T^{-1}x) M^{-1}$ and hence
$$\tilde L v(x) =L u(T^{-1}(x)).$$
Applying the ABP to $v$ and $\tilde L v(x)$ on $T(\Omega)$, we find
\begin{equation}\displaystyle\max_{\overline{T(\Omega)}} v \leq \max_{\p T(\Omega)} v + C_1(n) \text{diam} (T(\Omega))
\left\|\frac{\tilde Lv}{(\det \tilde A)^{1/n}}\right\|_{L^n(T(\Omega))}.
\label{ABP_v}
 \end{equation}
By changing variables $x= T(y)$ for $x\in T(\Omega)$, we find from $\det \tilde A= (\det M)^2 \det A$ that
\begin{eqnarray}
 \left\|\frac{\tilde Lv}{(\det \tilde A)^{1/n}}\right\|_{L^n(T(\Omega))} = \frac{1}{(\det M)^{1/n}}\left\|\frac{ Lu}{(\det  A)^{1/n}}\right\|_{L^n(\Omega)}
 \label{uv_ABP}
\end{eqnarray}
From (\ref{om_norm}), we have $\det M\geq c(n)|\Omega|^{-1}$ and $\text{diam} (T(\Omega))\leq 2n$. Using these estimates in (\ref{ABP_v}) and (\ref{uv_ABP}), we obtain the conclusion of 
the lemma.
\end{proof}
By employing Lemma~\ref{ABP_refined} and the  interior Harnack inequality in Theorem \ref{CGthm} for 
nonnegative solutions to the homogeneous linearized Monge-Amp\`ere equations, we get:
\begin{lem}[Harnack inequality for inhomogeneous linearized Monge-Amp\`ere]\label{inho_Harnack}
Assume that $\lambda\leq \det D^2 u\leq \Lambda$ in a convex domain $\Omega\subset\R^n$.
Let $f\in L^n(\Omega)$ and  $v\in W^{2,n}_{loc}(\Omega)$ satisfy  $U^{ij} v_{ij}= f$ almost everywhere in $\Omega$.
Then  if  $S_u(x, t)\subset\subset \Omega$ and $v\geq 0$ in $S_u(x, t)$, we have
\begin{equation}\label{eq:Harnack}
\sup_{S_u(x, \frac{t}{2})}{v} \leq C(n,\lambda,\Lambda)\Big( \inf_{S_u(x, \frac{t}{2})}{v} + t^{\frac{1}{2}}  \, \|f\|_{L^n(S_u(x, t))}\Big).
\end{equation}
\end{lem}
\begin{proof}
Let $w$ be the solution of 
$$
 U^{ij} w_{ij} =f~\text{ in }S_u(x, t),~\text{and}~
w =0~\text{ on } \partial S_u(x, t).
$$
Then, by Lemma \ref{ABP_refined} and the volume bound on sections in Theorem \ref{vol-sec1}, we get
\begin{equation}
\label{uomax}
\sup_{S_u(x, t)}{|w|} \leq  C(n, \lambda) |S_u(x, t)|^{\frac{1}{n}}   \|f\|_{L^n(S_u(x, t))} \leq Ct^{1/2}  \|f\|_{L^n(S_u(x, t))}.
\end{equation}
Furthermore, we have $U^{ij} (v-w)_{ij}=0$ in $S_u(x, t)$ and $v-w\geq 0$ on $\partial S_u(x, t)$. Thus we conclude from the ABP maximum  
principle that $v-w\geq 0$ in $S_u(x, t)$. Hence, we 
can apply the interior Harnack inequality, Theorem \ref{CGthm}, to obtain \[
\sup_{S_u(x,\frac{t}{2})} (v- w) \leq C \inf_{S_u(x, \frac{t}{2})} (v- w),
\]
for some constant $C$ depending only on $n, \lambda,$ and $
\Lambda$, 
which then implies
\[
\sup_{S_u(x,\frac{t}{2})} v \leq C'\Big( \inf_{S_u(x,\frac{t}{2})} v
+ \sup_{S_u(x,\frac{t}{2})} |w|\Big)\leq C\Big( \inf_{S_u(x, \frac{t}{2})}{v} + t^{\frac{1}{2}}  \, \|f\|_{L^n(S_u(x, t))}\Big).
\]

\end{proof}

As a consequence of Lemma~\ref{inho_Harnack}, we obtain the following  oscillation estimate:

\begin{cor}\label{inho_osci}
Assume that $\lambda\leq \det D^2 u\leq \Lambda$ in a convex domain $\Omega\subset\R^n$.
Let $f\in L^n(\Omega)$ and  $v\in W^{2,n}_{loc}(\Omega)$ satisfy  $U^{ij} v_{ij}= f$ almost everywhere in $\Omega$.
Then  if  $S_u(x, h)\subset\subset \Omega$, we have
\begin{equation*}
\text{osc}_{S_u(x,\rho)}{v} \leq C\big(\frac{\rho}{h}\big)^\alpha \Big[
\text{osc}_{S_u(x,h)}{v}  + h^{\frac{1}{2 }}  \, \|f\|_{L^n(S_u(x, h))}\Big]\quad \mbox{for all}\quad \rho\leq h,
\end{equation*}
where $C,\,\alpha>0$  depend only on $n$, $\lambda$, and $\Lambda$, and $\displaystyle \text{osc}_E v :=\sup_{E} v -\inf_{E} v$.
\end{cor}
\begin{proof}
Let us  write $S_t$ for the section  $S_u(x, t)$. 
Set
\begin{align*}
m(t) := \inf_{S_t} v,\quad  M(t) :=\sup_{S_t} v,\quad \mbox{and} \quad \omega(t) := M(t) -m(t).
\end{align*}
Let $\rho\in (0, h]$ be arbitrary.
Then since  $\tilde v := v - m(\rho)$ is a nonnegative solution of $U^{ij} \tilde v_{ij}= f$ in $S_\rho$,
we can apply Lemma~\ref{inho_Harnack} to $\tilde v$ to obtain
\[
\frac{1}{C} \sup_{S_{\frac{\rho}{2}}}\tilde v\leq \inf_{S_{\frac{\rho}{2}}}\tilde v +\rho^{\frac{1}{2 }}  \, \|f\|_{L^n(S_\rho)}.
\]
It follows that for all $\rho\in (0, h]$, we have
\begin{align*}
\omega(\frac{\rho}{2})
= \sup_{S_{\frac{\rho}{2}}}\tilde v- \inf_{S_{\frac{\rho}{2}}}\tilde v\leq \big(1- \frac{1}{C}\big)\sup_{S_{\frac{\rho}{2}}}\tilde v  +\rho^{\frac{1}{2 }}  \, \|f\|_{L^n(S_\rho)}
\leq \big(1- \frac{1}{C}\big)\omega(\rho) +\rho^{\frac{1}{2 }}  \, \|f\|_{L^n(S_h)}.
\end{align*}
Thus, by the standard iteration we deduce that
\begin{align*}
\omega(\rho)
\leq C'\big(\frac{\rho}{h}\big)^\alpha \Big[ \omega(h) +h^{\frac{1}{2 }}  \, \|f\|_{L^n(S_h)}\Big],
\end{align*}
giving the conclusion of the corollary.
\end{proof}

\begin{proof}[Proof of Theorem \ref{inho_Holder}] By Lemma \ref{slope-est}, there is a constant 
$M>1$ depending only on $n,\lambda$ and $\Lambda$ such that $|Du(z)|\leq M$ for all $z\in B_{3/4}(0)$. By Theorem \ref{strict_thm}, 
 there exists a constant $r_0>0$ depending only on $n,\lambda$ and $\Lambda$ such that $S_u(z, r_0)\subset B_{3/4}(0)$ for 
 all $z\in B_{1/2}(0)$. The gradient bound implies that $B(z, \frac{r}{2M})\subset S_u(z, r)$ for all $z\in B_{1/2}(0)$ and $r\leq r_0$.
 Fix $x\in B_{1/2}(0)$. It suffices to prove the lemma for $y\in S_u(x, r_0/4)$. Let $r\in (0, r_0/2)$ be such that $y\in S_u(x, r)\backslash S_u(x, r/2)$. Then
 $|y-x|\geq \frac{r}{4M}$. The above corollary gives
 \begin{eqnarray*}
  |v(y)-v(x)|\leq \text{osc}_{S_u(x, r)} v &\leq& C (\frac{r}{r_0})^{\alpha} \left[ \|v\|_{L^{\infty}(S_\phi(x,r_0))}  + r_0^{\frac{1}{2 }}  \, \|f\|_{L^n(S_u(x, r_0))}\right]
  \\&\leq& C|x-y|^{\alpha}\left[ \|v\|_{L^{\infty}(B_1(0))}  +   \, \|f\|_{L^n(B_1(0)))}\right].
 \end{eqnarray*}

\end{proof}

\begin{rem}
The proof of Theorem \ref{CGthm} follows the presentation in \cite{Le_Harnack} where the case of lower order terms was treated. For related results, see also \cite{Maldo_H}.
\end{rem}

\newpage

\section{Global H\"older estimates for the linearized Monge-Amp\`ere equations}
\label{global-CG}
In this section, we prove Proposition \ref{global-holder} and Theorem \ref{global-h}. 

\subsection{Boundary H\"older continuity for solutions of non-uniformly elliptic equations}
\begin{proof} [Proof of Proposition \ref{global-holder}]
By considering the equation satisfied by $\frac{v}{\|\varphi\|_{C^{\alpha}(\p\Omega)} + \|g\|_{L^{n}(\Omega)}}$,
we can assume that
$$\|\varphi\|_{C^{\alpha}(\p\Omega)} + \|g\|_{L^{n}(\Omega)}=1$$
and we need to prove that 
$$|v(x)-v(x_{0})|\leq C|x-x_{0}|^{\frac{\alpha}{\alpha +2}}~\text{for all}~ x\in \Omega\cap B_{\delta}(x_{0}). $$
Moreover, without loss of generality,  we assume that $\lambda =1$ and
$$\Omega\subset \R^{n}\cap \{x_{n}>0\},~0\in\p\Omega.$$
Take $x_{0} =0$. 
By the ABP estimate in Theorem \ref{ABPmax} and the assumption $\det (a^{ij})\geq 1$, we have
$$|v(x)|\leq \|\varphi\|_{L^{\infty}(\p\Omega)} + C_{n}\text{diam} (\Omega) \|g\|_{L^{n}(\Omega)}\leq C_{0}~\forall~ x\in \Omega$$
for a constant $C_0>1$ depending only on $n$ and $\text{diam}(\Omega)$,
and hence, for any $\varepsilon \in (0,1)$
\begin{equation}|v(x)-v(0)\pm \e|\leq 3C_{0}:= C_{1}.
\label{gen-ineq}
\end{equation}
Consider now the functions
$$h_{\pm}(x) := v(x)- v(0)\pm \e\pm C_{1} (\inf \{y_{n}: y\in \overline{\Omega}\cap\partial B_{\delta_{2}}(0)\})^{-1} x_{n}$$
in the region $A:= \Omega\cap B_{\delta_{2}}(0)$ where $\delta_{2}$ is small to be chosen later.\\
Note that, if $x\in\partial \Omega$ with $$|x|\leq \delta_{1}(\e):= \e^{1/\alpha}$$ then, we have from $\|\varphi\|_{C^{\alpha}(\p\Omega)}\leq 1$ that
\begin{equation}
\label{bdr-ineq}|v(x)-v(0)| =|\varphi(x)-\varphi(0)| \leq |x|^{\alpha} \leq \e.
\end{equation}
It follows that, if we choose $\delta_{2}\leq \delta_{1}$ then from (\ref{gen-ineq}) and (\ref{bdr-ineq}), we have
$$h_{-}\leq 0, h_{+}\geq 0~\text{on}~\partial A.$$
On the other hand,
$$a^{ij}(h_{\pm})_{ij}= g~\text{in}~A.$$
The ABP estimate in Theorem \ref{ABPmax} applied in $A$ gives
$$h_{-}\leq  C_{n}\text{diam} (A) \|g\|_{L^{n}(A)}\leq C_{n}\delta_{2}~\text{in}~ A$$
and $$
h_{+}\geq - C_{n}\text{diam} (A) \|g\|_{L^{n}(A)}\geq  -C_{n}\delta_{2}~\text{in}~ A.$$
By restricting $\e\leq C_n^{\frac{-\alpha}{1-\alpha}}$, we can assume that
$$\delta_{1} = \e^{1/\alpha}\leq \frac{\e}{C_{n}}.$$
Then, for $\delta_{2}\leq \delta_{1}$, we have $C_{n}\delta_{2}\leq \e$ and thus, for all $x\in A$, we have
$$|v(x)-v(0)|\leq 2\e + C_{1} (\inf \{y_{n}: y\in \overline{\Omega}\cap\partial B_{\delta_{2}}(0)\})^{-1} x_{n}.$$
The uniform convexity of $\Omega$ gives
\begin{equation}
\inf \{y_{n}: y\in \overline{\Omega}\cap\partial B_{\delta_{2}}(0)\} \geq C_{2}^{-1}\delta^2_{2}.
\end{equation}
Therefore, choosing $\delta_{2}= \delta_{1}$, we obtain
$$|v(x)-v(0)|\leq 2\e + C_{1} (\inf \{y_{n}: y\in \overline{\Omega}\cap\partial B_{\delta_{2}}(0)\})^{-1} x_{n}= 2\e + \frac{2C_{1}C_{2}}{\delta_{2}^2}x_{n}~\text{in}~ A.$$
As a consequence, we have just obtained the following inequality
\begin{equation}
\label{op-ineq}
|v(x)-v(0)|\leq 2\e + \frac{2C_{1}C_{2}}{\delta_{2}^2}|x| = 2\e + 2C_{1}C_{2}\e^{-2/\alpha}|x|
\end{equation}
for all $x,\e$ satisfying the following conditions
\begin{equation}
\label{xe-ineq}
|x|\leq \delta_{1}(\e):= \e^{1/\alpha}, \e\leq C_{n}^{\frac{-\alpha}{1-\alpha}}: = c_{1}(\alpha, L, K, n).
\end{equation}
Finally, let us choose 
$\e = |x|^{\frac{\alpha}{\alpha + 2}}.$
It satisfies the conditions in (\ref{xe-ineq}) if 
$$|x|\leq \min\{c_{1}^{\frac{\alpha +2}{\alpha}}, 1\}:=\delta.$$
Then, by (\ref{op-ineq}), we have for all $x\in \Omega\cap B_{\delta}(0)$
$$|v(x)-v(0)| \leq C|x|^{\frac{\alpha}{\alpha + 2}},~C=2 + 2C_{1}C_{2}.$$
\end{proof}

Proposition \ref{global-holder}
gives the boundary H\"older continuity for solutions to the linearized Monge-Amp\`ere equation
$$U^{ij} v_{ij} =g$$
where $(U^{ij})$ is the cofactor matrix of the Hessian matrix $D^2 u$ of the convex function $u$ satisfying
$$\lambda\leq \det D^{2}u\leq \Lambda.$$
This combined with the interior H\"older continuity estimates of Caffarelli-Guti\'errez 
in Theorem \ref{inho_Holder} 
gives the global H\"older estimates for solutions to the linearized Monge-Amp\`ere equations on uniformly convex domains as stated in Theorem \ref{global-h}. 
The rest of this section will be devoted to the proof of 
these global H\"older estimates.

The main tool to connect the interior and boundary H\"older continuity for solutions to the the linearized Monge-Amp\`ere equation is
Savin's Localization Theorem at the boundary for the Monge-Amp\`ere equation. 
\subsection{Savin's Localization Theorem }
We now state the main tool used in the proof of Theorem \ref{global-h}, the localization theorem.\\  Let $\Omega\subset \R^{n}$ be a bounded convex set with
\begin{equation}\label{om_ass}
B_\rho(\rho e_n) \subset \, \Omega \, \subset \{x_n \geq 0\} \cap B_{\frac 1\rho}(0),
\end{equation}
for some small $\rho>0$. Here $e_n= (0, \cdots, 0, 1)\in \R^n$. Assume that 
\begin{equation}
\text{for each $y\in\p\Omega \cap\ B_\rho(0)$ there is a ball $B_{\rho}(z)\subset \Omega$ that is tangent to}~ \p 
\Omega~ \text{at y}. 
\label{tang-int}
\end{equation}
Let $u : \overline \Omega \rightarrow \R$, $u \in C^{0,1}(\overline 
\Omega) 
\cap 
C^2(\Omega)$  be a convex function satisfying
\begin{equation}\label{eq_u}
\det D^2u =f, \quad \quad 0 <\lambda \leq f \leq \Lambda \quad \text{in $\Omega$},
\end{equation} and 
assume that
\begin{equation}\label{0grad}
u(0)=0, \quad \nabla u(0)=0.
\end{equation}
If the boundary data has quadratic growth near $\{x_n=0\}$ then, as $h \rightarrow 0$, 
the section $S_u(0, h)$ of $u$ at $0$ with level $h$ is equivalent to a half-ellipsoid centered at 0; here we recall that
$$ S_u(x,h) :=\{y\in \overline \Omega:  u(y) < u(x) + \nabla u(x)\cdot (y- x) +h\}.$$
This is the content
of Savin's Localization Theorem proved in \cite{S1,S2}. Precisely, this theorem reads as follows.

\begin{thm}[Localization Theorem \cite{S1,S2}]\label{main_loc}
 Assume that $\Omega$ satisfies \eqref{om_ass}-(\ref{tang-int}) and $u$ satisfies 
\eqref{eq_u}, 
\eqref{0grad} above and,
\begin{equation}\label{commentstar}\rho |x|^2 \leq u(x) \leq \rho^{-1} 
|x|^2 \quad \text{on $\p \Omega \cap \{x_n \leq \rho\}.$}\end{equation}
Then, for each $h<k$ there exists an ellipsoid $E_h$ of volume $\omega_{n}h^{n/2}$ 
such that
$$kE_h \cap \overline \Omega \, \subset \, S_u(0, h) \, \subset \, k^{-1}E_h \cap \overline \Omega.$$

Moreover, the ellipsoid $E_h$ is obtained from the ball of radius $h^{1/2}$ by a
linear transformation $A_h^{-1}$ (sliding along the $x_n=0$ plane)
$$A_hE_h= h^{1/2}B_1,\quad \det A_{h} =1,$$
$$A_h(x) = x - \tau_h x_n, \quad \tau_h = (\tau_1, \tau_2, \ldots, 
\tau_{n-1}, 0), $$
with
$$ |\tau_{h}| \leq k^{-1} |\log h|.$$
The constant $k$ above depends only on $\rho, \lambda, \Lambda, n$.
\end{thm}

 The ellipsoid $E_h$, or equivalently the linear map $A_h$, 
provides useful information about the behavior of $u$ 
near the origin. From Theorem \ref{main_loc} we also control the shape of sections that are tangent to $\p \Omega$ at the origin. 

\begin{prop}\label{tan_sec}
Let $u$ and $\Omega$ satisfy the hypotheses of the Localization Theorem \ref{main_loc} at the 
origin. Assume that for some $y \in \Omega$ the section $S_u(y,h) \subset \Omega$
is tangent to $\p \Omega$ at $0$ for some $h \le c$ with $c$ universal. Then there exists a small 
 constant $k_0>0$ depending on $\lambda$, $\Lambda$, $\rho $ and $n$ such that
$$ D u(y)=a e_n 
\quad \mbox{for some} \quad   a \in [k_0 h^{1/2}, k_0^{-1} h^{1/2}],$$
$$k_0 E_h \subset S_u(y,h) -y\subset k_0^{-1} E_h, \quad \quad k_0 h^{1/2} \le dist(y,\p \Omega) \le k_0^{-1} h^{1/2}, \quad $$
with $E_h$ the ellipsoid defined in the Localization Theorem \ref{main_loc}.
\end{prop}

Proposition \ref{tan_sec}, proved in \cite{S3}, is a consequence of Theorem \ref{main_loc}. 
We sketch its proof here.
\begin{proof}[Proof of Proposition \ref{tan_sec}]
Assume that the hypotheses of the Localization Theorem \ref{main_loc} hold at the origin. For $a\ge 0$ we denote
$$S_a':=\{ x \in \overline \Omega| \quad u(x)<ax_n\} ,$$
and clearly $S_{a_1}'\subset S_{a_2}'$ if $a_1 \le a_2$.
The proposition easily follows once we show that $S_{ch^{1/2}}'$ has the shape of the ellipsoid $E_h$ for all small $h$.

From Theorem \ref{main_loc} we know $$S_u(0, h):=\{u<h\} \subset k^{-1} E_h \subset \{x_n \le k^{-1} h^{1/2} \} $$ and since $u(0)=0$ we use the convexity of $u$ and obtain
\begin{equation}\label{f_sub}
S_{kh^{1/2}}' \subset S_u(0, h) \cap \Omega.
\end{equation}
This inclusion shows that in order to prove that $S_{kh^{1/2}}'$ is equivalent to $E_h$ it suffices to bound its volume by below
$$|S_{kh^{1/2}}'| \ge c|E_h|.$$

From Theorem \ref{main_loc}, there exists $y \in \partial S_{\theta h}$ such that $y_n \ge k(\theta h)^{1/2}$. 
We evaluate $\tilde u:=u-k h^{1/2}x_n, $ at $y$ and find $$\tilde u(y) \le \theta h - k h^{1/2} k (\theta h)^{1/2} \le -\delta h,$$ for some $\delta>0$ provided that we choose $\theta$ small depending on $k$. 
Since $\tilde u=0$ on $\p S_{kh^{1/2}}'$ and $ \det D^2 \tilde u \le \Lambda$, we apply Lemma \ref{ABP_refined} to $-\tilde u$ which solves $U^{ij}(-\tilde u)_{ij}=
-n\det D^2 u$. We have 
$$\delta h \leq \max_{ S_{kh^{1/2}}'} -\tilde u \le C(\Lambda, n)|S_{kh^{1/2}}'|^{2/n},$$ hence $$ c h^{n/2} \le |S_{kh^{1/2}}'|.$$
\end{proof}

The quadratic separation from tangent planes on the boundary for solutions to the Monge-Amp\`ere equation is a crucial assumption in the Localization Theorem \ref{main_loc}. This is the case for $u$ in Theorem \ref{global-h} as proved in \cite[Proposition 3.2]{S2}.
\begin{prop}
Let $u$ be as in Theorem \ref{global-h}. Then, on $\p \Omega$, 
$u$ separates quadratically from its tangent planes on $\p \Omega$. This means that
if $x_0 \in 
\p \Omega$ then
\begin{equation}
 \rho\abs{x-x_{0}}^2 \leq u(x)- u(x_{0})-\nabla u(x_{0})\cdot (x- x_{0}) \leq 
\rho^{-1}\abs{x-x_{0}}^2,
\label{eq_u1}
\end{equation}
for all $x \in \p\Omega,$ for some small constant $\rho$ universal. 
\label{quadsep}
\end{prop}
\begin{proof}
We prove the Proposition for the case $x_0\in\p\Omega$. By rotation of coordinates, we can assume that $x_0=0$ and $$\Omega\subset \{x\in\R^n: x_n>0\}.$$
We denote a point $x=(x_1, \cdots, x_{n-1}, x_n)\in\R^n$ by $x= (x', x_n)$ where $x'= (x_1,\cdots, x_{n-1})$.
By the Aleksandrov maximum principle, we have that $u$ is universally bounded. Since $\Omega$ is uniformly convex at the origin and $\det D^2 u$ is bounded from above, we can use barriers and obtain that $l_0$, the tangent plane at the origin, has bounded slope. 
The proof of this fact is quite similar to that of Lemma \ref{Dubound}.
After subtracting this linear function from $u$ and $\phi=u|_{\p\Omega}$, we may 
assume $l_0=0$. Thus, $u\geq 0$ and it suffices to show that
\begin{equation}\rho\abs{x-x_{0}}^2 \leq u(x) \leq 
\rho^{-1}\abs{x-x_{0}}^2,
\label{u_quad1}
\end{equation}
for all $x\in\p\Omega$. Since $u$ is universally bounded, we only need to prove (\ref{u_quad1}) for $|x|$ universally small. 

Since $\phi=u|_{\p\Omega}$, $\p \Omega $ are $C^3$ at the origin,  we find that 
\begin{equation}\phi(x)=Q_0(x') + o(|x'|^3)~\text{for} ~x=(x', x_n)\in\p\Omega,
 \label{varphi_eq}
\end{equation}
 with $Q_0$ a cubic polynomial. Indeed, locally around $0$, $\p\Omega$ is given by the graph of a $C^3$ function $\psi$: for some $c$ small,
$$\p\Omega\cap B_c(0)=\{(x', x_n): x_n=\psi(x')\}.$$
Thus, we can write for $(x', x_n)\in \p\Omega\cap B_c(0):$
\begin{equation}x_n= Q_1 (x') +  o(|x'|^3)
 \label{psi_eq}
\end{equation}
 with $Q_1$ a cubic polynomial. Since $\phi\in C^3(\overline{\Omega})$, we can again write around $0$:
 \begin{equation*}\phi=Q_2(x) + o(|x|^3) ~\text{for} ~x=(x', x_n)\in \overline{\Omega}
\end{equation*}
 with $Q_2$ a cubic polynomial. Substituting (\ref{psi_eq}) into this equation, we obtain (\ref{varphi_eq}) as claimed.

Now we use (\ref{varphi_eq}). Because $u=\phi \ge 0$ on $\p\Omega$, $Q_0$ has no linear part and its quadratic part is given by, say
$$\sum_{i<n} \frac {\mu_i}{2} x_i^2  , \quad \mbox{with} \quad \mu_i \ge 0.$$
We need to show that $\mu_i>0$.

If $\mu_1=0$, then the coefficient of $x_1^3$ is $0$ in $Q_0$. Thus, if we restrict to $\p \Omega$ in a small neighborhood near the origin, then for 
all small $h$ the set $\{\phi <h \}$ contains $$ \{|x_1| \le r(h)h^{1/3}\} \cap \{ |x'| \le c h^{1/2} \} $$
for some $c>0$ and with $$r(h) \to \infty \quad \mbox{as $h \to 0$}.$$ Now $S_u(0, h)$ contains the convex set 
generated by $\{\phi <h\}$ thus, since $\Omega$ is uniformly convex,
$$|S_u(0, h)| \ge c'(r(h)h^{1/3})^3 h^{(n-2)/2} \ge c' r(h)^3 h^{n/2}.$$
On the other hand, since $\det D^2 u\geq \lambda$ and $$0 \le u \le h \quad \mbox{in $S_u(0,h)$}$$ we obtain from Lemma \ref{vol-sec2} that 
$$|S_u(0, h)| \le C(\lambda, n) h^{n/2},$$ and we contradict the inequality above as $ h \to 0$.
\end{proof}
\subsection{Proof of global H\"older estimates for the linearized Monge-Amp\`ere equation} 
\begin{proof}[Proof of Theorem \ref{global-h}] We recall from Proposition \ref{quadsep} that $u$ separates quadratically from its tangent planes on $\p\Omega$. Therefore, Proposition \ref{tan_sec} applies.
Let $y\in \Omega $ with $r:=\text{dist} (y,\partial\Omega) \le c,$ for $c$ universal, and consider the maximal section $S_u(y,\bar{h}(y))$ centered at $y$, i.e.,
$$\bar{h}(y)=\max\{h\,| \quad S_u(y,h)\subset \Omega\}.$$
When it is clear from the context, we write $\bar{h}$ for $\bar{h}(y)$.
By Proposition \ref{tan_sec} applied at the point $x_0\in \p S_u(y,\bar h) \cap \p \Omega,$ we have
 \begin{equation}\bar h^{1/2} \approx r,
\label{hr}
\end{equation}
and $S_u(y,\bar h)$ is equivalent to an ellipsoid $E$ i.e
$$cE \subset S_u(y, \bar h)-y \subset CE,$$
where
\begin{equation}E :=\bar h^{1/2}A_{\bar{h}}^{-1}B_1(0), \quad \mbox{with} \quad \|A_{\bar{h}}\|, \|A_{\bar h}^{-1} \| \le C |\log \bar h|; \det A_{\bar{h}}=1.
\label{eh}
\end{equation}
We denote $$u_y:=u-u(y)-D u(y)\cdot (x-y).$$
The rescaling $\tilde u: \tilde S_1 \to \R$ of $u$ 
$$\tilde u(\tilde x):=\frac {1}{ \bar h} u_y(T \tilde x) \quad \quad x=T\tilde x:=y+\bar h^{1/2}A_{\bar{h}}^{-1}\tilde x,$$
satisfies
$$\det D^2\tilde u(\tilde x)=\tilde f(\tilde x):=f(T \tilde x),  $$
and
\begin{equation}
\label{normalsect}
B_c(0) \subset \tilde S_1 \subset B_C(0), \quad \quad \tilde S_1=\bar h^{-1/2} A_{\bar h}(S_u(y, \bar h)- y),
\end{equation}
where $\tilde S_1$ represents the section of $\tilde u$ at the origin at height 1.

We define also the rescaling $\tilde v$ for $v$
$$\tilde v(\tilde x):= v(T\tilde x)- v(x_{0}),\quad \tilde x\in \tilde S_{1}.$$
Then $\tilde v$ solves
$$\tilde U^{ij} \tilde v_{ij} = \tilde g(\tilde x):= \bar{h} g(T\tilde x).$$
Now, we apply Caffarelli-Guti\'errez's interior H\"older estimates in Theorem \ref{inho_Holder} to $\tilde v $ to obtain
$$\abs{\tilde v (\tilde z_{1})-\tilde v(\tilde z_{2})}\leq C\abs{\tilde z_{1}-\tilde z_{2}}^{\beta} \{\norm{\tilde v }_{L^{\infty}(\tilde S_{1})} + \norm{\tilde g}_{L^{n}(\tilde S_{1})}\},\quad\forall \tilde z_{1}, \tilde z_{2}\in \tilde S_{1/2},$$
for some small constant $\beta\in (0,1)$ depending only on $n, \lambda, \Lambda$.\\
By (\ref{normalsect}), we can decrease $\beta$ if necessary and thus we can assume that
$$2\beta\leq \frac{\alpha}{\alpha + 2}: =2\gamma.$$ Note that, by (\ref{eh})
$$ \norm{\tilde g}_{L^{n}(\tilde S_{1})} = \bar{h}^{1/2}\norm{g}_{L^{n}(S_u(y, \bar{h}))}.$$
We observe that (\ref{hr}) and (\ref{eh}) give
$$B_{C r\abs{\log r}}(y)\supset S_u(y,\bar{h}) \supset S_u(y,\bar{h}/2)\supset B_{c\frac{r}{\abs{\log r}}}(y)$$
and
$$\text{diam} (S_u(y,\bar{h}))\leq Cr\abs{\log r}.$$
By Proposition \ref{global-holder}, we have
$$\norm{\tilde v }_{L^{\infty}(\tilde S_{1})} \leq C \text{diam} (S_u(y, \bar{h}))^{2\gamma} \leq C (r\abs{\log r})^{2\gamma}.$$
Hence
$$\abs{\tilde v (\tilde z_{1})-\tilde v(\tilde z_{2})}\leq C\abs{\tilde z_{1}-\tilde z_{2}}^{\beta}\{(r\abs{\log r})^{2\gamma}  + \bar{h}^{1/2}\norm{g}_{L^{n}(S_u
(y, \bar{h}))}\}~\forall \tilde z_{1}, \tilde z_{2}\in \tilde S_{1/2}.$$
 Rescaling back and using
$$\tilde z_1-\tilde z_2=\bar h^{-1/2}A_{\bar h}(z_1-z_2),$$
and the fact that
$$\abs{\tilde z_1-\tilde z_2}\leq \norm{\bar h^{-1/2}A_{\bar h}}\abs{z_1-z_2} \leq C \bar{h}^{-1/2}\abs{\log \bar{h}}\abs{z_1-z_2}\leq
C r^{-1}\abs{\log r}\abs{z_1-z_2},$$
we find
\begin{equation}|v(z_1)-v( z_2)|  \le  |z_1-z_2|^{\beta} \quad \forall  z_1, z_2 \in  S_u(y,\bar h/2) .
\label{oscv}
\end{equation}
Notice that this inequality holds also in the Euclidean ball $B_{c\frac{r}{\abs{\log r}}}(y)\subset S_u(y,\bar h/2)$. Combining this with Proposition \ref{global-holder}, we easily 
obtain that $$\|v\|_{C^\beta(\bar \Omega)} \le C,$$ for some $\beta\in (0,1)$, $C$ universal.\\ For completeness, we include the details. By rescaling the domain, we can assume that
$\Omega\subset B_{1/100}(0).$
We estimate $\frac{\abs{v(x)-v (y)}}{\abs{x-y}^{\beta}}$ for $x$ and $y$ in $\Omega$. Let $r_{x} = \text{dist}(x, \p\Omega)$ and $r_{y}= \text{dist} (y, \p\Omega).$ Suppose that 
$r_{y}\leq r_{x},$ say. Take $x_{0}\in\p\Omega$ and $ y_{0}\in \p\Omega$ such that $r_{x}= \abs{x- x_{0}}$ and $r_{y} = \abs{y-y_{0}}.$ From the interior H\"older estimates of Caffarelli-Guti\'errez, we only need to consider the case $r_{y}\leq r_{x}\leq c.$\\
Assume first that
$\abs{x-y}\leq c \frac{r_{x}}{\abs{\log r_{x}}}.$
Then $y\in B_{c \frac{r_{x}}{\abs{\log r_{x}}}}(x)\subset S_u(x,\bar{h}(x)/2).$ By (\ref{oscv}), we have
$$\frac{\abs{v(x)-v (y)}}{\abs{x-y}^{\beta}}\leq 1.$$
Assume finally that
$
\abs{x-y}\geq c \frac{r_{x}}{\abs{\log r_{x}}}.$
We claim that 
$r_{x}\leq C\abs{x-y}\abs{\log \abs{x-y}}.$
Indeed, if 
$$1>r_{x}\geq \abs{x-y}\abs{\log \abs{x-y}}\geq \abs{x-y} $$
then
$$r_{x}\leq \frac{1}{c}\abs{x-y}\abs{\log r_{x}}\leq \frac{1}{c}\abs{x-y}\abs{\log \abs{x-y}}.$$
Now, we have
$$\abs{x_0-y_0}\leq r_{x} + \abs{x-y} + r_{y}\leq C \abs{x-y}\abs{\log \abs{x-y}}.$$
Hence, by Proposition \ref{global-holder} and recalling $2\gamma =\frac{\alpha}{\alpha + 2},$
\begin{eqnarray*}\abs{v(x)-v (y)}&\leq& \abs{v(x)- v(x_{0})} + 
\abs{v(x_{0})- v(y_{0})} + \abs{v(y_{0})- v(y)}\\ & \leq& C \left(r_{x}^{2\gamma} + \abs{x_{0}- y_{0}}^{\alpha} + r_{y}^{2\gamma}\right) \\ &\leq& C\left(\abs{x-y}\abs{\log \abs{x-y}}\right)^{2\gamma}\leq C \abs{x-y}^{\beta}.
\end{eqnarray*}

\end{proof}

\part{The Monge-Amp\`ere equation}
In the following Sections \ref{MA1_sec} and \ref{MA2_sec}, we present the most basic geometric properties of solutions to the Monge-Amp\`ere equation that were used in the 
previous section. Good references for these sections
include the books by Guti\'errez \cite{G} and Figalli \cite{Fi}, the survey papers by Trudinger and Wang \cite{TW3}, De Philippis and Figalli \cite{DPF} and Liu and Wang \cite{LW}.

Important results in these sections include:
\begin{myindentpar}{1cm}
 1. Aleksandrov's maximum principle, Theorem \ref{Alekmp};\\
 2. John's lemma, Lemma \ref{John_lem};\\
 3. The comparison principle, Lemma \ref{comp-prin};\\
 4. The solvability of the nonhomogeneous Dirichlet problem with continuous boundary data, Theorem \ref{muthm}. \\
 5. The volume estimate for sections, Theorem \ref{vol-sec1};\\
 6. Caffarelli's localization theorem, Theorem \ref{Caf_loc};\\
 7. The size of sections, Lemma \ref{sec-size};\\
 8. Caffarelli's $C^{1,\alpha}$ regularity of strictly convex solutions, Theorems \ref{C1alpha} and \ref{C1alpha2};\\
 9. The engulfing property of sections, Theorem \ref{engulfthm};\\
10. The inclusion and exclusion property of sections, Theorem \ref{pst}.
\end{myindentpar}

\newpage
\section{Maximum principles and sections of the Monge-Amp\`ere equation}
\label{MA1_sec}
In this introductory section on the Monge-Amp\`ere equation, we will prove various maximum principles including 
Aleksandrov's maximum principle in Theorem \ref{Alekmp}, the Aleksandrov-Bakelman-Pucci maximum principle in Theorem \ref{ABPmax} and the comparison principle in Lemma \ref{comp-prin}.
 We will also prove John's lemma in Lemma \ref{John_lem} and use it to obtain 
  optimal volume estimates for sections in Theorem \ref{vol-sec1}. Moreover, we establish the solvability of the nonhomogeneous Dirichlet problem for
  the Monge-Amp\`ere equation with continuous boundary data
  in Theorem \ref{muthm}. 
\subsection{Basic definitions}  
Let $\Omega$ be an open subset of $\R^{n}$ and $u: \Omega\rightarrow \R$ .
\begin{defn}[Supporting hyperplane]
Given $x_0\in \Omega$, a supporting hyperplane to the graph of $u$ at $(x_0, u(x_0))$ is an 
affine function $$l(x) = u(x_0) + p\cdot(x-x_0)$$ where $p\in\R^n$ such that $u(x) \geq l(x)$ for all $x\in\Omega.$ 
\end{defn}
\begin{defn} (The normal mapping/subdifferential of $u$) The normal mapping $\p u(x_0)$ of $u$ at $x_0$ is the set of slopes of supporting hyperplanes
to the graph of $u$ at $(x_0, u(x_0))$: 
$$\p u(x_0)=\{p\in\R^n: u(x) \geq u(x_0) + p\cdot(x-x_0)~\text{for all~} x\in\Omega\}.$$
\label{pdefn}
\end{defn}
\begin{rem}
We note that $\p u(x_0)$ can be empty. 
If $u\in C^1 (\Omega)$ and $\p u(x) \neq\emptyset$ then $\p u(x) =\{Du(x)\}.$
If $u\in C^2(\Omega) $ and $\p u(x) \neq\emptyset$ then $D^2 u(x) \geq 0$, that is, the Hessian matrix $D^2u(x)$ is nonnegative definite.
The proof of the later fact is simple, using Taylor's theorem. Indeed, we have
$$u(x+ h) = u(x) + Du(x) \cdot h + \frac{1}{2}D^2 u(\xi) h\cdot h$$
where $\xi$ is on the segment between $x$ and $x+ h$. Now, since $\p u(x) =\{Du(x)\}$, we use $u(x+ h) \geq u(x) + Du(x) \cdot h$ to conclude.
\end{rem}
Central to the theory of the Monge-Amp\`ere equation is the Monge-Amp\`ere measure. The following definition and its content are due to Aleksandrov.
\begin{defn} [The Monge-Amp\`ere measure]
\label{MAdef}
Let $u:\Omega\rightarrow \R$ be a convex function.
Given $E\subset \Omega$, we define
$$\p u(E) = \bigcup_{x\in E} \p u(x).$$
Let
$$Mu(E) = |\p u(E)|.$$
Then $Mu: \mathcal{S}\rightarrow \bar \R=\R\cup\{+\infty\}$ is a measure, finite on compact sets where $\mathcal{S}$ is the Borel $\sigma$-algebra defined by
$$\mathcal{S}=\{E\subset \Omega: \p u(E) ~\text{is Lebesgue measurable}\}.$$
$Mu$ is called the Monge-Amp\`ere measure associated with the convex function $u$.
\end{defn}
We will prove the statements in Definition \ref{MAdef} in Section \ref{Ex_sec}.
One way to prove that $Mu$ defined above is a measure is to use the Legendre transform.
\begin{defn}[Legendre transform] The Legendre transform of $u:\Omega\subset\R^n\rightarrow\R$ is the function $u^{\ast}:\R^n\rightarrow\R\cup\{+\infty\}$ defined by
$$u^{\ast}(p)=\sup_{x\in\Omega} (x\cdot p- u(x)).$$
\label{Leg_defn}
 \end{defn}
 Since $u^{\ast}$ is a supremum of linear functions, it is a convex function in $\R^n$. If $\Omega$ is bounded and $u$ is bounded on $\Omega$ then $u^{\ast}$ is finite.

\begin{defn}[Aleksandrov solutions] Given an open convex set $\Omega\subset\R^n$ and a Borel measure $\mu $ on $\Omega$, a convex
function $u:\Omega \to \R$ is called an {\it Aleksandrov solution} to the Monge-Amp\`ere equation
\[
\det D^{2} u =\mu,
\]
if $\mu=Mu$ as Borel measures. When $\mu=f~dx $ we will simply say that $u$ solves 
\begin{equation*}
\det D^2 u=f
\end{equation*}
and this is the notation we use in these notes.
Similarly, when we write  $\det D^2 u \geq  \lambda \ (\leq \Lambda)$ we mean that $Mu \ge \lambda \,dx \ (\leq \Lambda\,dx)$.
\label{Alek_defn}
\end{defn}

\subsection{Examples and properties of the normal mapping and the Monge-Amp\`ere measure}
\label{Ex_sec}
Here are some examples of the normal mapping and the Monge-Amp\`ere measure. 
\begin{exam} [The normal mapping and Monge-Amp\`ere measure of a cone]
\label{cone_eg}
Let $\Omega= B_R(x_0)$ and $u(x) = a|x-x_0|$ for $x\in\Omega$ where $a>0$. Then
\begin{equation*}  
\p u(x) =\left\{
 \begin{alignedat}{1}
  a \frac{x-x_0}{|x-x_0|} ~&\text{if} ~  0<|x-x_0|<R, \\\
\overline{B_a(0)}~&\text{if}~x=x_0.
 \end{alignedat} 
  \right.
\end{equation*} 
As a consequence, we have
$Mu = |B_a(0)| \delta_{x_0}$ where $\delta_{x_0}$ is the Dirac measure at $x_0$.
\end{exam}
\begin{exam}\label{examC2}
If $u\in C^2 (\Omega)$ is convex then 
$Mu= (\det D^2 u) dx.$
\end{exam}
\begin{proof} [Proof of Example \ref{cone_eg}]
 If $0<|x-x_0|<R$, then $u$ is differentiable at $x$ with $Du(x)= a \frac{x-x_0}{|x-x_0|}.$ Hence $\p u(x)=\{a \frac{x-x_0}{|x-x_0|}\}.$ 
 
 Let us now show that 
 $\p u(x_0)=\overline{B_a(0)}$. If $p\in\R^n$ with $|p|\leq a$, then, for all $x\in\Omega$, we have
 $$u(x_0) + p\cdot(x- x_0)\leq |p||x-x_0|\leq a |x-x_0|= u(x).$$
 Therefore, $\overline{B_a(0)}\subset \p u(x_0)$. On the other hand, if $p\in \p u(x_0)$, then
 in the inequality
 $$u(x)\geq p\cdot(x-x_0) + u(x_0)= p\cdot(x- x_0)~\text{for all}~ x\in\Omega,$$
 we can choose $x= x_0 + R\frac{p}{|p|+\e}\in\Omega$ for any $\e>0$ to obtain
 $a \frac{R |p|}{|p|+\e}\geq \frac{R|p|^2}{|p|+\e},$
 or $|p|\leq a.$ Hence $\overline{B_a(0)}\subset \p u(x_0) $.
\end{proof}
The proof of Example \ref{examC2} is based on the following theorem.
\begin{thm} [Sard's theorem]
Let $\Omega\subset \R^{n}$ be an open set and $g: \Omega\rightarrow \R^{n}$ a $C^{1}$ function in $\Omega.$ If 
$S_0=\{x\in \Omega: \det~Jac~ g=0\},$
then
$|g(S_0)|=0.$
\end{thm}
\begin{proof} [Proof of Example \ref{examC2}]
We claim that: if u is convex and $C^2$ then $Du$ is one-to-one on
$$A=\{x\in\Omega: D^2 u(x)>0\}.$$
Indeed, suppose that $x_1, x_2\in A$ with $Du(x_1)= Du(x_2)$ then using convexity, we have
\begin{equation*}
 u(x_1)- u(x_2)\geq  Du (x_2) \cdot(x_1- x_2),~ u(x_2)- u(x_1)\geq Du (x_1) \cdot(x_2- x_1).
\end{equation*}
It follows that
\begin{equation}u(x_1)- u(x_2)= Du (x_1) \cdot(x_1- x_2)= Du(x_2)\cdot(x_1-x_2).
\label{Du11}
\end{equation}
Now, we use Taylor's formula
$$u(x_1)= u(x_2) + Du (x_2)\cdot(x_1- x_2) + \int_{0}^{1} t D^2 u(x_2 + t(x_1- x_2))(x_1- x_2)\cdot (x_1- x_2) dt.$$
Since $x_2\in A,$ we have $x_2 + t(x_1-x_2)\in A$ for $t$ small. Thus, from (\ref{Du11}), we must have $x_1= x_2.$

Since $u\in C^2$, we have $g= Du\in C^1(\Omega)$. Let $S_0=\Omega\backslash A$. Now, for any Borel $E\subset \R^{n}$, 
$$Du(E)= Du (E\cap S_0) \cup Du(E\backslash S_0).$$
From Sard's theorem, we have  $Mu(E\cap S_0)=|Du(E\cap S_0|\leq |g(S_0)|=0$. Since $\det D^2 u(x)=0$ on $S_0$, 
we have by the change of variables $y= Du(x)$ with $dy =\det D^2 u(x) dx,$
\begin{eqnarray*}
Mu(E)& =& Mu (E\cap S_0) + Mu(E\backslash S_0) = |Du (E\backslash S_0)|\\ &=&\int_{Du(E\backslash S_0)} dy =
\int_{E\backslash S_0} \det D^2 u(x) dx = \int_{E}\det D^2 u(x) dx.
\end{eqnarray*}
This shows that $Mu = (\det D^2 u) dx.$
\end{proof}
Some properties of normal mapping hold for general continuous functions. However, we will mostly restrict ourselves to convex functions on convex domains. One of the nice things about
this restriction is that if a hyperplane is locally below the graph of a convex function then it is also globally below. We record this locality property, whose proof
is geometrically obvious, in the following remark.
\begin{rem}
\label{locality}
 Let $u$ be a convex function on a convex set $\Omega\subset\R^n$. If $x_0\in\Omega, p\in\R^n$ and $u(y)\geq u(x_0) + p\cdot (y-x_0)$ for all $y$ in an open 
 set $\Omega'\subset\Omega$ containing
 $x_0$ then $u(y)\geq u(x_0) + p\cdot (y-x_0)$ for all $y$ in $\Omega$. 
 \end{rem}
The next lemma give a quantitative estimate for the boundedness of the set of normal mappings in the interior of the domain.
\begin{lem} [Estimates of the size of the slopes of supporting hyperplanes to a convex function]
\label{slope-est}
Let $\Omega\subset \R^n$ be a bounded convex set and $u$ a convex function in $\overline{\Omega}$. If $p\in\p u(x)$ where $x\in\Omega$, then
$$|p|\leq \frac{\max_{y\in\p\Omega} u (y) -u(x)}{\emph{dist} (x, \p\Omega)}.$$
\end{lem}
\begin{proof} For $r:=\text{dist} (x,\p\Omega)$, and $\e>0$, we have
$y_0 = x + r \frac{p}{|p|+\e}\in\Omega$.
Using convexity and 
$u(y) \geq u(x) + p\cdot(y-x)~\text{for all}~ y\in \Omega,$
we find that 
$$\max_{y\in\p\Omega} u(y)\geq u(y_0) \geq u(x) + r \frac{|p|^2}{|p|+\e}.$$
Thus
$$|p|-\e\leq \frac{|p|^2}{|p|+\e}\leq \frac{\max_{y\in\p\Omega} u (y) -u(x)}{\text{dist} (x,\p\Omega)}.$$
Letting $\e\rightarrow 0$, we obtain the desired estimate for $p$.
\end{proof}
Lemma \ref{slope-est}, though simple, implies the following properties of convex functions and their normal mappings:
 \begin{lem} \label{Rade_lem} Let $u$ be a convex function in $\Omega\subset\R^n$. Then:
 \begin{myindentpar}{1cm}
  (i) If $E\subset\subset\Omega$ then $\p u(\bar E)$ is a compact set and $u$ is uniformly Lipschitz in $E$, that is, there is a constant $C= C(u, \text{dist}(\bar E,\p\Omega))$ such that
  $$|u(x)-u(y)|\leq C|x-y|~\text{for all~} x, y\in E.$$
  (ii) $u$ is differentiable a.e. in $\Omega$. 
 \end{myindentpar}
\end{lem}
\begin{proof} (i) Using the compactness of $\bar E$ and the continuity of $u$, we can show that $\p u(\bar E)$ is a closed set in $\R^n$. By Lemma \ref{slope-est},
$\p u(\bar E)$ is a bounded set. Hence $\p u(\bar E)$ is a compact set.

To prove that $u$ is  uniformly Lipschitz in $E$, we just note that since $u$ is convex, its graph has a supporting hyperplane at any $(x, u(x))$ where $x\in E$ and hence $\p u(x)\neq\emptyset$.
If $p\in \p u(x)$ then by Lemma \ref{slope-est}, we have
$|p|\leq C(u, \text{dist}(\bar E,\p\Omega))$.
Therefore,
for all $y\in E$, we have
$$u(y)\geq u(x) + p\cdot(y-x)\geq u(x)-C|y-x|.$$
Reversing the role of $x$ and $y$, we obtain the desired Lipschitz estimate.

(ii) The conclusion follows from (i) and Rademacher's theorem which says that
  a Lipschitz continuous function on $\R^n$ is differentiable a.e.
\end{proof}
We next record the following relationship between normal mapping and the Legendre transform.

\begin{lem}[Normal mapping and the Legendre transform]\label{pLeg}
Let $u^{\ast}$ be the Legendre transform of $u$ as defined in Definition \ref{Leg_defn}.
If $p\in \p u(x)$ then $x\in \p u^{\ast}(p)$. 
\end{lem}
\begin{proof}Since $p\in\p u(x)$, we have for all $y\in\Omega$, 
$u(y)\geq u(x) + p\cdot(y-x)$, and hence $p\cdot x-u(x)\geq p\cdot y- u(y)$. From the definition of $u^{\ast}$, we find $u^{\ast}(p)\leq p\cdot x-u(x)$. 
It follows that, for any $z\in\R^n$, we also have
from the definition of $u^\ast$ that
$u^\ast (z)\geq x\cdot z- u(x) \geq u^{\ast}(p) + x\cdot (z-p)$. Therefore, $x\in\p u^{\ast}(p).$
\end{proof}
A simple consequence of Lemma \ref{pLeg} is the following:
\begin{lem}
 \label{2touch}
 Let $\Omega\subset\R^n$ be an open, bounded set and $u$ be a continuous function on $\Omega$. Then the set of slopes of supporting hyperplanes that touch the graph of $u$
 at more than one point has Lebesgue measure zero.
\end{lem}
\begin{proof}
 Let  $p$ be the slope of a supporting hyperplane that touches the graph of $u$ at $(x, u(x))$ and $(y, u(y))$, that is $p\in \p u(x)\cap \p u (y)$, where $x\neq y$.
 By Lemma \ref{pLeg}, we have $x, y\in \p u^{\ast}(p)$. Thus, $u^{\ast}$ is not differentiable at $p$. By Lemma \ref{Rade_lem}, the set of such $p$ has Lebesgue measure zero.
\end{proof}
\begin{rem} Let $u$ be a convex function on $\Omega$.
 \label{strict_rem}
 Suppose that $p\in \p u(x_0)$ and the supporting hyperplane $l(x)= u(x_0) + p\cdot (x-x_0)$ touches the graph of $u$ at only $(x_0, u(x_0))$
 then $v(x):=u(x)-l(x) (\geq 0)$  is strictly convex at $x_0$, that is $v(x)>0$ for all $x\neq x_0$. Indeed, if $v(x_1)=0$ for some $x_1\neq x_0$ then $0\in \p v(x_1)$ which
 implies that $p\in \p u(x_1)$, a contradiction.
\end{rem}

\begin{proof}[Proof of Definition \ref{MAdef}]
 The main observation in the proof is the following fact, which is an easy consequence of Lemma \ref{2touch}: If $A$ and $B$ are disjoint subsets of $\Omega$ then $\p u(A)$ and $\p u(B)$ are also disjoint in the measure-theoretic sense,
 that is $|\p u(A)\cap \p u(B)|=0.$ 
 
 First we show $\Omega\subset\mathcal{S}$. This follows from writing $\Omega$ as a union of compact sets whereas by Lemma \ref{Rade_lem} (i), they belong to $\mathcal{S}$.

 Next, we show that if $E\in\mathcal{S}$  then $\Omega\backslash E\in \mathcal{S}$. Indeed, from
 $$\p u(\Omega\backslash E)= \left(\p u(\Omega) \backslash \p u(E) \right)\cup \left(\p u(\Omega\backslash E) \cap \p u(E)\right)$$
 and from the observation, we have $|\p u(\Omega\backslash E) \cap \p u(E)|=0$ and hence  $\Omega\backslash E\in \mathcal{S}$.
 
 Finally, we show that $Mu$ is $\sigma$-additive. This means that if $\{E_i\}_{i=1}^\infty$ is a sequence of disjoint sets in $\mathcal{S}$, then we must show that
 $$|\p u (\bigcup_{i=1}^\infty E_i)|= \sum_{i=1}^\infty|\p u(E_i)|.$$
 This easily follows from the identities
 $$\p u (\bigcup_{i=1}^\infty E_i)=\bigcup_{i=1}^\infty \p u ( E_i)
 =\p u (E_1)\cup \bigcup_{i=2}^{\infty}  \left(\p u(E_i)\setminus \bigcup_{k=1}^{i-1} \p u(E_k)\right)
  $$
 and the fact that $\{\p u(E_i)\}_{i=1}^{\infty}$ are disjoint in measure as observed above.
\end{proof}
A very basic fact of the Monge-Amp\`ere measure is its weak continuity property stated as follows.
\begin{lem}[Weak continuity of Monge-Amp\`ere measure] \label{weakMA} Let $\{u_k\}$ be a sequence of convex functions on $\Omega$ which converges to $u$ uniformly on compact subsets
of $\Omega$. 
Then $Mu_k$ converges weakly to $Mu$, that is, for all continuous functions $f$ with compact support in $\Omega$, we have
$$\lim_{k\rightarrow \infty}\int_{\Omega} f(x) dMu_k(x) = \int_{\Omega} f(x) dMu(x).$$
\end{lem}
The Monge-Amp\`ere equation is well-behaved under rescaling using affine transformation.
\begin{lem}[Rescaling the Monge-Amp\`ere equation]
\label{rescaling_lem}
 Suppose that $u$ is a convex function on a convex domain $\Omega\subset\R^n$. Let $Tx = Ax + b$ where $A$ is an invertible $n\times n$ matrix and
 $b\in\R^n$. Define the following function on $T^{-1}\Omega$:
 $v(x)= u(Tx).$
 If $u$ is smooth then 
  $$\det D^2 v(x)= (\det A)^2\det D^2 u(Tx).$$
 If, in the sense of Aleksandrov,
 $$\lambda\leq \det D^2 u\leq \Lambda,$$
 then we also have, in the sense of Aleksandrov,
 $$\lambda|\det A|^2\leq\det D^2 v \leq \Lambda|\det A|^2.$$
\end{lem}

\begin{proof}[Proof of Lemma \ref{weakMA}]
To prove the lemma, we only need to verify the following:
\begin{myindentpar}{1cm}
 (i) If $K\subset\Omega$ is a compact set then
 $$\limsup_{k\rightarrow\infty} |\p u_k (K)|\leq |\p u(K)|.$$
 (ii) If $K$ is a compact set and $U$ is open such that $K\subset U\subset\subset \Omega$ then
 $$ |\p u(K)|\leq \liminf_{k\rightarrow\infty} |\p u_k(U)|.$$
\end{myindentpar}
The proof of these inequalities uses Fatou's lemma together with Lemma \ref{2touch},
and the following inclusions:
\begin{myindentpar}{1cm}
 (a) $$\limsup_{k\rightarrow\infty}\p u_k (K):=\bigcap_{i=1}^{\infty}\bigcup_{k=i}^{\infty}  \p u_k(K) \subset\p u(K) $$
 (b) $$\p u(K)\backslash S\subset \liminf_{k\rightarrow\infty} \p u_k(U):= \bigcup_{i=1}^{\infty}\bigcap_{k=i}^{\infty}  \p u_k(U),$$
 where
 $$S=\{p\in\R^n|\text{there are } x\neq y\in\Omega~\text{such that}~ p\in \p u(x)\cap \p u(y) \}.$$
\end{myindentpar}

To prove (a), let $p\in \limsup_{k\rightarrow\infty}\p u_k (K)$. Then for each $i$, there is $k_i$ and $x_{k_i}\in K$ such that $p\in \p u_{k_i}(x_{k_i})$. 
Since $K$ is compact, extracting a subsequence, still
labeled $x_{k_i}$, we have $x_{k_i}\rightarrow x_0\in K$. Thus, using 
$u_{k_i}(x)\geq u_{k_i}(x_{k_i}) + p\cdot(x- x_{k_i})~\text{for all}~x\in\Omega$,
and the uniform convergence of $u_{k_i}$ to $u$ on compact subsets of $\Omega$, we obtain
$$u(x)\geq u(x_0) + p\cdot(x-x_0)~\text{for all}~x\in\Omega$$ and therefore, $p\in \p u(x_0)\subset \p u(K)$.

To prove (b), let $p\in \p u(x_0)\subset \p u(K)\backslash S$. Then, by Remark \ref{strict_rem}, $u(x)- l(x)(\geq 0)$ where $$l(x)= u(x_0) + p\cdot (x-x_0),$$ is strictly convex at $x_0$.
By subtracting $l(x)$ from $u_k$ and $u$, it suffices to show that $0\in \p u_k (x_k)$ for all $k$ large and some $x_k\in U$. Recalling Remark \ref{locality}, 
we prove this by choosing a minimum point
$x_k$ of the continuous function
$u_k$ in the compact set $\bar U$. It remains to show that $x_k\not\in \p U$ when $k$ is large. This is easy. Indeed, from the strict convexity of $u$ at $x_0\in K\subset U$, we 
can find some $\delta>0$ such that
$u(x)\geq \delta $ on $\p U$. Hence, from the uniform convergence of $u_k$ to $u$ on compact sets, we find that $u_k\geq \delta/2$ on $\p U$ if $k$ is large. On the other hand, since
$u(x_0)=0$, we also find that $u_k(x_0)\leq \delta/4$ when $k$ is large. Therefore, $x_k\not\in \p U$ when $k$ is large.
\end{proof}

\begin{proof} [Proof of Lemma \ref{rescaling_lem}] The proof is simple in the case $u$ is smooth. 
In this case, we have 
$$Dv(x) = A^{t} Du (Tx)~\text{and}~D^2 v(x) = A^t (D^2 u(Tx))A.$$
Therefore,
$$\det D^2 v(x) =(\det A)^2 \det D^2 u(Tx).$$
For a general convex function $u$, we use the normal mapping. Note that, for any $x\in T^{-1}(\Omega)$, we get by using
the definition of the normal mapping in Definition \ref{pdefn}
\begin{equation}\p v(x) = A^{t}\p u(Tx).
 \label{uvD}
\end{equation}
Now, let $E\subset T^{-1}(\Omega)$ be a Borel set. Then, by (\ref{uvD})
\begin{eqnarray*}Mv(E)=|\p v(E)|=|A^t \p u(T(E))|=|\det A^t| |Mu (T(E))|=|\det A| |Mu (T(E))|.
  \end{eqnarray*}
Because
$\lambda\leq \det D^2 u\leq \Lambda$
 in the sense of Aleksandrov,
we have
$$\lambda |\det A||E|=\lambda|T(E)|\leq |Mu (T(E))| \leq \Lambda |T(E)|=\Lambda |\det A||E|.$$
The conclusion of the lemma then follows from
$$\lambda |\det A|^2|E|\leq Mv(E) \leq \Lambda |\det A|^2|E|.$$
\end{proof}
\subsection{Maximum principles}
\label{mp1-sec}
The following basic maximum principle for convex functions roughly says that if two convex functions defined on the same domain and having the same boundary values, 
the one below the other will 
have larger total opening. 
\begin{lem}[Maximum principle]\label{mp1}
Let $\Omega\subset \R^n$ be a bounded open set and $u, v\in C(\overline{\Omega})$. If $u=v$ on $\p\Omega$ and $v\geq u$ in $\Omega$ then
$$\p v(\Omega)\subset \p u(\Omega).$$
\end{lem}
In the above maximum principle, no convexity on the functions nor the domain is assumed.
Lemma \ref{mp1} is a corollary of the following pointwise maximum principle:
\begin{lem}[Pointwise Maximum Principle]\label{mp_pw}
Let $\Omega\subset \R^n$ be a bounded open set and $u, v\in C(\overline{\Omega})$. If $u\geq v$ on $\p\Omega$ and $v(x_0)\geq u(x_0)$ where $x_0\in \Omega$ then
$\p v(x_0)\subset \p u(\Omega).$
\end{lem}
\begin{proof}[Proof of Lemma \ref{mp_pw}]
Let $p\in \p v(x_0)$. Then, $p$ is the slope of the tangent hyperplane to the graph of $v$ at $(x_0, v(x_0))$, that is
\begin{equation}v(x) \geq v(x_0) + p\cdot(x-x_0)~\text{for all } x\in\overline{\Omega}.
 \label{pdef}
\end{equation}
We will slide down this hyperplane to obtain a tangent hyperplane for the graph of $u$. Let
$$a= \sup_{x\in\Omega}\{v(x_0) + p\cdot(x-x_0)- u(x)\}.$$
This is the amount that we will slide down. We now claim that $$l(x):=v(x_0) + p\cdot(x-x_0)-a$$ is a supporting hyperplane to the graph of $u$ at some point $(z, u(z))$ where 
$z\in \Omega$. Indeed, since $v(x_0)\geq u(x_0)$, we have $a\geq 0$; moreover, if $v(x_0)>u(x_0)$ then $a>0$.
Let $x_1\in\overline{\Omega}$ be such that
$$a=v(x_0) + p\cdot(x_1-x_0)- u(x_1) \geq v(x_0) + p\cdot(x-x_0)- u(x)~\text{for all } x\in\Omega.$$
It follows that
\begin{equation}u(x) \geq p\cdot(x-x_1) + u(x_1)= v(x_0) + p\cdot(x-x_0)-a:= l(x).
 \label{slide_l}
\end{equation}

If $a>0$ then $x_1\not\in \p\Omega$ because otherwise, we have $u(x_1)\geq v(x_1)$ and $a\leq v(x_0) + p(x_1- x_0)- v(x_1)\leq 0$, a contradiction to the definition of $p$
in (\ref{pdef}). In this case, we deduce from 
(\ref{slide_l}) that $l(x)$ is a supporting hyperplane to the graph of $u$ at $(x_1, u(x_1))$.

If $a=0$ then there is no sliding down. In this case, $u(x_0)= v(x_0)$. Hence, by
(\ref{slide_l}), $l(x)$ is a supporting hyperplane to the graph of $u$ at $(x_0, u(x_0))$.
\end{proof}

The next theorem, due to Aleksandrov \cite{Alek68}, is of fundamental importance in the theory of Monge-Amp\`ere equations. It says that a convex function having bounded 
Monge-Amp\`ere measure, or more generally, having finite total Monge-Amp\`ere measure, can only drop its value when it steps into the domain.
\begin{thm} [Aleksandrov's maximum principle]
If $\Omega\subset \R^n$ is a bounded, open and convex set with diameter $D$, and $u\in C(\overline{\Omega})$ is a convex function with $u=0$ on $\p\Omega$, then
$$|u(x_0)|^{n}\leq C_n D^{n-1} \emph{dist} (x_0, \p\Omega)|\p u(\Omega)|$$
for all $x_0\in\Omega$ where $C_n$ is a constant depending only on the dimension $n$.
\label{Alekmp}
\end{thm}
\begin{proof}[Proof of Theorem \ref{Alekmp}]
Let $v$ be the convex function whose graph is the cone with vertex $(x_0, u(x_0))$ and the base $\Omega$, with $v=0$ on $\p\Omega$. Since
$u$ is convex, $v\geq u$ in $\Omega$. By the maximum principle in Lemma \ref{mp1}, 
$\p v(\Omega)\subset \p u(\Omega).$
The proof is based on the following observations:
\begin{myindentpar}{1cm}
(1) $\p v(\Omega)=  \p v(x_0)$ and thus $\p v(\Omega)$ is convex. \\
 (2) $\p v(\Omega)$ contains $B_{\frac{|u(x_0)|}{D}}(0)$.\\
 (3) There is $p_0\in \p v(\Omega)$ such that
$|p_0|= -\frac{u(x_0)}{\text{dist} (x_0, \p\Omega)}.$
\end{myindentpar}
Assuming (1)-(3), we see that $\p v(\Omega)$ contains the convex hull of  $B_{\frac{|u(x_0)|}{D}}(0)$ and $p_0$. This convex hull has measure at least
$$\max\left\{\omega_n \left(\frac{|u(x_0)|}{D}\right)^{n}, \frac{ \omega_{n-1}}{n} \left(\frac{|u(x_0)|}{D}\right)^{n-1} |p_0|\right\}.$$
Since $ |\p v(\Omega)| \leq  |\p u(\Omega)|$, the conclusion of the theorem now follow from
$$|u(x_0)|^{n}\leq \min\left\{C_n D^{n-1} \text{dist} (x_0, \p\Omega)|\p u(\Omega)|,~ \omega_n^{-1} D^n |\p u(\Omega)|\right\}~
\text{with}~
C_n= \frac{n}{\omega_{n-1}}.$$
Let us now verify (1)-(3). To see (1), we note that if $p\in\p v(\Omega)$ then there is $x_1\in\Omega$ such that $p= \p v(x_1)$. It 
suffices to consider the case $x_1\neq x_0.$ Since the graph of $v$ is a cone, $v(x_1) + p\cdot(x-x_1)$ is a supporting hyperplane to the graph of $v$ at $(x_0, v(x_0))$, that is $p\in \p v(x_0).$

For (2) and (3), we note that, since the graph of $v$ is a cone with vertex $(x_0, v(x_0))=(x_0, u(x_0))$ and the base $\Omega$, $p\in \p v(x_0)$
if and only if $v(x)\geq v(x_0) + p\cdot(x-x_0)$ for all $x\in\p\Omega$. Thus (2) is straightforward. 

To obtain (3), 
take $x_1\in\p\Omega$ such that $|x_1- x_0| = \text{dist} (x_0, \p\Omega)$. Then $p_0= -u(x_0)\frac{x_1-x_0}{|x_1-x_0|^2}$ is the desired slope. Indeed, for any $x\in\p\Omega$, $
(x-x_0)\cdot \frac{x_1-x_0}{|x_1-x_0|}$ is the vector projection of $x-x_0$ onto the ray from $x_0$ to $x_1$. 
Using the convexity of $\Omega$, we find
$$(x-x_0) \cdot\frac{x_1-x_0}{|x_1-x_0|} \leq |x_1-x_0|$$
and hence, from the formula for $p_0$, we find that for all $x\in\p\Omega$
$$0= v(x)= u(x_0) + |p_0| |x_1-x_0|\geq v(x_0) + p_0 \cdot(x-x_0).$$

\end{proof}
 \begin{lem}\label{ABPmax2}
 For $u\in C^{2}(\Omega)\cap C^{0}(\overline{\Omega})$, we have
 $$\sup_{\Omega} u\leq \sup_{\partial\Omega} u + \frac{\emph{diam}(\Omega)}{\omega_{n}^{1/n}}\left(\int_{\Gamma^{+}}|\det D^2 u|\right)^{1/n}$$
 where $\Gamma^{+}$ is the upper contact set $$\Gamma^{+}=\{y\in\Omega| u(x) \leq u(y) + p\cdot (x-y)~\text{for all}~x\in\Omega,~\text{for some}~ p= p(y)\in \R^{n}\}
 .$$
\end{lem}
\begin{proof} [Proof of Lemma \ref{ABPmax2}] The proof is similar to that of Theorem \ref{Alekmp} but for the sake of completeness, we include it here. Let $D=\text{diam}(\Omega)$.
By considering $\hat u:= - (u-\sup_{\p\Omega} u)$ instead of $u$, we need to show that an equivalent statement: 
If $u\in C^{2}(\Omega)\cap C^{0}(\overline{\Omega})$ with $\inf_{\p\Omega} u=0$ then
 \begin{equation}-\inf_{\Omega} u \leq  \frac{D}{\omega_{n}^{1/n}}\left(\int_{\mathcal{C}}|\det D^2 u|\right)^{1/n}
  \label{abpeq1}
 \end{equation}
where $\mathcal{C}$ is the lower contact set $$\mathcal{C}=\{y\in\Omega| u(x) \geq u(y) + p\cdot (x-y)~\text{for all}~x\in\Omega,~\text{for some}~ p= p(y)\in \R^{n}\}.$$
 Clearly, $y\in\mathcal{C}$ if and only if $\p u(y)\neq\emptyset$. As seen before, since $u\in C^2(\Omega)$, we have $D^2 u(y)\geq 0$ when $y\in\mathcal{C}$. We have then
 \begin{equation}|\p u(\Omega)| = |\p u(\mathcal{C})|= \int_{\mathcal{C}}|\det D^2 u|.
  \label{abpeq2}
 \end{equation}
The last equality follows from the proof of Example \ref{examC2} where the set $A$ there only needs to be modified to $A=\{x\in\mathcal{C}: D^2 u(x)>0\}$ and $S_0=\mathcal{C}\backslash
A$.

From $\inf_{\p\Omega} u=0$, it suffices to consider the case where the minimum of $u$ on $\overline{\Omega}$ is attained at $x_0\in \Omega$ with $u(x_0)<0$.
Let $v$ be the convex function whose graph is the cone with vertex $(x_0, u(x_0))$ and the base $\Omega$, with $v=0$ on $\p\Omega$. By Lemma \ref{mp_pw},
\begin{equation}\p v(\Omega)= \p v(x_0)\subset \p u(\Omega).
 \label{abpeq3}
\end{equation}
Moreover, $$\p v(\Omega)=\p v(x_0)\supset B_{\frac{|u(x_0)|}{D}}(0).$$
Hence, from (\ref{abpeq3}) and (\ref{abpeq2}), we have
$$\omega_n \left(\frac{|u(x_0)|}{D}\right)^n\leq |\p v(\Omega)|\leq |\p u(\Omega)|= \int_{\mathcal{C}}|\det D^2 u|.$$
Thus (\ref{abpeq1}) is proved.
 \end{proof}
We now give the proof of the ABP maximum principle stated in Theorem \ref{ABPmax}.
\begin{proof}[Proof of Theorem \ref{ABPmax}] Since $u\in C^2(\Omega)$, on the upper contact set $\Gamma^{+}$, we have $D^2 u\leq 0$.
 Using Lemma \ref{trlem} for $A=-D^2 u$, $B= (a^{ij})$, we have on $\Gamma^{+}$
 $$|\det D^2 u| =\det(-D^2 u)\leq \frac{1}{\det (a^{ij})} \left(\frac{-a^{ij} u_{ij}}{n}\right)^n.$$
 Hence
 $$\left(\int_{\Gamma^{+}}|\det D^2 u|\right)^{1/n} \leq \frac{1}{n}\left\|\frac{a^{ij} u_{ij}}{(\det (a_{ij})^{1/n})}\right\|_{L^{n}(\Gamma^{+})}$$
 Now, applying Lemma \ref{ABPmax2}, we obtain the ABP estimate in Theorem \ref{ABPmax}.
\end{proof}

The conclusion of Theorem \ref{Alekmp} raises the following question:
\begin{quest}Will a convex function drop its value when stepping inside the domain?
\label{dropquest}
\end{quest}
Clearly, without a lower bound on the Monge-Amp\`ere measure $Mu$, the answer is in 
the negative as can be seen from the constant $0$. We will prove in Theorem \ref{drop-thm} that the above 
question has a positive answer when the density of the Monge-Amp\`ere measure $Mu$ has a positive lower bound.
\subsection{John's lemma}
In this section, we will prove a crucial result, due to John \cite{John}, in the investigation of the Monge-Amp\`ere equations. It says that all convex 
bodies are equivalent to balls modulo affine transformations.
\begin{lem}[John's Lemma]\label{John_lem}
Let $K\subset \R^{n}$ be a convex body (that is, a compact, convex set with nonempty interior). Then there is an ellipsoid (the image of the unit ball $B_1(0)$ in $\R^n$ under a positive definite affine map) so that if $c$ is the center of $E$ then
$$E\subset K \subset c + n (E-c),$$
where $$c + n(E-c)=\{c + n(x-c), x\in E\}.$$
\end{lem}
Lemma \ref{John_lem} motivates the following definition.
\begin{defn}[Normalized convex set]
 If $\Omega\subset\R^n$ is a convex body then there is an affine transformation $T$ such that
 $$B_1(0)\subset T(\Omega)\subset B_n(0).$$
 We say that $T$ {\it normalizes} $\Omega$.
  A convex set $\Omega\subset\R^n$ is called {\it normalized} if $B_1(0)\subset \Omega\subset B_n(0).$
\end{defn}
\begin{proof}[Proof of Lemma \ref{John_lem}] 
We first show that $K$ contains an ellipsoid of maximal volume. Let $N:= n^2 + n.$ The set of ordered pairs $(A, b)$ where $A$ is an $n\times n$ matrix and $b$ is a vector in $\R^n$ is identified as $\R^N$. Let
$$\mathcal{E}:= \{(A, b)\in \R^N: AB_1(0) + b\subset K\}.$$
Then $\mathcal{E}$ is a non-empty, closed, and bounded. Thus it is a compact set of $\R^N$. The map 
$$(A, b)\rightarrow |AB_1(0) + b|\equiv \omega_n|\det A|$$
is a continuous function on $\mathcal{E}$. Thus there is an $(A_0, b_0)\in \mathcal{E}$ that maximizes this function on $\mathcal{E}$. Then $E:= A_0 B_1(0) + b_0$ is 
the desired ellipsoid.

Now, replacing $K$ by $A_{0}^{-1}(K-b_0)$ if necessary, we can assume 
that $B_1(0)= A_0^{-1}(E-b_0)$ is an ellipsoid of maximal volume in $K$. Hence, to prove the lemma, it suffices to show that if $p\in K$ then $|p|\leq n.$

Assume that there is a point $p\in K$ with $a:=|p|>n$. Choose orthogonal coordinates $(x_1, x_2,\cdots, x_n)$ on $\R^n$ such that 
$p= (a, 0,\cdots, 0)$. Consider the following affine map $\Psi^{\lambda}_t$
$$(x_1, x_2, \cdots, x_n)\mapsto (-1 + e^t (x_1 + 1), e^{-\lambda t}x_2, \cdots, e^{-\lambda t} x_n).$$
We claim that for small $t$, $$\Psi_t^{\lambda}(B_1(0))\subset C_a^n,$$ provided that
$\lambda>\frac{1}{a-1}$ 
where $ C_a^n\subset K$ is the convex hull of $B_1(0)\subset\R^n$ and $p$. 

Granted the claim, then, the volume of the ellipsoid $\Psi_t^{\lambda}(B_1(0))$ is
$$|\Psi_t^{\lambda}(B_1(0))|= e^{(1-(n-1)\lambda)t} |B_1(0)|> |B_1(0)|$$
provided that $t>0$ small and 
$\lambda <\frac{1}{n-1}.$

Therefore, when $a>n$, we can choose $\lambda$ satisfying $\frac{1}{n-1}>\lambda>\frac{1}{a-1}.$ This choice of $\lambda$ contradicts the maximality of the volume of $B_1(0)$.

We now prove the claim. By symmetry, it suffices to consider $n=2$
and we need to show that, under the map $\Psi_t^{\lambda}$, the "top" point $(z_1, z_2)=(\frac{1}{a}, \frac{\sqrt{a^2-1}}{a})$ is below the tangent line $l$ from $(a, 0)$ to 
$B_1(0)\subset\R^2$ at this point.
The equation for the tangent line is
$$\frac{x_1-a}{\sqrt{a^2-1}} + x_2=0.$$
The criterion for $(x_1, x_2)$ being below $l$ is that
$\frac{x_1-a}{\sqrt{a^2-1}} + x_2<0.$
We show that the above inequality holds for $\Psi^{\lambda}_t(z_1, z_2)\equiv (-1 + e^t (z_1 + 1), e^{-\lambda t} z_2)$ by evaluating
$$f(t) = \frac{e^t (z_1+ 1) -1-a}{\sqrt{a^2-1}} + e^{-\lambda t} z_2.$$
We have
$$f^{'}(0) = \frac{z_1+1}{\sqrt{a^2-1}} -\lambda z_2 = \frac{1+a}{a\sqrt{a^2-1}} -\lambda \frac{\sqrt{a^2-1}}{a}<0$$
provided that
$\lambda > \frac{1}{a-1}.$
Thus $f(t) < f(0) = 0$ for small positive values of $t$.
\end{proof}

\subsection{Comparison principle and applications}
\label{comp-sec}
We start this section with a converse of Lemma \ref{mp1}. It implies that if two convex functions having the same boundary values, the function with larger opening at all scales 
(that is, with larger Monge-Amp\`ere measure) is in fact smaller than the other one in the interior.
\begin{lem}[Comparison principle]\label{comp-prin}
Let $u, v\in C(\overline{\Omega})$ be convex functions such that 
$$\det D^2 v \geq \det D^2 u$$
in the sense of Aleksandrov, that is, 
$|\p u(E)|\leq |\p v(E)|$
for all Borel set $E\subset \Omega$. Then
$$\min_{x\in\overline{\Omega}} (u(x)- v(x)) =\min_{x\in\p\Omega} (u(x)- v(x)).$$
In particular, if $u\geq v$ on $\p\Omega$ then $u\geq v$ in $\Omega$.
\end{lem}
\begin{proof} By adding a constant to $v$, we can assume that $\min_{x\in\p\Omega} (u(x)- v(x))=0$ and hence $u\geq v$ on $\p\Omega$. We need to show that $u\geq v$ in $\Omega$.
Arguing by contradiction, we suppose $u-v$ attains its minimum at $x_0\in\Omega$ with
$u(x_0)- v(x_0)=-M<0.$
Choose $\e>0$ small such that
$\varepsilon (\text{diam } \Omega)^2 <M/2.$
Let us consider
$$w(x) = u(x) - v(x) - \e |x-x_0|^2.$$
If $x\in\p\Omega$ then $w(x)\geq -\e (\text{diam }\Omega)^2 \geq -\frac{M}{2}$ while at $x_0$,
$w(x_0) =-M <-M/2.$
Thus, $w$ attains its minimum value at $a\in\Omega$.

The choice $w$ comes from simple investigation in the case $u$ and $v$ are smooth. 
In this case $\det D^2 u(x) \leq \det D^2 v(x)$ for all $x\in \Omega$ and we try to use a more quantitative version of the following facts:
$$Du(x_0) = Dv(x_0), \det D^2 u(x_0)\geq \det D^2 v(x_0).$$
This almost gives a contradiction. The 
 minimum point $a$ of $w$ actually helps us do this. In fact, we have $D^2 w(a)\geq 0$
 and hence
$D^2 u(a) \geq D^2 v(a) + 2\e I_n.$
Therefore, by Lemma \ref{concavelem}, we have
$$\det D^2 u(a) \geq \det (D^2 v(a) +2\varepsilon I_n)\geq \left[(\det D^2 v(a))^{1/n} + (\det (2\e I_n))^{1/n}\right]^n>\det D^2 v(a),$$
thus obtaining a contradiction to the assumption $\det D^2 u(x) \leq \det D^2 v(x)$ for all $x\in \Omega$.

Let us now derive a contradiction for general convex functions $u$ and $v$. Let
$$E=\{x\in\bar \Omega: w<-\frac{3}{4}M\}.$$
Then $E$ is open, nonempty because $a\in E$ and, furthermore $$\p E=\{x\in\bar \Omega: w=-\frac{3}{4}M\}\subset\Omega.$$ 
The function $u$ is below $v+ \e |x-x_0|^2-\frac{3}{4}M $ in $E$ and they coincide on $\p E$. 
By applying the maximum principle in Lemma \ref{mp1}
to these functions in $E$, we get
$$\p u(E)\supset \p (v+ \e |x-x_0|^2-\frac{3}{4}M) (E)= \p (v+ \e |x-x_0|^2) (E).$$
It follows that
\begin{equation}|\p u(E)|\geq |\p (v+ \e |x-x_0|^2) (E)|.
 \label{uvE}
\end{equation}
We claim that
\begin{equation}|\p (v+ \e |x-x_0|^2) (E)|\geq |\p v(E)| + |\p \e |x-x_0|^2) (E)|= |\p v(E)| + (2\e)^n |E|> |\p v(E)|
 \label{MA_add}
\end{equation}
and thus by (\ref{uvE}) obtaining a contradiction to the hypothesis $|\p u(E)|\leq |\p v(E)|$. 

It remains to prove the claim. If $v\in C^2(\Omega)$ then, by Lemma \ref{concavelem},  we have
\begin{eqnarray*}|\p (v+ \e |x-x_0|^2) (E)|& =& \int_{E} \det (D^2 v + 2\e I_n)
\\ & \geq& \int_E (\det D^2 v + (2\e)^n) = |\p v(E)| + (2\e)^n |E|.
 \end{eqnarray*}
In general, we can approximate $v$ by convex $C^2$ functions $v^{\eta}$ ($\eta>0$) such that as $\eta\rightarrow 0$, $v^{\eta}$ converges uniformly to $v$ on compact subsets of $\Omega$. Thus can by done by 
setting $v^{\eta}= v\ast \phi_{\eta}$ where $\varphi_{\eta}$ is a standard mollifier, that is, $\phi_\eta$ is smooth with support in $B_{\eta}(0)$ and 
$\int_{\R^n}\phi_{\eta}=1$. The claim holds for $v^{\eta}$ and by letting $\eta
\rightarrow 0$, using Theorem \ref{weakMA}, we obtain the claim for $v$.
\end{proof}
Finally, we give a positive answer to Question \ref{dropquest}. The following theorem says that a convex function with a positive lower bound on its opening at all scale will drop its values when stepping inside the domain.
\begin{thm}\label{drop-thm}
If $B_1(0)\subset \Omega\subset B_n(0)$ where $\Omega$ is a convex set, and if $u$ is a convex function on $\Omega$ with $u=0$ on $\p\Omega$ and
$$\lambda \leq \det D^2 u \leq \Lambda,$$
then
$$c(\lambda, n) \leq |\min_{\Omega} u| \leq C(\Lambda, n).$$
\end{thm}
\begin{proof}
The inequality
$ |\min_{\Omega} u| \leq C(\Lambda, n)$
follows from Aleksandrov maximum principle, Theorem \ref{Alekmp}. However, we give here a proof of the theorem using the comparison principle in Lemma \ref{comp-prin}.
Consider
$$v(y)=\frac{\lambda^{1/n}}{2} (|y|^2 - 1).$$ Then, since $B_1(0)\subset \Omega$, $$v\geq 0=u~ \text{on} ~\p\Omega~ \text{and}~ \det D^2 v =\lambda\leq \det D^2 u~ \text{in}~ \Omega.$$ Thus, by Lemma \ref{comp-prin}, we have
$u\leq v$ in $\Omega$. Similarly, we have
$$\frac{\Lambda^{1/n}}{2} (|y|^2 - n^2) \leq u(y) \leq \frac{\lambda^{1/n}}{2} (|y|^2 - 1).$$
It follows that
$$ -\frac{\Lambda^{1/n}n^2}{2}  \leq \min_{\Omega} u \leq u(0) \leq -\frac{\lambda^{1/n}}{2}, $$
completing the proof of the theorem.
\end{proof}
\subsection{The Dirichlet problem and Perron's method}
The main result of this section is the solvability of the nonhomogeneous Dirichlet problem for the Monge-Amp\`ere equation with continuous boundary data. We essentially follow the presentation in Rauch-Taylor
 \cite{RT}.
\begin{thm}
\label{muthm}
 Let $\Omega\subset\R^n$ be an open, bounded and strictly convex domain. Let $\mu$ be a Borel measure in $\Omega$ with $\mu(\Omega)<\infty$. 
 Then for any $g\in C(\p\Omega)$, the problem
 \begin{equation*}
 \left\{
 \begin{alignedat}{2}
   \det D^2 u~&=\mu \h~&&\text{in} ~\Omega, \\\
 u&= g \h~&&\text{on}~ \p\Omega,
 \end{alignedat}
 \right.
\end{equation*}
has a unique convex solution $u\in C(\overline\Omega)$ in the sense of Aleksandrov.
\end{thm}
We will use the Perron method \cite{Per} which was designed to solve the Dirichlet problem for the Laplace equation with continuous boundary data. Let us brieftly recall this powerful
method in solving the following problem:
\begin{prob}
Let $\Omega\subset\R^n$ be an open, bounded and smooth domain and  $\varphi\in C(\p\Omega)$.
Find a solution $u\in C(\overline{\Omega})$ solving
 \begin{equation*}
 \left\{
 \begin{alignedat}{2}
   \Delta u~&=0 \h~&&\text{in} ~\Omega, \\\
 u&= \varphi \h~&&\text{on}~ \p\Omega.
 \end{alignedat}
 \right.
\end{equation*}
\label{Lap_Dir}
\end{prob}
Relevant to Problem \ref{Lap_Dir} are the following sets of candidates of supersolutions and subsolutions:
\begin{myindentpar}{1cm}
 (i) Overshooting paths
 $$S^{\varphi}=\{u\in C(\overline{\Omega})\mid u\geq \varphi ~\text{on}~\p\Omega~\text{and}~\text{u is superharmonic, that is, } -\Delta u\geq 0~\text{in }\Omega\}.$$
 (ii) Undershooting paths
 $$S_{\varphi}=\{v\in C(\overline{\Omega})\mid v\leq \varphi ~\text{on}~\p\Omega~\text{and}~\text{v is subharmonic, that is, } -\Delta v\leq 0~\text{in }\Omega\}.$$
\end{myindentpar}
\begin{thm} (Perron \cite{Per})
\label{Perthm}
 The function
 $u(x)=\sup_{v\in S_{\varphi}} v(x)$
 is harmonic in $\Omega$.
\end{thm}
Key ingredients in the proof of Theorem \ref{Perthm} include:
\begin{myindentpar}{1cm}
 (a) The maximum principle for harmonic functions.\\
 (b) The solvability of the Dirichlet problem for $\Omega$ being any ball $B$ and $\varphi$ is any continuous function on $\p B$. This uses essentially the Poisson integral of $\varphi$.\\
 (c) The stability of the lifting of subharmonic and harmonic functions. More precisely, suppose $\bar u$ is harmonic in a ball $B\subset\subset \Omega$ and $u$ is subharmonic
 in $\Omega$. Define the lifting of $\bar u$ and $u$ by
 $$U(x) = \left\{\begin{array}{rl}
 \bar u (x) &  x\in B,\\
u (x)  & x\in  \Omega\setminus B.
\end{array}\right.$$
 Then $U$ is also a subharmonic function in $\Omega$.
\end{myindentpar}
Features of Perron's method:
\begin{myindentpar}{1cm}
 (F1) It separates the interior existence problem from that of the boundary behavior of solutions.\\
 (F2) It can be extended easily to more general classes of second order elliptic equations.
\end{myindentpar}
The main question regarding Perron's method is:
\begin{quest}
\label{Dquest}
 Does $u$ defined by Perron's theorem satisfy $u=\varphi$ on $\p\Omega$?
\end{quest}
The answer to Question \ref{Dquest} depends on local behavior of $\p\Omega$ near each boundary point $x_0\in\p\Omega$. But the answer is always YES if $\Omega$ is convex. This is 
based on the concept
of {\it barriers}; see \cite[Chapter 2]{GT} for more details.

Before proving Theorem \ref{muthm}, we consider a simpler theorem regarding the solvability of the homogeneous Dirichlet problem for the Monge-Amp\`ere equation with continuous boundary data.
\begin{thm}
\label{mu0thm}
Let $\Omega\subset\R^n$ be an open, bounded and strictly convex domain. Then for any $g\in C(\p\Omega)$, the problem
 \begin{equation*}
 \left\{
 \begin{alignedat}{2}
   \det D^2 u~&=0 \h~&&\text{in} ~\Omega, \\\
 u&= g \h~&&\text{on}~ \p\Omega,
 \end{alignedat}
 \right.
\end{equation*}
has a unique convex solution $u\in C(\overline\Omega)$ in the sense of Aleksandrov.
\end{thm}
\begin{proof}[Proof of Theorem \ref{mu0thm}] We first note that, if there is such a convex solution $u\in C(\overline{\Omega})$, then it is unique by the comparison principle (Lemma 
\ref{comp-prin}). Our main task now is to show the 
existence. Heuristically, we look at the supremum of subsolutions which are convex functions $u\in C(\overline{\Omega})$ satisfying
$-\det D^2 u(x)\leq 0~\text{in}~\Omega~\text{and}~u\leq g~\text{on}~\p\Omega.$
By we can simply try affine functions. 
Let $$\mathcal{F}=\{a(x): \text{a is an affine function and } a\leq g~\text{on}~\p\Omega\}.$$
Since $g$ is continuous, $\mathcal{F}\neq\emptyset$ because
$a(x)\equiv \min_{y\in\p\Omega}g(y)\in\mathcal{F}.$\\
{\bf Claim:} 
\begin{equation*}
 u(x)=\sup_{a\in\mathcal{F}} a(x),~x\in\overline{\Omega}
\end{equation*}
is the unique desired solution.

Clearly, $u$ is convex and $u\leq g$ on $\p\Omega$.
The proof of the claim is proceeded in 3 steps.\\
{\bf Step 1}: $u=g$ on $\p\Omega$.\\
{\bf Step 2}: $u\in C(\overline{\Omega})$.\\
{\bf Step 3:} 
$$\p u(\Omega)\subset\{p\in\R^n: \text{there are}~x\neq y\in\Omega~\text{such that}~p\in\p u(x)\cap \p u(y)\}.$$
Assuming all these steps have been verified, we conclude the proof as follows. By Lemma \ref{2touch}, we have $|\p u(\Omega)|=0$ and hence $Mu=0$ in $\Omega$, or 
$\det D^2 u=0~ \text{in} ~\Omega$ in the sense of Aleksandrov. This completes the proof of Theorem \ref{mu0thm}.\\
{\bf Proof of Step 1}. Let $x_0\in\p\Omega$. We show $u(x_0)\geq g(x_0)$. Without loss of generality, we can assume that $x_0=0\in\p\Omega$ and that
$\Omega\subset \{x\in\R^n: x_n>0\}.$
From the continuity of $g$, given $\e>0$, there exists $\delta>0$ such that
\begin{equation}
 \label{gcts}
 |g(x)- g(0)| <\e~\text{for all~} x\in\p\Omega\cap B_{\delta}(0).
\end{equation}
Since $\Omega$ is strictly convex, there exists $\eta>0$ such that $x_n\geq \eta$ for all $x\in\p\Omega\setminus B_{\delta}(0)$.\\
{\bf Claim 1.} The following function belongs to $\mathcal{F}$: $a(x)= g(0)-\e- C_1 x_n\in \mathcal{F}$ where
\begin{equation}C_1 = \frac{2\|g\|_{L^{\infty}(\p\Omega)}}{\eta}.
 \label{C1eq}
\end{equation}
The proof of {\bf Claim 1} is quite elementary. If $x\in \p\Omega\cap B_{\delta}(0)$ then by (\ref{gcts}), we have $g(x)> g(0)-\e\geq a(x)$. If 
$x\in\p\Omega\setminus B_{\delta}(0)$ then $x_n\geq\eta$ and hence 
$$a(x)\leq g(0)-\e -C_1\eta <g(0)-2\|g\|_{L^{\infty}(\p\Omega)} \leq g(x).$$
From {\bf Claim 1}, and the definition of $u$, we have
$u(0)\geq a(0)= g(0)-\e$. This holds for all $\e>0$ so $u(0)\geq g(0)$. {\bf Step 1} is proved.\\
{\bf Proof of Step 2}. We note that the proof of {\bf Step 2} in 
Rauch-Taylor \cite{RT} relies on
\begin{myindentpar}{1cm}
 (a) The maximum principle for harmonic functions.\\
 (b) The solvability of the Dirichlet problem for the Laplace equation: There is a unique solution $w\in C(\overline{\Omega})$ solving the equation
 \begin{equation*}
 \left\{
 \begin{alignedat}{2}
   \Delta u~&=0 \h~&&\text{in} ~\Omega, \\\
 u&= g \h~&&\text{on}~ \p\Omega.
 \end{alignedat}
 \right.
\end{equation*}
\end{myindentpar}
We present here a proof without using (a) nor (b).

Since $u$ is convex in $\Omega$, it is continuous there. It remains to prove that $u$ is continuous on $\p\Omega$. Let us assume that $x_0=0\in\p\Omega$
and also $\Omega$ is as in {\bf Step 1.} Let $\{y_k\}_{k=1}^{\infty}\subset\overline{\Omega}$ be such that $y_k\rightarrow 0$. We show that $u(y_k)\rightarrow u(0)=g(0)$. Let $a$ be as in 
{\bf Step 1.} Then $u(x)\geq a(x)$. Hence
$u(y_k)\geq a(y_k)$. Thus, for all $\e>0$, we have
\begin{equation*}
 \liminf_{k\rightarrow\infty} u(y_k)\geq  \liminf_{k\rightarrow\infty} a(y_k) = g(0)-\e.
\end{equation*}
It follows that
$ \liminf_{k\rightarrow\infty} u(y_k)\geq g(0).$
To prove the continuity of $u$ on $\p\Omega$, we are left with showing 
\begin{equation} \limsup_{k\rightarrow\infty} u(y_k)\leq g(0).
 \label{usupg}
\end{equation}
It relies on the following claim:\\
{\bf Claim 2.} Let $A(x)= g(0)+ \e +  C_1 x_n$ where $C_1$ is as in (\ref{C1eq}). 
Then $A(x)\geq g(x)$ on $\p\Omega$.

The proof of {\bf Claim 2} is also quite elementary. If $x\in \p\Omega\cap B_{\delta}(0)$ then by (\ref{gcts}), we have $g(x)< g(0)+ \e\leq A(x)$. If 
$x\in\p\Omega\setminus B_{\delta}(0)$ then $x_n\geq\eta$ and hence 
$$A(x)\geq g(0)+ \e + C_1\eta >g(0) + 2\|g\|_{L^{\infty}(\p\Omega)} \geq g(x).$$
Now, returning to the proof of {\bf Step 2}. If $a\in\mathcal{F}$ then $a(x)\leq g(x)\leq A(x)$ for all $x\in\p\Omega$. Since both $a$ and $A$ are affine, we have $a(x)\leq A(x)$ for all 
$x\in\overline{\Omega}$. By taking the supremum over $a\in\mathcal{F}$, we find $u(x)\leq A(x)$ for all $x\in\overline{\Omega}$.
In particular, (\ref{usupg}) then follows from 
\begin{equation*}
 \limsup_{k\rightarrow\infty} u(y_k)\leq  \limsup_{k\rightarrow\infty} A(y_k) = g(0)+ \e.
\end{equation*}
{\bf Proof of Step 3.} Let $p\in \p u(\Omega)$. Then $p\in\p u(x_0)$ for some $x_0\in\Omega$, and hence
\begin{equation}
 \label{pstep3} u(x)\geq u(x_0) + p\cdot(x-x_0):= a(x)~\text{for all } x\in\Omega.
\end{equation}
{\bf Claim 3}: There is $y\in\p\Omega$ such that $g(y)=a(y)$.

Indeed, from (\ref{pstep3}), $u\equiv g$ on $\p\Omega$ by {\bf Step 1}, and the continuity of both $u$ and $g$, we find $g(x)\geq a(x)$ for all $x\in\p\Omega$. 
If {\bf Claim 3} does not hold, then by the continuity of $g$ and $a$, there is $\e>0$ such that $g(x)\geq a(x) +\e$ for all $x\in\p\Omega$. Therefore, $a+\e\in\mathcal{F}$. By the definition of $u$,
we have $u(x)\geq a(x) +\e$ for all $x\in\Omega$ but this contradicts $u(x_0)= a(x_0)$. \\
{\bf Claim 4}: $a(x)$ is a supporting hyperplane to the graph of $u$ at $(z, u(z))$ for $z$ on a whole open segment $I$ connecting $x_0\in\Omega$ to $y\in\p\Omega$.

To prove {\bf Claim 4}, we show that $u(z)\leq a(z)$ for all $z\in I$ because we already have $u(x)\geq a(x)$ for all $x\in\Omega$. Let $z=\alpha x_0 + (1-\alpha) y$ where
$0\leq\alpha\leq 1$. By convexity and the fact that $a$ being affine,
$$u(z)\leq \alpha u(x_0) + (1-\alpha) u(y) = \alpha a(x_0)+ (1-\alpha) a (y) = a(\alpha x_0 + (1-\alpha)y)= a(z).$$
From {\bf Claim 4}, we have $p\in u(z)$ for all $z\in I$ and we are done with {\bf Step 3}.
\end{proof}

{\bf Strategy of the proof of Theorem \ref{muthm}.} We use the Perron method as in the case $\mu\equiv 0$. Let
$$\mathcal{F}(\mu, g)=\{v\in C(\overline{\Omega}): \text{ v convex}, \det D^2 v\geq \mu~\text{in}~\Omega, v=g~\text{on}~\p\Omega\}$$
and
$$u(x) = \sup_{v\in \mathcal{F}(\mu, g)} v(x).$$
Our goal is to show that $u$ is the desired solution.

When trying to work out the details, the first obstacle we encounter is to show that $\mathcal{F}(\mu, g)\neq\emptyset.$
This should not be too difficult, at least heuristically, for the following reason. If we focus on a point $x_0\in\Omega$, then locally, $\mu$ can be viewed as being 
squeezed between two extremes: $0$ (corresponding to $v$ being affine) and $\infty$ (corresponding to $v$ being a cone with vertex at $x_0$). Thus, we can construct an element of 
$\mathcal{F}(\mu, g)$ from the above two extremes. But it is in fact easier to work directly with the extreme cases.

\begin{rem} 
 \label{rem_reduce}
There exists a sequence of measures $\mu_j$ converging weakly to $\mu$ such that each $\mu_j$ is a finite combination of delta masses with positive coefficients and 
$\mu_j(\Omega)\leq A<\infty$ for some constant $A$.
\end{rem}
Theorem \ref{muthm} follows from the following lemmas.
\begin{lem} \label{dirac_lem} Let $\Omega\subset\R^n$ be an open, bounded and strictly convex domain and  $g\in C(\p\Omega)$. 
Let $\displaystyle\mu=\sum_{i=1}^N a_i \delta_{x_i}$ where $x_i\in\Omega$ and $a_i>0$. Then
 the problem
 \begin{equation*}
 \left\{
 \begin{alignedat}{2}
   \det D^2 u~&=\mu \h~&&\text{in} ~\Omega, \\\
 u&= g \h~&&\text{on}~ \p\Omega,
 \end{alignedat}
 \right.
\end{equation*}
has a unique convex solution $u\in C(\overline\Omega)$ in the sense of Aleksandrov.
\end{lem}

\begin{lem}
\label{comp_lem}
 Let $\Omega\subset\R^n$ be an open, bounded and strictly convex domain. Let $\mu_j, \mu$ be Borel measures in $\Omega$ such that $\mu_j(\Omega)\leq A<\infty$
 and $\mu_j$ converges weakly to $\mu$ in $\Omega$. Let 
 $g_j, g\in C(\p\Omega)$ be such that $g_j$ converges uniformly to $g$ in $C(\p\Omega)$. Let  $u_j\in C(\overline\Omega)$ be the unique convex solution in the sense of Aleksandrov
 to
 \begin{equation*}
 \left\{
 \begin{alignedat}{2}
   \det D^2 u_j~&=\mu_j \h~&&\text{in} ~\Omega, \\\
 u_j&= g_j \h~&&\text{on}~ \p\Omega.
 \end{alignedat}
 \right.
\end{equation*}
Then $\{u_j\}$ contains a subsequence, also denoted by $\{u_j\}$, such that $u_j$ converges uniformly on compact subsets of $\Omega$ to 
the unique convex solution $u\in C(\overline{\Omega})$ in the sense of Aleksandrov to
\begin{equation*}
 \left\{
 \begin{alignedat}{2}
   \det D^2 u~&=\mu \h~&&\text{in} ~\Omega, \\\
 u&= g \h~&&\text{on}~ \p\Omega.
 \end{alignedat}
 \right.
\end{equation*}
\end{lem}

The proofs here follow closely the arguments in \cite{G} and \cite{RT}.

We give a proof of Lemma \ref{comp_lem} in the special case where $\Omega$ is strictly convex and $g_j=g$ for all $j$. This suffices to prove Theorem \ref{muthm}.

We first observe a simple consequence of Aleksandrov's maximum principle, Theorem \ref{Alekmp}.
\begin{cor} 
Let $\Omega\subset \R^n$ be a bounded, open and convex set with diameter $D$, and let $u\in C(\overline{\Omega})$ be a convex function with $u\geq 0$ on $\p\Omega$. Then
$$u(x)\geq -C_n [\emph{dist} (x, \p\Omega)]^{1/n}D^{\frac{n-1}{n}} |\p u(\Omega)|^{1/n}$$
for all $x\in\Omega$ where $C_n$ is a constant depending only on the dimension $n$.
\label{cor_Alekmp}
\end{cor}
\begin{proof}[Proof of Corollary \ref{cor_Alekmp}]
 If $u(x)\geq 0$ for all $x\in\Omega$ then we are done. If this is not the case, then
 $E=\{x\in\Omega: u(x)<0\}$
 is a convex domain, with $u=0$ on $\p E$. We apply Aleksandrov's maximum principle, Theorem \ref{Alekmp}, to conclude that for each $x\in E$,
 $$(-u(x))^n=|u(x)|^{n}\leq C_n (\text{diam}(E))^{n-1} \text{dist} (x, \p E)|\p u(E)|\leq C_n D^{n-1} \text{dist} (x, \p\Omega)|\p u(\Omega)|.$$
 The corollary follows.
\end{proof}
\begin{proof} [Proof of Lemma \ref{comp_lem} in the special case]  This is the case where $\Omega$ is strictly convex and $g_j=g$ for all $j$.
Let $U\in C(\overline{\Omega})$ be the convex solution to 
\begin{equation*}
 \left\{
 \begin{alignedat}{2}
   \det D^2 U~&=0 \h~&&\text{in} ~\Omega, \\\
 U&= g \h~&&\text{on}~ \p\Omega.
 \end{alignedat}
 \right.
\end{equation*}
This $U$ exists by Theorem \ref{mu0thm}. Since $\det D^2 U\leq \det D^2 u_j$ in $\Omega$ and $U= u_j$ on $\Omega$, we have
\begin{equation}
 \label{upuj}
 u_j\leq U
\end{equation}
by the comparison principle in Lemma \ref{comp-prin}. In particular, $\{u_j\}$ are uniformly bounded from above.

Now, we try to obtain a good lower bound for $u_j$ from below that matches $U$ locally. Fix a boundary point $x_0\in\p\Omega$. We can assume that $x_0=0\in\p\Omega$ and that
$$\Omega\subset \{x\in\R^n: x_n>0\}.$$
From the continuity of $g$, given $\e>0$, there exists $\delta>0$ such that
\begin{equation*}
 |g(x)- g(0)| <\e~\text{for all~} x\in\p\Omega\cap B_{\delta}(0).
\end{equation*}
Let $a(x)= g(0)-\e-C_1 x_n$
be as in the proof of Theorem \ref{mu0thm} (see Claim 1 there). Then $a\leq g$ on $\p\Omega$. Consider
$$v_j(x)= u_j(x)- a(x).$$
Then $v_j\geq 0$ on $\p\Omega$ and $\det D^2 v_j= \det D^2 u_j=\mu_j$. By Corollary \ref{cor_Alekmp}, we have for all $x\in\Omega$
$$v_j(x)\geq -C_n [\text{dist} (x, \p\Omega)]^{1/n}(\text{diam} (\Omega))^{\frac{n-1}{n}} |\p u_j(\Omega)|^{1/n},$$
or, since $ |\p u_j(\Omega)|= \mu_j(\Omega)\leq A$ and $\text{dist}(x,\p\Omega)\leq x_n$,
\begin{equation}
 \label{uj_below}
 u_j(x)\geq g(0)-\e- C_1 x_n- C_n x_n^{1/n}(\text{diam} (\Omega))^{\frac{n-1}{n}} A^{1/n}.
\end{equation}
Hence $\{u_j\}$ are uniformly bounded from below.

From (\ref{upuj}) and (\ref{uj_below}), we can use Lemma \ref{slope-est} to conclude that $\{u_j\}$ are locally uniformly Lipschitz in $\Omega$. Hence, by the Arzela-Ascoli theorem, 
$\{u_j\}$ contains a subsequence, also denoted by $\{u_j\}$, such that $u_j$ converges uniformly on compact subsets of $\Omega$ to 
a convex function $u$ in $\Omega$. From (\ref{upuj}) and (\ref{uj_below}), we also have $u\in C(\overline{\Omega})$ and $u=g$ on $\p\Omega$. That $\det D^2 u=\mu$ follows from the weak
compactness property of the Monge-Amp\`ere measure in Lemma \ref{weakMA}.
\end{proof}
\begin{proof}[Proof of Lemma \ref{dirac_lem}] Recall that $\displaystyle\mu=\sum_{i=1}^N a_i \delta_{x_i}$ where $x_i\in\Omega$ and $a_i>0$.
  Let
$$\mathcal{F}(\mu, g)=\{v\in C(\overline{\Omega}): \text{ v convex}, \det D^2 v\geq \mu~\text{in}~\Omega, v=g~\text{on}~\p\Omega\}$$
and
\begin{equation}u(x) = \sup_{v\in \mathcal{F}(\mu, g)} v(x).
 \label{uPerron}
\end{equation}
Our goal is to show that $u$ is the desired solution. We proceed with the following steps.\\
{\bf Step 1:} $\mathcal{F}(\mu, g)\neq\emptyset$ and there is $v_0\in \mathcal{F}(\mu, g)$ with $Mv_0(\Omega)<\infty$.\\
{\bf Step 2:} If $v_1, v_2\in \mathcal{F}(\mu, g)$ then $\max\{v_1, v_2\}\in \mathcal{F}(\mu, g)$.\\
{\bf Step 3:} (Approximation property of $u$): 
The function $u$ defined by (\ref{uPerron}) is bounded from above. Moreover, for each $y\in\Omega$, there exists a sequence $v_m\in \mathcal{F}(\mu, g)$, converging uniformly on compact subsets
of $\Omega$ to a function $w\in \mathcal{F}(\mu, g)$ so that $w(y)= u(y)$.\\
{\bf Step 4:} $u\in C(\overline{\Omega})$.\\
{\bf Step 5:} $\det D^2 u=Mu\geq \mu$ in $\Omega$.\\
{\bf Step 6:} $Mu$ is concentrated on the set $X=\{x_1, \cdots, x_N\}$.\\
{\bf Step 7:} $\det D^2 u=Mu\leq \mu$ in $\Omega$.\\
{\bf Proof of Step 1.} We use the fact that $M (|x-x_i|)= \omega_n \delta_{x_i}$ where we recall that $\omega_n= |B_1(0)|$.  Let
$$u(x) = \frac{1}{\omega_n^{1/n}}\sum_{i=1}^N a_i^{1/n} |x-x_i|.$$
Then $\det D^2 u(x)\geq \mu.$ By Theorem \ref{mu0thm}, there exists a unique convex solution $U_1\in C(\overline{\Omega})$ to 
\begin{equation*}
 \left\{
 \begin{alignedat}{2}
   \det D^2 U_1~&=0 \h~&&\text{in} ~\Omega, \\\
 U_1&= g-u \h~&&\text{on}~ \p\Omega.
 \end{alignedat}
 \right.
\end{equation*}
Let $v_0= u+ U_1$. Then $v_0\in C(\overline{\Omega})$ and $v_0=g$ on $\p\Omega$. Since both $u$ and $U_1$ are convex, we have as in (\ref{MA_add})
$$\det D^2 v_0 = \det D^2 (u + U_1)\geq \det D^2 u + \det D^2 U_1\geq\mu.$$
Therefore, $v_0\in \mathcal{F}(\mu, g)$ and $Mv_0 (\Omega)<\infty$.\\
{\bf Proof of Step 2.} Let $v=\max\{v_1, v_2\}$. Given a Borel set $E\subset\Omega$, we write $E=E_0\cup E_1\cup E_2$, $E_i\subset \Omega_i$, where
$$\Omega_0=\{x\in\Omega: v_1(x)= v_2(x)\}, ~\Omega_1=\{x\in\Omega: v_1(x) >v_2(x)\}, ~\Omega_2=\{x\in\Omega: v_1(x)< v_2(x)\}.$$
We show that for each $i=0, 1,2$, $$Mv(E_i)\geq \mu(E_i).$$

The cases $i=1, 2$ are similar so we consider $i=1$. We only need show that $\p v_1(E_1)\subset \p v(E_1)$. Indeed, if $p\in \p v_1 (x)$ where $x\in E_1\subset\Omega_1$, then
$p\in \p v(x)$. This is because
$v(x)= v_1(x)$ and 
for all $y\in\Omega$, we have
$$v(y)\geq v_1 (y)\geq v_1(x) + p\cdot(y-x)= v(x) + p\cdot(y-x).$$
It remains to consider the case $i=0$. Then the same argument as above shows that $\p v_1(E_0)\subset \p v(E_0)$, and $ \p v_2(E_0)\subset \p v(E_0)$ and we are done.\\
{\bf Proof of Step 3.} 
By Theorem \ref{mu0thm}, there exists a unique convex solution $W\in C(\overline{\Omega})$ to 
\begin{equation*}
 \left\{
 \begin{alignedat}{2}
   \det D^2 W~&=0 \h~&&\text{in} ~\Omega, \\\
 W&= g \h~&&\text{on}~ \p\Omega.
 \end{alignedat}
 \right.
\end{equation*}
For any $v\in \mathcal{F}(\mu, g)$, we have 
$\det D^2 W\leq \det D^2 v$ in $\Omega$ while $W= v$ on $\Omega$. Hence
$
 v\leq W
$
by the comparison principle in Lemma \ref{comp-prin}. In particular, $v$ is uniformly bounded from above and so is the function  $u$ defined by (\ref{uPerron}).

Now, let $y\in\Omega$. Then, by the definition of $u$,
there is a sequence $\bar v_m\in \mathcal{F}(\mu, g)$ such that $\bar v_m(y)\rightarrow u(y)$ as $m\rightarrow \infty$.
By {\bf Step 1}, there is $v_0\in \mathcal{F}(\mu, g)$ with $Mv_0(\Omega)<\infty$.
Let $$v_m=\max\{v_0, \bar v_m\}.$$ By {\bf Step 2}, we have $v_m\in \mathcal{F}(\mu, g)$. Moreover, $v_m\leq u$ in $\Omega$ while $\bar v_m(y)\leq v_m(y)\leq u(y)$ and so $v_m(y)\rightarrow u(y)$ as
$m\rightarrow\infty$.
It follows from $v_0= v_m$ on $\p\Omega$ and
$v_0(x)\leq v_m(x)$
that $\p v_m(\Omega)\subset \p v_0 (\Omega)$
and 
$$M v_m (\Omega)\leq Mv_0(\Omega):=A<\infty.$$
Up to extracting a subsequence, $Mv_m$ converges weakly to a Borel measure $\nu$ in $\Omega$ with $\nu\geq \mu$. By Lemma \ref{comp_lem}, 
$\{v_m\}$ contains a subsequence, also denoted by $\{v_m\}$, such that $v_m$ converges uniformly on compact subsets of $\Omega$ to 
the unique convex solution $w\in C(\overline{\Omega})$, in the sense of Aleksandrov, to
\begin{equation*}
 \left\{
 \begin{alignedat}{2}
   \det D^2 w~&=\nu \h~&&\text{in} ~\Omega, \\\
 w&= g \h~&&\text{on}~ \p\Omega.
 \end{alignedat}
 \right.
\end{equation*}
Clearly $w\in  \mathcal{F}(\mu, g)$ and $w(y)= u(y)$.\\
{\bf Proof of Step 4.} It suffices to show that $u$ is continuous on the boundary. We use the same notation as in {\bf Step 3}. As in the proof of Lemma
\ref{comp_lem}, at a boundary point, say $0\in\p\Omega$ where $\Omega\subset\{x\in\R^n: x_n>0\}$, we have
$$g(0)-\e-C_1 x_n - Cx_n^{1/n}\leq v_m(x)\leq u(x)\leq W(x).$$
The continuity of $u$ at $0\in\p\Omega$ follows.\\
{\bf Proof of Step 5.} It suffices to prove that $Mu(\{x_i\})\geq a_i$ for each $i=1, \cdots, N$. We prove this estimate for $i=1$. By {\bf Step 3}, 
there exists a sequence $v_m\in \mathcal{F}(\mu, g)$, converging uniformly on compact subsets
of $\Omega$ to a convex function $w$ with $Mw\geq \mu$ so that $w(x_1)= u(x_1)$ and $w\leq u$ in $\Omega$. Thus $Mw(\{x_1\})\geq a_1$. If $p\in \p w(x_1)$ then
$p\in \p u(x_1)$ because for all $x\in\Omega$, we have
$$u(x)\geq w(x)\geq w(x_1) + p\cdot(x-x_1)=u(x_1) + p\cdot(x-x_1).$$
Therefore $\p u(x_1)\supset \p w(x_1)$ and hence $$Mu (\{x_1\})= |\p u(x_1)|\geq |\p w(x_1)|= Mw(\{x_1\})\geq a_1.$$
{\bf Proof of Step 6.} We use a lifting argument to show that $Mu$ is concentrated on the set $X=\{x_1, \cdots, x_N\}$. Let $x_0\in\Omega\setminus X$. We can choose $r>0$ such that
$B_{2r}(x_0)\in \Omega\setminus X$. Let $B=B_r(x_0)$ and $v\in C(\overline{B})$ be the convex solution to 
\begin{equation*}
 \left\{
 \begin{alignedat}{2}
   \det D^2 v~&=0 \h~&&\text{in} ~B, \\\
 v&= u \h~&&\text{on}~ \p B.
 \end{alignedat}
 \right.
\end{equation*}
Define the lifting $w$ of $u$ and $v$ by
 $$w(x) = \left\{\begin{array}{rl}
 v (x) &  x\in B,\\
u (x)  & x\in  \Omega\setminus B.
\end{array}\right.$$
Then $w\in C(\overline{\Omega})$ with $w=g$ on $\p\Omega$. We claim that $w\in \mathcal{F}(\mu, g).$ Since $\det D^2 u\geq 0= \det D^2 v$ in $B$ and $u=v$ on $\p B$, we have $v\geq u$ in $B$. Thus
$w$ is convex. 

We now verify that $Mw(E)\geq \mu(E)$ for each Borel set $E\subset \Omega.$ Let $E= E_1 \cup E_2$ where $E_1 = E\cap B$ and $E_2= E\cap (\Omega\setminus B).$
As in {\bf Step 2}, we have $Mw(E_1)\geq Mv (E_1)$ and $Mw(E_2)\geq Mu(E_2)$. Hence,
$$Mw(E)= Mw(E_1) + Mw(E_2)\geq Mv(E_1)+ Mu(E_2)\geq Mu(E_2)\geq \mu (E_2)\geq \mu (E\cap X)= \mu (E).$$
This shows that $w\in \mathcal{F}(\mu, g)$. From the definition of $u$, we have $w\leq u$. By the above argument, we have $w=v\geq u$ in $B$. Thus, we must have $u=v$
in $B$. It follows that $Mu(B)=0$ for any ball $B= B_r(x_0)$ with $B_{2r}(x_0)\in \Omega\setminus X$. Hence, if $E$ is a Borel set with $E\cap X=\emptyset$ then $Mu(E)=0$
by the regularity of $Mu$. Therefore, $Mu$ is concentrated on the set $X$, that is
$$Mu=\sum_{i=1}^n \lambda_i a_i \delta_{x_i}$$
with $\lambda_i\geq 1$ for all $i=1, \cdots, N$.\\
{\bf Proof of Step 7.} We show that $\lambda_i=1$ for all $i$. We argue by contradiction. Suppose that $\lambda_i>1$ for some $i$. To fix the idea, we can assume that $a_i=1$ and 
in some ball, say $B_r(0)$, we have $Mu=\lambda \delta_0$ with $\lambda>1$ while  $\mu=\delta_0$.
The main idea here is to locally insert a cone with Monge-Amp\`ere measure $\delta_0$ that is above $u$. This will contradict the maximality of $u$. 

Since $\p u(0)$ is convex, there is a ball $B_{\e}(p_0)\subset \p u (0).$ Then $u(x)\geq u(0) + p\cdot x$ for all $p\in B_{\e}(p_0)$ and all $x\in\Omega$. By subtracting a 
linear function $p_0\cdot x$ from $u$ and $g$, we can assume that for all $x\in\Omega$, 
\begin{equation}u(x)\geq u(0) + \e |x|.
 \label{ucone}
\end{equation}

Indeed, let $v(x)= u(x)-p_0\cdot x$. Then $v(x)\geq v(0) + (p-p_0)\cdot x$ for all  $p\in B_{\e}(p_0)$ and all $x\in\Omega$. Given $x\in\Omega$, we take $p-p_0=\e x/|x|$ and so
$v(x)\geq v(0) + \e |x|.$

Given (\ref{ucone}), we continue the proof as follows. By subtracting a constant from $u$ and $g$ we can assume that $u(0)<0$ but $|u(0)|$ is small while $u(x)\geq 0$ for $|x|\geq r$. 
The set 
$E=\{x\in\Omega: u(x)<0\}$
is a convex set of $\Omega$. It contains a neighborhood of $0$.  On $E$, we have $M (\lambda^{-1/n} u )= \delta_0.$
We now define the lifting of $u$ and $\lambda^{-1/n} u $ by
 $$w(x) = \left\{\begin{array}{rl}
 u (x) &  \text{if } x\in \Omega\setminus E,\\
\lambda^{-1/n} u (x) & \text{if } x\in E.
\end{array}\right.$$
As in {\bf Step 6}, we have 
$w\in \mathcal{F}(\mu, g)$ but $w(0)=\lambda^{-1/n} u(0)> u(0)$. This contradicts the definition of $u$. Therefore $\lambda=1$ and the proof of Theorem \ref{muthm} is complete.
\end{proof}

\subsection{Sections of convex functions}
We first observe that the Monge-Amp\`ere equation
$\det D^2 u(x) = f(x)$ is invariant under affine transformation: we can "stretch" $u$ in one direction and at the same time "contract" it in other directions 
to get another solution. If $T$ is an invertible affine transformation then by Lemma \ref{rescaling_lem},
$$\tilde u(x) = (\det T)^{-2/n} u(Tx)$$
solves
$$\det D^2 \tilde u(x) = \tilde f(x):= f(Tx).$$
In particular, the Monge-Amp\`ere equation is invariant under the special linear group 
$$SL(n)=\{n\times n ~\text{matrix}~ A~\text{such that}~\det A=1\}.$$
This property is extremely important in studying fine properties of solutions to the  Monge-Amp\`ere equation.
The John lemma and rescaling the  Monge-Amp\`ere equation using 
the invariant group $SL(n)$ allow us to focus on the domains that are roughly Euclidean balls.

A central notion in the theory of Monge-Amp\`ere equation is that of sections of convex functions, introduced and investigated by 
Caffarelli \cite{C1, C2, C3, C4}. They play the role that balls have in the uniformly elliptic equations.
\begin{defn}(Section)
\label{sec_def}
Let $u$ be a convex function on $\overline{\Omega}$ and let $p\in \p u(x)$ be a subgradient of $u$ at $x\in\Omega$. The section of $u$ centered $x$ with slope $p$ and 
height $h$, denoted by
$S_{u}(x, p, h)$, is defined by
$$S_u(x, p, h) =\{y\in\overline{\Omega}: u(y) < u(x) +p\cdot (y-x) + h\}.$$
\end{defn}
\begin{rem}
 When $u$ is $C^1$ at $x$, if $p\in \p u(x)$ then $p= Du(x)$. In this case, we simply write $S_u(x, h)$ for $S_u(x, Du(x), h)$.
\end{rem}

As an example, consider the sections of $u(x) = \frac{M}{2}|x|^2$ for a positive constant $M$, defined on $\R^n$. Then the section of $u$ at the origin with height $h$ is
\begin{eqnarray*}
S_u(0, h)=\{y\in\R^n: u(y)<h\}=B_{\sqrt{2h/M}} (0).
\end{eqnarray*}
Note that
$$|S_u(0, h)| =\omega_n\frac{2^{n/2} h^{n/2}}{M^{n/2}}= \frac{\omega_n 2^{n/2} h^{n/2}}{(\det D^2 u)^{1/2}}.$$
Remarkably, up to a factor of $h^{n/2}$, the volume growth of sections $S_u(x, p, h)$ that are compactly included in the domain depends only on the bounds on the 
Monge-Amp\`ere measure, as stated in the following theorem.
\begin{thm}[Volume of sections]\label{vol-sec1}
Suppose that $u$ is a convex solution to the Monge-Amp\`ere equation
$\lambda \leq \det D^2 u \leq \Lambda~\text{in}~\Omega$. Suppose that $p\in \p u(x)$ 
and that $S_u(x, p, h)\subset\subset \Omega.$
Then
$$c(\Lambda, n)h^{n/2} \leq |S_u(x, p, h)| \leq C(\lambda, n) h^{n/2}.$$
Thus, up to constants depending only on $n,\lambda,\Lambda$, the volume of a compactly included section $S_u(x, p, h)$ does not depend on the subgradient $p$. 
\end{thm}
\begin{proof} 
Let 
$\bar u(y)= u(y) -[u(x) + p\cdot (y-x)+ h].$
Then 
$\bar u\mid_{\p S_u(x, p, h)} =0$ and 
$\bar u$ achieves its minimum $-h$ at $x$. 
By John's lemma, we can 
find an affine transformation $Tx =Ax+ b$ such that
\begin{equation}B_1 (0)\subset T^{-1} (S_u(x, p, h))\subset B_n(0).
 \label{B1n}
\end{equation}
Let $\tilde u (y) = (\det A)^{-2/n} \bar u(Ty). $
Then, by Lemma \ref{rescaling_lem}, we have
$$\lambda \leq \det D^2 \tilde u(y) \leq \Lambda,$$
and $\tilde u =0$ on $\p T^{-1}(S_u(x, p, h)).$
Now, by Theorem \ref{drop-thm}, we know that
$$c(\lambda, n) \leq |\min_{T^{-1}(S_u(x, p, h))} \tilde u| \leq C(\Lambda, n).$$
However, since $|\min_{S_u(x, p, h)} \bar u| =h$, we have
$$c(\lambda, n) \leq h (\det A)^{-2/n} \leq C(\Lambda, n).$$
This gives
$$c(\Lambda, n)h^{n/2} \leq  \det A \leq C(\lambda, n) h^{n/2}.$$
On the other hand, by (\ref{B1n}), we deduce that
$$\omega_n\leq (\det A)^{-1}|S_u(x, p, h)| \leq n^n \omega_n,$$
hence 
$$\omega_n \det A\leq |S_u(x, p, h)| \leq n^n \omega_n\det A$$
and the desired inequalities follow.
\end{proof}
An important question is the following:
\begin{quest}\label{strict_quest}
 Under what conditions can we conclude from $\lambda\leq \det D^2 u\leq \Lambda$ in $\Omega$ that for all $x\in \Omega$, there is a small $h(x)>0$ such that 
 $S_u(x, p, h(x))\subset\subset
 \Omega$ for all $p\in\p u(x)$?
\end{quest}
An answer will be given in a later theorem using Caffarelli's localization theorem \cite{C1}. 

The following lemma partially 
extends the volume estimates in Theorem \ref{vol-sec1} for sections of $u$ that may not be compactly supported in the domain. 
\begin{lem} \label{vol-sec2}
Assume that $\det D^2 u\geq \lambda$ in a bounded and convex domain $\Omega$ in $\R^n$. Then for any section
$$S_{u}(x, p, h)=\{y\in\overline\Omega: u(y) < u(x) + p\cdot (y-x) + h\},~x\in\overline{\Omega},~p\in\p u(x),$$
we have
$$|S_{u}(x, p, h)|\leq C(\lambda, n) h^{n/2}.$$
\end{lem}
\begin{proof} Because $\Omega$ is bounded and convex, the section $S_{u}(x, p, h)$ is also bounded and convex.
Let
$\bar u(y)= u(y)- [u(x) + p\cdot (y-x)+ h].$
Then
$\bar u\mid_{\p S_{u}(x, p, h)}\leq 0$ and $|\min_{S_{u}(x, p, h)}\bar u|=h.$ By John's lemma, there is an affine transformation $Tx =Ax+ b$ that normalizes $S_{u}(x, p, h)$, that is
\begin{equation}B_1(0) \subset \tilde \Omega = T (S_{u}(x, p, h))\subset B_{n}(0).
 \label{uJohn}
\end{equation}
Let $\tilde u (x) = |\det A|^{2/n} \bar u(T^{-1}x).$
Then $$\det D^2 \tilde u(x) =\det D^2 \bar u (T^{-1}x)\geq \lambda,~\tilde u\mid_{\p\tilde\Omega}\leq 0.$$
Let $v(x) =\frac{\lambda^{1/n}}{2}(|x|^2-1)$. Then 
$$\det D^2 v=\lambda \leq \det D^2 \tilde u~\text{in}~\tilde \Omega.$$
Using (\ref{uJohn}), we have on $\p\tilde\Omega$, $v\geq 0\geq \tilde u$.
By the comparison principle in Lemma \ref{comp-prin}, we find that $\tilde u\leq v$ in $\tilde\Omega$. It follows that
$\min_{\tilde\Omega} \tilde u\leq \min_{\tilde\Omega} v = -\frac{\lambda^{1/n}}{2}.$
Therefore,
$$|\det A|^{2/n} h=|\min_{\tilde\Omega} \tilde u|\geq \frac{\lambda^{1/n}}{2}.$$
The result now follows from (\ref{uJohn}) since
$$|S_{u}(x, p, h)|\leq |\det A|^{-1} C(n) \leq C(\lambda, n) h^{n/2}.$$
\end{proof}
\begin{rem}
 In what follows, we simply write $S_u(x, h)$ for the section of $u$ centered $x$ with slope $p\in \p u(x)$ and 
height $h$. This is due to the fact that all of our statements will not depend on a specific choice of $p$ in $\p u(x)$ (in case $u$ is not $C^1$ at $x$).
\end{rem}

\newpage
\section{Geometry of sections of solutions to the Monge-Amp\`ere equation}
\label{MA2_sec}
In this section, we discuss some compactness results in the Monge-Amp\`ere setting and use them to prove Caffarelli's celebrated $C^{1,\alpha}$ regularity of strictly convex solutions. 
A very important result in this section is
Caffarelli's localization theorem, Theorem \ref{Caf_loc}. Refined geometric properties of sections will be proved, including:
 Estimates on the size of sections in terms of their height and the Monge-Amp\`ere measure, Lemma \ref{sec-size}; 
 Engulfing property of sections, Theorem \ref{engulfthm}; and 
 Inclusion and exclusion property of sections, Theorem \ref{pst}.
 
 The functions $u$ involved in the Monge-Amp\`ere equations in this section are always assumed to be convex. Similarly, unless otherwise indicated, $\Omega$ is also assumed
 to be a convex set in $\R^n$.
\subsection{Compactness of solutions to the Monge-Amp\`ere equation}
The first result gives compactness of a family of convex functions on a normalized domain with zero boundary value and having Monge-Amp\`ere measure bounded from above.
\begin{thm} \label{Blaschke_thm}
Let 
$$C_{\Lambda}=\{(\Omega, u): \det D^2 u\leq\Lambda, u\mid_{\p\Omega}=0, B_1(0)\subset \Omega\subset B_{C(n)}(0)\}.$$
Then $C_{\Lambda}$ is compact in the following sense. For any sequence $\{(\Omega_i, u_i)\}_{i=1}^\infty\subset C_{\Lambda}$, we can find a subsequence, still labeled $\{(\Omega_i, u_i)\}_{i=1}^\infty$, and $(\Omega, u)\in C_\Lambda$ such that
\begin{myindentpar}{1cm}
(i) $\Omega_i$ converges to $\Omega$ in the Hausdorff distance;\\
 (ii) $u_i$ converges to $u$ locally uniformly on compact subsets of $\Omega$.
\end{myindentpar}
Similar statement holds for the set
$$C_{\lambda, \Lambda}=\{(\Omega, u): \lambda\leq \det D^2 u\leq\Lambda, u\mid_{\p\Omega}=0, B_1(0)\subset \Omega\subset B_{C(n)}(0)\}.$$
\end{thm}
\begin{proof} 

Observe the following uniform H\"older estimate: There is a universal constant $C$ depending only on $n$ and $\Lambda$ such that if $(\Omega, u)\in C_{\Lambda}$ then
\begin{equation}|u(x)- u(y)|\leq C|x-y|^{1/n}~\forall x, y\in \overline{\Omega}.
\label{uHolder}
\end{equation}
Indeed, if $y\in\p\Omega$ and $x\in\Omega$ then the above inequality follows from Aleksandrov's maximum principle (Theorem \ref{Alekmp}). Consider now the case $x, y\in\Omega$. 
Suppose the ray $yx$ intersects $\p\Omega$ at $z$. Then $x= \alpha y + (1-\alpha) z$ for some $\alpha\in (0,1)$. It follows that $x-y= (1-\alpha)(z-y)$. By convexity,
$u(x)\leq \alpha u(y) + (1-\alpha)u(z) =\alpha u(y),$ which implies
$$u(x)- u(y) \leq (\alpha-1) u(y)= (1-\alpha)|u(y)|\leq C(1-\alpha)|z-y|^{1/n}\leq C(1-\alpha)^{\frac{n-1}{n}}|x-y|^{1/n}.$$
Suppose we are given any sequence $\{(\Omega_i, u_i)\}_{i=1}^\infty\subset C_{\Lambda}$.
By the Blaschke selection theorem,  we can find a subsequence, still labeled $\{(\Omega_i, u_i)\}_{i=1}^\infty$, such that
$\Omega_i$ converges to $\Omega$ in the Hausdorff distance. By Aleksandrov's theorem, Theorem \ref{Alekmp}, $|u_i|\leq C(n,\Lambda)$ on $\Omega_i$ for all i.  

It follows from (\ref{uHolder}) and the Arzela-Ascoli theorem that, up to extracting a  further subsequence, $u_i\rightarrow u$ locally uniformly in $\Omega$. The bound on 
the Monge-Amp\`ere 
measure of $u$ follows from Lemma \ref{weakMA}. It remains to show that $u=0$ on $\p\Omega$. Because $u_i\leq 0$ in $\Omega_i$, we have $u\leq 0$ on $\p\Omega$. Let
$K\subset \Omega$ be such that $\text{dist}(x, \p\Omega)\leq \delta$ for all $x\in K$ where $\delta>0$. Then there is a large $i_0$ depending on $K$ and $\delta$
such that $K\subset \Omega_i$ and $\text{dist} (x, \p\Omega_i)<2\delta$ for all $x\in K$ and for $i\geq i_0$. By Aleksandrov's
theorem, Theorem \ref{Alekmp}, we have $|u_i|\leq C(n,\Lambda)\delta^{1/n}$ in $ K$. Thus, from 
$u_i\rightarrow u$ locally uniformly in $\Omega$, we conclude that $u=0$ on $\p\Omega$.
\end{proof}
\subsection{Caffarelli's localization theorem}
\begin{defn} (Extremal point) Let $\Omega\subset \R^n$ be a convex set. A point $x_0\in\p\Omega$ is called an extremal point of $\Omega$ if $x_0$ is not a convex combination of other points in $\overline{\Omega}$.
\end{defn}
Let $\Omega\neq\emptyset$ be a closed convex and bounded subset of $\R^n$. Then the set $E$ of extremal points of $\Omega$ is non-empty.
\begin{lem}[Balancing] \label{balance_lem} Suppose that $u$ satisfies
$\lambda\leq \det D^2 u\leq \Lambda$ in $\Omega$, and $u=0$ on $\p\Omega.$
Let $l$ be a line segment in $\Omega$ with 2 endpoints $z^{'}, z^{''}\in\p\Omega.$ Let $z\in l$ be such that $$u(z)\leq \alpha\inf_{\Omega} u~(0<\alpha<1).$$
Then
$$|z'-z|\geq c(n,\alpha,\lambda,\Lambda)|z^{''}-z'|.$$
\end{lem} 
\begin{proof}
Note that the ration $\frac{|z'-z|}{|z^{''}-z^{'}|}$ is invariant under linear transformations. Hence by making a linear 
transformation using John's lemma, we may assume that $B_1(0)\subset \Omega\subset B_n(0)$. By Theorem \ref{drop-thm}, we have $\inf_{\Omega} u\leq -c(n,\lambda).$ Hence, when $u(z)\leq -\alpha c(n,\lambda)$, 
we have $\text{dist} (z,\p\Omega)\geq c(n,\alpha,\lambda, \Lambda)$ by Aleksandrov's estimate, Theorem \ref{Alekmp}. The lemma then follows.
\end{proof}
\begin{thm} [Caffarelli's localization theorem \cite{C1}] Suppose that $\lambda\leq \det D^2 u\leq \Lambda$ in $\Omega$. 
Then for any point $x_0\in\Omega$, either the contact set $\mathcal{C}= \{x\in\Omega: u(x)= l_{x_0}(x)\}$ is a single point, where 
$l_{x_0}$ is a supporting hyperplane to the graph of the function $u$ at $(x_0, u(x_0))$, or $\mathcal{C}$
has no interior extremal points in $\Omega$.
\label{Caf_loc}
\end{thm}
\begin{proof} By subtracting an affine function, we can assume that $u\geq 0$ in $\Omega$ and $\mathcal{C}=\{u=0\}$. 
Suppose that $\mathcal{C}$ contains more than one point and that the conclusion of the theorem is false, that is $\mathcal{C}$ contains an interior extremal point in $\Omega$.
By changing coordinates, we can assume that
$0\in\Omega$ is an interior extremal point of the set $\{u=0\}\subset \{x_n\leq 0\}$, and furthermore, $\mathcal{C}_{\delta}:=\{0\geq x_n\geq -\delta\}\cap\mathcal{C}$ is compactly contained in $\Omega$ for some small $\delta$. Choose $x''\in \{x_n=-\delta\}\cap\mathcal{C}\subset\Omega$. Let
$$G_{\e}:=\{x\in\Omega: v_{\e}(x):=u(x) -\e (x_n +\delta)\leq 0\}.$$
Then $G_{\e}\subset \{x_n\geq -\delta\}$ and $G_{\e}$ shrinks to  
$\mathcal{C}_{\delta}$ as $\e\rightarrow 0$.  Hence, when $\e$ is small,  $G_{\e}\subset\Omega$  and $v_{\e}=0$ on $\p G_{\e}$. 
Clearly $x''\in \p G_{\e}$ and $0$ is an interior point of $G_{\e}$.  We observe that on $G_{\e}$
$$v(x) \geq -\e(M+\delta),$$
where $M= \text{diam}(\Omega)$
and therefore
$$v(0) =-\e\delta \leq \frac{\delta}{M + \delta} \inf_{G_{\e}} v(x).$$
Now, $0$ is an interior point on the segment connecting some point $x^{\e}$ of $\p G_{\e}\cap\{x_n\geq 0\}$ and $x''$ . By Lemma \ref{balance_lem}, we have $$|x_n^{\e}| \geq c(n,\delta, \lambda,\Lambda, M)|x_n^{\e}-x''_n|\geq c(n,\delta, \lambda,\Lambda, M).$$
This contradicts the fact that $x_n^{\e}\rightarrow 0$ when $\e\rightarrow 0.$
\end{proof}
A consequence of Theorem \ref{Caf_loc} is the following strict convexity result and its quantitative version. It gives an answer to Question \ref{strict_quest}.
\begin{thm} \label{strict_thm}
~~
\begin{myindentpar}{1cm}
(i) Suppose that $\lambda\leq \det D^2 u\leq \Lambda$ in a convex domain $\Omega$ and $u=0$ on $\p\Omega$. Then $u$ is strictly convex in $\Omega$. This implies that $u$ cannot coincide a supporting hyperplane in 
$\Omega$.\\
(ii) Suppose that $u\geq 0$, $u(0)=0$,
$$B_{1}(0)\subset \Omega= S_u(0, 1)\subset B_n(0), ~u=1~\text{on}~\p\Omega,~ \text{and } \lambda\leq \det D^2 u\leq \Lambda.$$
Then for any $x\in\Omega$ with $dist(x,\Omega)\geq \alpha>0$, there is a universal constant $h(\alpha, n,\lambda,\Lambda)>0$ such that $S_u(x, h)\subset\subset\Omega$.
\end{myindentpar}
\end{thm}
\begin{proof}~\\ (i) Let $x_0\in\Omega$ and let $l$ be a supporting hyperplane to the graph of $u$ at $(x_0, u(x_0))$. If $u$ is not strictly convex at $x_0$, then the contact set $\mathcal{C}= \{x\in\Omega: u(x)=l(x)\}$ is
not a single point. By Theorem \ref{Caf_loc}, all extremal points of $\mathcal{C}$ lie on the boundary $\p\Omega$. It is easy to see that $l=0$ on $\mathcal{C}$. Using the convexity of $u$, we find that $u\leq 0$ in $\Omega$ but $u$ is above $l$. This is only possible when $u=l=0$ in $\Omega$, contradicting $\det D^2 u\geq\lambda$.\\
(ii) This follows from a compactness argument using Theorem \ref{Blaschke_thm} and the strict convexity result in part (i). 
Let $\Omega_{\delta}=\{x\in\Omega: \text{dist}(x,\p\Omega)\geq \delta\}$. Let $\nabla u(x)$ be the slope of a supporting hyperplane to the graph of 
$u$ at $(x, u(x))$ where $x\in\Omega$.
It suffices to derive a contradiction from
the following scenario: suppose we can find a sequence of convex functions $u_{k}\geq 0$ on $\Omega$ with $B_{1}(0)\subset \Omega= S_{u_k}(0, 1)\subset B_n(0), ~u_k=1~\text{on}~\p\Omega$, 
$u_k(0)=0$ 
such that
\begin{equation}
u_{k}(y_{k})\leq u_{k}(x_k) + \nabla u_{k}(x_k) \cdot(y_k-x_k) + h_{k}
 \label{uk}
\end{equation}
for sequences $x_{k}\in\Omega_{\alpha}$, $y_k\in\p\Omega_{\alpha/2}$ and $h_{k}\rightarrow 0$. From Theorem  \ref{Blaschke_thm}, 
after passing to a subsequence if necessary, we may assume
$$u_{k} \to u_* \quad \mbox{locally uniformly on $\Omega $}, \quad x_k \to x_*, \quad y_k \to y_*.$$
Moreover $u_*$ satisfies $$\lambda \le \det D^{2} u_* \le \Lambda~ \text{in} ~\Omega ~\text{and}~ u^{\ast}=1 ~\text{on} ~\p\Omega.$$
On the other hand, 
using Lemma \ref{slope-est}, we find that
\begin{equation*}
 \abs{\nabla u_k}\leq C(\alpha, n,\Lambda)~\text{in}~ \Omega_{\alpha/4}.
\end{equation*}
Thus, after passing a further subsequence, we can deduce from (\ref{uk}) that
\begin{equation*}
 u_*(y_*) = u_*(x_*) + \nabla u_*(x_*) \cdot (y_*-x_*),
\end{equation*}
where $\nabla u_{\ast} (x_{\ast})$ is the slope of a supporting hyperplane to the graph of $u_{\ast}$ at $(x_*, u_*(x_*))$. 
This implies that the contact set $\{x\in\Omega: u_\ast (x)= u_*(x_*) + \nabla u_*(x_*)\cdot (x-x_*)\}$ is not a single point and
 we reached a contradiction to the conclusion in (i).
\end{proof}
\subsection{Strict convexity and $C^{1,\alpha}$ estimates}
The following lemma is a consequence of Aleksandrov's maximum principle and convexity.
\begin{lem}\label{theta_lem}
Suppose that u is a convex function defined on $\Omega$ that contains the origin. Furthermore, assume that $u$ satisfies $$u(0)=0, u\geq 0,~\lambda\leq \det D^2 u\leq\Lambda~\text{in}~\Omega.$$
If $S_u(0, h)\subset\subset\Omega$ then for any $x\in \p S_u(0, h)$, we have
$$u(\theta x)\geq \frac{1}{2} u(x)$$
where $\theta\in (\frac{1}{2}, 1)$ depending only on $n,\Lambda,\lambda.$
\end{lem}
\begin{proof}
By John's lemma (Lemma \ref{John_lem}), there is a linear transformation $T$ such that
$B_{1/n}(0)\subset T(S_u(0, h))\subset B_1(0).$
Let
$$v(y) =\frac{u(T^{-1} y)}{h}, y\in T(S_u(0, h)).$$
Then $v=1$ on $\p T(S_u(0,h))= T(\p S_u(0,h))$. Using the volume estimates on $S_u(0,h)$, we find that
$$C^{-1}(n,\lambda,\Lambda)\leq \det D^2 v(y) \leq C(n,\lambda,\Lambda).$$
From the Aleksandrov maximum principle in Theorem \ref{Alekmp}, we have
$$\text{dist} (T(S_u(0, h/2)), T(\p S_u(0, h)))= \text{dist}(\{v<1/2\}, \{v=1\})\geq C_0$$
where $C_0$ depends only on $n,\lambda,\Lambda$. If $y\in\{v=1\}$ then the segment $0y$ intersects $\p \{v=1/2\}$ at $z$ with $|y-z|\geq C_0$. Since $|y|\leq 1$, we have
$$|z|= |y|-|z-y| \leq |y|(1-\frac{1}{2} C_0)=\theta |y|,~\theta:= 1-\frac{1}{2}C_0.$$
Therefore, for all $y\in\{v=1\}$, we have by convexity, 
$$v(\theta y)\geq \frac{1}{2} v(y).$$ For each $x\in\p S_u(0, h)$, we apply the above inequality to $y= Tx$ and then rescale back to obtain the desired result.
\end{proof}
Combining the above lemma with strict convexity, we can estimate the size of sections.
\begin{lem}[Size of sections]\label{sec-size}
Assume that that
$$\lambda\leq \det D^2 u\leq \Lambda~\text{in}~\Omega,~u(0)=0, u\geq 0.$$
Suppose that $S_u(0, t_0)\subset\subset\Omega$ is a normalized section, that is,
$B_{1}(0)\subset S_u(0, t_0)\subset B_n(0).$
Then there is a universal constant $\mu\in (0,1)$ such that for all sections $S_u(z, h)\subset\subset\Omega$ with $z\in S_u(0, 3/4 t_0)$, we have
$$S_u(z, h)\subset B_{Ch^{\mu}}(z).$$
\end{lem}
\begin{proof} By the volume estimates for sections in Theorem \ref{vol-sec1}, we see that $t_0$ is bounded from above and below by a universal constant depending
only on $n,\lambda,\Lambda$. Without loss of generality, we can assume that $t_0=1$.

We first consider the case where the center $z$ of the section is the origin.
 From Lemma \ref{theta_lem}, 
$$u(x)\geq 2^{-1} u(\theta^{-1}x)$$
for $x$ near the origin. Repeating the argument, we see that
$$u(x)\geq 2^{-k} u(\theta^{-k}x)$$ as long as $\theta^{-k}x\in\Omega$ (by then $\theta^{-k}x\in S_u(0, h)$ for some $h$). For any $x$ near the origin, let $k$ be such that
$$\frac{1}{n}\geq |\theta^{-k}x|\geq\frac{\theta}{n}.$$
Then $\theta^k\approx |x|$. Using the strict convexity result of Theorem \ref{strict_thm} and a compact argument invoking Theorem \ref{Blaschke_thm}, we obtain
$$u\mid_{\p B_{\theta/n}(0)}\geq C^{-1}(n,\lambda,\Lambda).$$
It follows that for $\theta^{1+\beta}= 2^{-1}$, we have
$$u(x)\geq \theta^{k(1+\beta)} u(\theta^{-k}x) \geq C^{-1}(n,\lambda,\Lambda)\theta^{k(1+\beta)} \geq c|x|^{1+\beta}.$$
The above estimate says that if $y\in S_u(0, h)$ then 
$h\geq c|y|^{1+\beta},$
or
$$S_u(0, h)\subset B_{Ch^{\mu}}(0), ~\mu =\frac{1}{1+\beta}.$$

Consider now the general case $z\in S_u(0, 3/4).$ There is a universal constant $\delta>0$ depending only on $n,\lambda,\Lambda$ such that $S_u(z,\delta )\subset S_u(0, 4/5)$. Let 
$Tx= Ax + b$ be an affine transformation  that normalizes $S_u(z,\delta)$, that is
$$B_1 (0)\subset T(S_u(z, \delta))\subset B_n(0).$$ 
By the gradient bound of $u$ in $S_u(0, 4/5)$, we can find a small $c$ universal such that $S_u(z, \delta)\supset B_{c\delta}(z)$. It is now easy to see that the eigenvalues $\lambda_i$ of $A$ satisfy $|\lambda_i|\leq C\delta^{-1}$. 
By the volume bound on section we have
$$|\det A| |S_u(z, \delta)|\approx 1$$
and hence all $\lambda_i$ satisfy the lower bound $|\lambda_i|\geq c\delta^{-n/2 + 1}$. Subtracting an affine function, we now assume that $u(z)=0$ and $u\geq 0$. Thus 
$S_u(z, t)=\{y\in \Omega: u(y)<t\}.$ 
Let
$$v(x) = \frac{u(T^{-1 }x)}{\delta}, x\in T(S_u(z, \delta)).$$
By using the result of the previous case to $v$, we see that 
$$S_u(z, t)= T^{-1} S_v(Tz, \frac{t}{\delta})\subset T^{-1} (B_{C(\frac{t}{\delta})^{\mu}}(Tz))\subset B_{Ct^{\mu}}(z).$$

\end{proof}
As a consequence of the strict convexity result, we have the following $C^{1,\alpha}$ regularity, due to Caffarelli \cite{C3}.
\begin{thm}[Caffarelli's pointwise $C^{1,\alpha}$ regularity \cite{C3}]\label{C1alpha}
Assume that $\lambda\leq \det D^2 u \leq\Lambda$ in $S_u(0, 1)\subset\subset\Omega$, with $u(0)=0, u\geq 0$. Then, for some universal $\delta$ depending only on $n,\lambda$ and $\Lambda$, we have
$$(\frac{1}{2} +\delta) S_{u}(0, 1) \subset S_{u}(0, 1/2)\subset (1-\delta) S_u(0, 1).$$
\end{thm}
\begin{proof}
The right inclusion follows from the Lemma \ref{theta_lem} for any $0<\delta\leq 1-\theta$. For the left inclusion, we assume by contradiction and use the compactness of normalized solution. The limiting solution is linear on some ray from $0$, contradicting the 
strict convexity result in Theorem \ref{strict_thm}. 

We can also use the argument in \cite[Lemma 3.5]{TW3} to show that the left inclusion holds. Here is how the proof goes. 
For any $x\in S_u(0, 1), x\neq 0$, we show that $u((\frac{1}{2}+\delta)x)\leq \frac{1}{2}u(x)$ if $\delta$ is universally small. For each $t\in \R$, let $v(t)= u(tx)$.
By dividing $v$ by $v(1)$, we can 
assume that $v(1)=1.$ Set
$$w(t)= v(\frac{1}{2}+\delta + t) -v'(\frac{1}{2}+ \delta) t- v(\frac{1}{2}+ \delta).$$ By the convexity of $u$, we have $w\geq 0$ and $w(0)=0$. 
Suppose now $v(\frac{1}{2}+\delta)>\frac{1}{2}$. Then using the convexity of $v$ and $v(0)=0,$ we find 
\begin{equation}\frac{1}{1+2\delta}\leq v'(\frac{1}{2}+\delta)\leq \frac{1}{1-2\delta}.
 \label{vdbound}
\end{equation}
Thus
$$w(-\frac{1}{2}-\delta)= v'(\frac{1}{2}+\delta) (\frac{1}{2}+\delta) - v(\frac{1}{2}+\delta)\leq \frac{1}{1-2\delta} (\frac{1}{2}+\delta) - \frac{1}{2}= \frac{2\delta}{1-2\delta}.$$
Using Lemma \ref{theta_lem} to $w$, we then find
\begin{eqnarray*}
 w(-\theta^{-1}(\frac{1}{2}+\delta))\leq 2 w(-\frac{1}{2}-\delta)\leq \frac{4\delta}{1-2\delta}.
\end{eqnarray*}
It follows from the definition of $w$, (\ref{vdbound}) and $v(\frac{1}{2}+\delta )\leq \frac{1}{2}+\delta $ that
\begin{eqnarray*}
 v((\frac{1}{2}+\delta )(1-\theta^{-1})) \leq \frac{4\delta}{1-2\delta} -\theta^{-1}/2 + \frac{1}{2}+\delta<0
 \end{eqnarray*}
if $\delta>0$ small because $1/2<\theta<1$. This is a contradiction to the fact that $v\geq 0$.

Now, we prove that the left inclusion implies the pointwise $C^{1,\alpha}$ regularity of $u$ in $S_u(0,1)$. If $x\in \p S_u(0,h)$ then 
$u((\frac{1}{2}+\delta)x)\leq h/2.$ Using
convexity, we find that
$$u(\frac{1}{2} x)\leq  \frac{1}{1+ 2\delta} u( (\frac{1}{2} +\delta) x) + \frac{2\delta}{1+ 2\delta}u(0)\leq \frac{1}{1+ 2\delta} \frac{u(x)}{2}.$$

Let $\alpha$ be defined by $\frac{1}{1 + 2\delta}= 2^{-\alpha}.$ 
Then for any $t\in (\frac{1}{2^{k+1}},\frac{1}{2^k})$, and $x\in \Omega$, we have
$$u(tx)\leq 2^{-k} (\frac{1}{1+ 2\delta})^k u(x) = (2^{-k})^{1+\alpha} u(x) \leq 2 t^{1+\alpha} u(x).$$
Hence $u\in C^{1,\alpha}$ at the origin. The proof at other points in $S_u(0, 1)$ is the same. 
\end{proof}
\subsection{Engulfing property of sections}
The following theorem is concerned with the engulfing property of sections \cite{GH}. It can be viewed as a triangle inequality in the Monge-Amp\`ere setting. This together with Theorem \ref{vol-sec1} shows that sections of solutions to the Monge-Amp\`ere equation have many properties similar to those of balls.
\begin{thm}[Engulfing property of sections]
Suppose that $u$ is a strictly convex solution to the Monge-Amp\`ere equation
$$\lambda \leq \det D^2 u \leq \Lambda~\text{in}~\Omega.$$ There is a universal constant $\theta_0>2$ with the following property: If $S_u(y, 2h)\subset\subset\Omega$
and 
$x\in S_u(y, h)$, then we have $S_{u}(y, h)\subset S_u(x, \theta_0 h).$
\label{engulfthm}
\end{thm}
\begin{rem}
 When $u(x)=|x|^2/2$, we have $\lambda=\Lambda=1$ and we can take $\theta_0=4$. Indeed, if $x\in S_u(y, h)= B_{\sqrt{2h}}(y)$ then from the triangle inequality, we find
 that $B_{\sqrt{2h}}(y)\subset B_{2\sqrt{2h}}(x)$, or $S_u(y, h)\subset S_u(x, 4h).$
\end{rem}

It suffices to prove the theorem in the normalized setting.
\begin{lem}
Suppose that $\lambda\leq \det D^2 u\leq \Lambda$ and $u$ is normalized in $\Omega= S_u(0, 1)$, that is $u(0)=0, Du(0)=0$, $u=1$ on $\p\Omega$ and
$B_1(0)\subset\Omega\subset B_n(0)$. There is a universally small $\delta$ such that if $x\in S_u(0, 1/2)$ then 
$S_u(0, 1/2)\subset S_u(x, \delta^{-1})$.
\end{lem}
\begin{proof} 
By the Aleksandrov maximum principle in Theorem \ref{Alekmp}, we have
$$\text{dist} (S_u(0, 1/2),\p S_u(0, 1))\geq c(n,\lambda,\Lambda).$$
By the gradient bound in Lemma \ref{slope-est}, the slopes  of $u$ in $S_u(0, 1/2)$ are 
bounded by $C$. Thus, if $z\in S_u(0, 1/2)$
then, for $\delta$ universally small,
$$u(z) - u(x) -D u(x)\cdot (z-x)\leq u(z) + |Du(x)||z-x|\leq 1/2 + 2n C<\delta^{-1}.$$
\end{proof}
\begin{proof}[Proof of Theorem \ref{engulfthm}] 
Suppose $S_u(y, 2h)\subset\subset\Omega$
and 
$x\in S_u(y, h)$. We can normalize $S_u(y, 2h)$ and thus we can assume $y=0, h=1/2$ and that $S_u(0, 1)\subset\subset\Omega$. Therefore, $x\in S_u(0, 1/2)$. 
Apply the previous lemma to find
$S_u(y, h)= S_u(0, 1/2)\subset S_u(x, \delta^{-1}).$
The result follows by choosing $\theta_0=2\delta^{-1}.$
\end{proof}
\begin{thm} [Inclusion and exclusion property of sections]
 \label{pst}
Suppose that $u$ is a strictly convex solution to the Monge-Amp\`ere equation
$$\lambda \leq \det D^2 u \leq \Lambda~\text{in}~\Omega.$$ 
There exist universal constants $c_0>0$ and $p_1\geq 1$ such that 
 \begin{myindentpar}{1cm}
 (i) if $0< r<s\leq 1$ and $x_1\in S_u(x_0, r t)$ where $S_u(x_0, 2t)\subset\subset \Omega$, then
 $$S_{u}(x_1, c_0 (s-r)^{p_1} t)\subset S_{u}(x_0, s t).$$
 (ii)  if $0< r<s< 1$ and $x_1\in S_u(x_0, t)\backslash S_{u}(x_0, st)$ where $S_u(x_0, 2t)\subset\subset
 \Omega$, then
 $$S_{u}(x_1, c_0 (s-r)^{p_1} t)\cap S_{u}(x_0, rt)=\emptyset.$$
 \end{myindentpar}
\end{thm}
\begin{proof}[Proof of Theorem \ref{pst}] We give the proof of (i) because that of (ii) is very similar. We can assume $x_0=0$.
The conclusion of the theorem
is invariant under affine transformation and rescaling of the domain $\Omega$ and function $u$. Thus, we can assume that $t=1$ and $u$ is normalized in $S_u(0, 2)$, that is 
$u=0, Du(0)=0$ and $B_1 (0)\subset S_u(0, 2)\subset B_n(0)$. 

Suppose now $x_1\in S_u(0, r)$ and $y\in S_u(x_1, c_0 (s-r)^{p_1})$. Then we have $u(x_1)< r$ and
$$u(y) <u(x_1)+ Du(x_1)\cdot (y-x_1) + c_0 (s-r)^{p_1}.$$
Because $S_u(0, r)\subset S_u(0, 1)\subset S_u(0, 2)$ and $S_u(0, 2)$ is normalized, we can deduce from the Aleksandrov maximum principle, Theorem \ref{Alekmp} applied to $u-2$, that
$$\text{dist}(S_u(0, 1), \p S_u(0, 2))\geq c(n,\lambda,\Lambda)$$ for some universal $c(n,\lambda,\Lambda)>0$. By the gradient bound of $u$ in Lemma \ref{slope-est}, we have
$|D u(x_1)|\leq C$ for some $C$ universal.
Using the estimate on the size of section in Lemma \ref{sec-size} together with the above gradient bound, we find
$$u(y)<r + C|y-x_1|\leq r + C'(c_0 (s-r)^{p_1})^{\mu} + c_0 (s-r)^{p_1}<s$$
if we choose $p_1=\mu^{-1}$ and $c_0$ small.
\end{proof}
We now prove a quantitative version of Theorem \ref{C1alpha} on interior $C^{1,\alpha}$ estimates for the Monge-Amp\`ere equation.
\begin{thm} [$C^{1,\alpha}$ estimates]\label{C1alpha2}
Assume that
$u$ is a strictly convex solution to the Monge-Amp\`ere equation
$\lambda \leq \det D^2 u \leq \Lambda$ in an open, bounded and convex domain $\Omega\subset\R^n.$
If $S_u(x_0, t)\subset\subset\Omega$
is a normalized section and $y, z\in S_u(x_0, t/2)$ then $$|Du(y)-Du(z)|\leq C(n,\lambda,\Lambda)|y-z|^{\alpha}.$$
\end{thm}
\begin{proof}
 By using the volume estimates for sections in Theorem \ref{vol-sec1}, we find that $t$ is bounded from above and below by universal constants. Without loss of generality, we can assume
 that $t=1$, $x_0=0$, $u\geq 0$ and $u(0)=0$. 
 Fix $z\in S_u(0, 1/2)$. Then by Theorem \ref{pst}, there is a universal constant $\gamma>0$ such that $S_u(z,\gamma) \subset S_u(0, 3/4)$. Using the Aleksandrov maximum principle
 in Theorem \ref{Alekmp} and the gradient bound in Lemma \ref{slope-est}, we find that $|Du|\leq C$ in $S_u(0,3/4)$. It follows that $S_u(z,\gamma)\supset B_{c\gamma}(z)$ for some
 $c$ universal. To prove the theorem, it suffices to consider $y\in S_u(0, 1/2)\cap S_u(z,\gamma)$. In this case, we can assume further that $z=0$, $\gamma=1/2$ and need to show that
 \begin{equation}\label{Dualpha}|Du(y)|\leq C|y|^{\alpha}
  \text{ for all y in } S_u(0, 1/2).
 \end{equation}
By Theorem \ref{C1alpha} , we have
 $$0\leq u(x)\leq C|x|^{1+\alpha}~\text{and}~|u(x)-u(y)-Du(y)\cdot (x-y)|\leq C|x-y|^{1+\alpha}~\text{for all~} x\in S_u(0, 3/4).$$
 We choose $x\in S_u(0, 3/4)$ so that
 $$C(n,\lambda,\Lambda)|y|\geq |x-y|\geq c(n,\lambda,\Lambda)|y|~
 \text{and}~ Du(y)\cdot (x-y)\geq \frac{1}{2}|Du(y)||x-y|.$$
Then
$$\frac{1}{2}|Du(y)||x-y| \leq Du(y)\cdot (x-y) \leq u(y) + u(x) + C|x-y|^{1+\alpha}\leq C|y|^{1+\alpha}.$$
Therefore
$$|Du(y)|\leq C|y|^{\alpha}.$$
\end{proof}

\begin{rem}

Theorem \ref{engulfthm} was extended to the boundary in \cite{LN1}. Theorem \ref{pst}, due to Guti\'errez-Huang \cite{GH},  has been recently extended to the boundary in \cite{Le_Bdr}.
\end{rem}
\begin{rem}
 The second proof of the left inclusion in Theorem \ref{C1alpha}, which avoids compactness arguments, gives an explicit dependence of $\delta$ on $n,\lambda$ and $\Lambda$. Note that
 $\delta$ only depends quantitatively on the explicit constant $\theta$ appearing in Theorem \ref{theta_lem}. As a consequence, the H\"older exponent $\alpha$ in Theorem
 \ref{C1alpha2} can be computed explicitly from the formula $\alpha=\log_2 (1+2\delta)$.
 
 On the other hand, our proof of Theorem \ref{C1alpha2} does not give an explicit dependence of the H\"older norm $C(n,\lambda,\Lambda)$ on $n,\lambda$ and $\Lambda$ because
 $C(n,\lambda,\Lambda)$ depends on the universal constants appearing in the statements of Theorems \ref{sec-size} and \ref{pst} which were obtained by compactness arguments.

 The conclusions of Theorems \ref{C1alpha} and \ref{C1alpha2} also hold when the Monge-Amp\`ere measure $\mu=\det D^2 u$ is doubling with respect to the center of mass 
 on the sections of $u$ as in (\ref{muDC}); see Caffarelli \cite{C3}. We note that in this generality, a direct proof of Theorem \ref{C1alpha2} (which avoids any compactness
argument) has been given by Forzani and Maldonado \cite{FM}, allowing one to compute
the explicit dependence of $\alpha$ and $C(n,\lambda,\Lambda)$ on $n,\lambda$ and $\Lambda$.
\end{rem}

\newpage

\addcontentsline{toc}{part}{\appendixname}

\appendix

\renewcommand{\thesection}{\Alph{section}}

\renewcommand{\theequation}{\arabic{equation}}

\section{Auxiliary Lemmas}
\begin{lem} \label{divfreeU} Let $u$ be a $C^3$ convex function. Let
 $U= (U^{ij})$ denotes the matrix of cofactors of the Hessian matrix
$D^2 u = (u_{ij})$.
Then $U$ is divergence-free, that is, for each $i=1,\cdots, n$, we have
$$\sum_{k=1}^n \partial_k U^{ik}=0.$$

\end{lem}

\begin{proof} By considering $u_{\e}(x) = u(x) +\e |x|^2$ for $\e>0$ and proving the conclusion of the lemma for $u_{\e}$ and then letting $\e\rightarrow 0$, it suffices
to consider the case $D^2u$ is strictly positive definite. In this case,
$$U = (\det D^2 u) (D^2 u)^{-1}.$$
Let us denote
$(D^2 u)^{-1}= (u^{ij}).$
Then, 
$U^{ik}= (\det D^2 u) u^{ik}.$
We know that for each $j=1,\cdots, n$
$$\det D^2 u=\sum_{k=1}^{n} U^{jk}u_{kj},$$
hence
$$\frac{\partial \det D^2 u}{\partial u_{ij}}= U^{ji}= U^{ij}.$$
Thus, from 
$$\frac{\p}{\p x_k} \det D^2 u= U^{rs} u_{rsk}~\text{and}~U^{ik}= (\det D^2 u) u^{ik},$$ we have
\begin{eqnarray*}\partial_k U^{ik} &=& U^{rs} u_{rsk} u^{ik} + (\det D^2 u)\partial_k u^{ik} \\&=& U^{rs}\partial_r (u_{sk}u^{ik}) - U^{rs} u_{sk}\partial_r u^{ik} + (\det D^2 u) \partial_k u^{ik}\\&=& -(\det D^2 u) \delta_{rk}\partial_r u^{ik} + (\det D^2 u) \partial_k u^{ik}=0.
\end{eqnarray*}
In the last line, $\delta_{rs}$ is the Kronecker symbol where $\delta_{rs}=1$ if $r=s$ and $\delta_{rs}=0$ if $r\neq s$.
\end{proof}
\begin{lem}
\label{concavelem} Let $\theta\in[0, \frac{1}{n}]$.
 Let $A$ and $B$ be two nonnegative symmetric $n\times n$ matrices, and $\lambda\in [0, 1]$. Then
 \begin{equation*}
  \left[\det (\lambda A + (1-\lambda)B)\right]^{\theta}\geq \lambda (\det A)^{\theta} + (1-\lambda)(\det B)^{\theta}.
 \end{equation*}
\end{lem}
\begin{proof} 
We first prove the lemma for $\theta=\frac{1}{n}$. Due to the identity $\det (\lambda M)=\lambda^n\det M$ for all $n\times n$ matrices $M$, 
we only need to prove that for two
nonnegative symmetric $n\times n$ matrices $A$ and $B$, we have
\begin{equation}
 \label{reduce1}
 [\det (A + B)]^{\frac{1}{n}}\geq (\det A)^{\frac{1}{n}} + (\det B)^{\frac{1}{n}}.
\end{equation}
It suffices to prove (\ref{reduce1}) for the particular case when $A$ is invertible. In the general case, $A+ \e I_n$ is invertible for each $\e>0$.
Hence 
$$[\det (A +\e I_n + B)]^{\frac{1}{n}}\geq [\det (A+ \e I_n)]^{\frac{1}{n}} + (\det B)^{\frac{1}{n}}.$$
Therefore, letting $\e\rightarrow 0$, we obtain (\ref{reduce1}).

Let us assume now that $A$ is invertible. In view of the identities
$$A + B = A^{1/2} (I_n + D) A^{1/2},~ D= A^{-1/2}B A^{-1/2},$$
and $\det (MN)= (\det M)(\det N)$ for all $n\times n$ matrices $M$ and $N$, (\ref{reduce1}) is a consequence of
\begin{equation}\det ( I_n+ D)^{\frac{1}{n}}\geq 1 + (\det D)^{\frac{1}{n}}
 \label{Dmat}
\end{equation}
for all nonnegative symmetric $n\times n$ matrices $D$. We diagonalize $D$ and let $\lambda_i$ ($i=1,\cdots, n$) be nonnegative eigenvalues of $D$. Then (\ref{Dmat}) reduces
to 
\begin{equation*}
 \prod_{i=1}^n (1 +  \lambda_i)^{\frac{1}{n}}\geq 1 + \prod_{i=1}^{n} \lambda_i^{\frac{1}{n}}.
\end{equation*}
But this is a consequence of the Arithmetic-Geometric inequality, since
$$\prod_{i=1}^n\left(\frac{1}{1+\lambda_i}\right)^{\frac{1}{n}} + \prod_{i=1}^{n} \left(\frac{\lambda_i}{1+\lambda_i}\right)^{\frac{1}{n}}\leq \frac{1}{n}\sum_{i=1}^n
\frac{1}{1+\lambda_i}+ \frac{1}{n}\sum_{i=1}^n
\frac{\lambda_i}{1+\lambda_i} =1.$$

Now, we prove the lemma for the general case $0\leq \theta\leq \frac{1}{n}$. Since the function $f(x) =x^{n\theta}$ is concave on
$[0, \infty)$, we use the result for the case $\theta=\frac{1}{n}$ to get
\begin{eqnarray*}
 \det (\lambda A + (1-\lambda)B)^{\theta} = f\left([\det (\lambda A + (1-\lambda)B)]^{\frac{1}{n}}\right)
 &\geq& f\left(\lambda (\det A)^{\frac{1}{n}} + (1-\lambda)(\det B)^{\frac{1}{n}}\right)\\
 &\geq& \lambda f\left((\det A)^{\frac{1}{n}}\right) + (1-\lambda) f\left( (\det B)^{\frac{1}{n}}\right)\\&=& \lambda (\det A)^{\theta} + (1-\lambda)(\det B)^{\theta}.
\end{eqnarray*}

\end{proof}
\begin{lem}
 \label{trlem}
  Let $A$ and $B$ be two nonnegative symmetric $n\times n$ matrices. Then
  $$\det A \det B\leq \left(\frac{\emph{trace} (AB)}{n}\right)^{n}.$$
\end{lem}
\begin{proof} 
The proof of this lemma is standard. If $M$ is a nonnegative symmetric $n\times n$ matrix then 
\begin{equation}\text{trace} (M)\geq n (\det M)^{1/n}.
 \label{Mtrace}
\end{equation}
Indeed, let $\alpha_1,\cdots,\alpha_n$ be nonnegative eigenvalues of $M$. By 
the Arithmetic-Geometric inequality, 
\begin{equation*}
 \text{trace} (M) =\sum_{i=1}^n \alpha_i\geq n \prod_{i=1}^n\alpha_i^{1/n}= n (\det M)^{1/n}.
\end{equation*}
Returning to our lemma. 
Let $\lambda_1,\cdots,\lambda_n$ be the nonnegative eigenvalues of $A$. There is
an orthogonal matrix $P\in O(n)$ such that $A= P \Lambda P^t$ where $\Lambda= \text{diag}(\lambda_1,\cdots,\lambda_n)$. Then, $D= P^t B P$ is symmetric, nonnegative definite because
for all vector $\xi\in \R^n$, we have
$$D\xi\cdot \xi= (P^t B P \xi)\cdot \xi= (B P\xi)\cdot P\xi\geq 0.$$
Furthermore, $\det D=\det B$. Thus, by (\ref{Mtrace}), we have
$$\text{trace} (AB)=\text{trace} (P\Lambda P^t B)= \text{trace} (\Lambda D)\geq n [\det (\Lambda D)]^{1/n}= n (\det A)^{1/n}(\det B)^{1/n}.$$
\end{proof}
\begin{lem}\label{uv_trace}
 For any symmetric, positive definite $n\times n$ matrix $A=(a_{ij})$ and any vector $b=(b_1,\cdots, b_n)\in \R^n$,
 we have
 $$a_{ij} b_i b_j\geq \frac{|b|^2}{\emph{trace}(A^{-1})}.$$
\end{lem}

\begin{proof}[Proof of Lemma \ref{uv_trace}]
 Let $P$ be an orthogonal matrix such that $A= P D P^t$ where $D=\text{diag} (\lambda_1,\cdots,\lambda_n)$. Then
 $$a_{ij}b_i b_j= Ab\cdot b= (P D P^t b)\cdot b= (D P^t b)\cdot P^t b.$$
 Because $|P^t b|= |b|$, it suffices to prove the lemma for the case $A=D=\text{diag} (\lambda_1,\cdots,\lambda_n)$. In this case, the lemma is equivalent to proving the obvious inequality:
 $$\left(\sum_{i=1}^n\lambda_i^{-1}\right)\left(\sum_{i=1}^n\lambda_i b_i^2\right)\geq \sum_{i=1}^n b_i^2.$$
\end{proof}
\begin{lem} 
\label{ELlem} Consider
the affine area functional over smooth, convex functions $u$ on $\overline{\Omega}$:
$$\mathcal{A}(u,\Omega) =\int_{\Omega}[\det D^2 u(x)]^{\frac{1}{n+ 2}}dx.$$
Then smooth, locally uniformly convex critical points $u$ of $\mathcal{A}(\cdot,\Omega)$
satisfy the Euler-Lagrange equation
\begin{equation*}\sum_{i, j=1}^{n}\frac{\partial^2}{\partial x_i \partial x_j}(U^{ij} w)=0,~ w = [\det D^2 u]^{-\frac{n+1}{n +2}}
\end{equation*}
where $U= (U^{ij})$ denotes the matrix of cofactors of the Hessian matrix
$D^2 u= (u_{ij}).$
\end{lem}
\begin{proof} Smooth, locally uniformly convex critical points $u$ of $\mathcal{A}(\cdot,\Omega)$ satisfies for all $\varphi\in C^{\infty}_{0}(\Omega)$
$$\frac{d}{dt}\mid_{t=0} \mathcal{A}(u+ t\varphi,\Omega)=0.$$
We compute, denoting $F(t) = t^{\frac{1}{n+2}}$
\begin{eqnarray*}\frac{d}{dt}\mid_{t=0} \mathcal{A}(u+ t\varphi,\Omega)&=& \int_{\Omega} F^{'}(\det D^2 (u+t\varphi)) \frac{d}{dt} \det D^2 (u+t\varphi)\mid_{t=0}\\ &=&
\int_{\Omega} F^{'}(\det D^2 u) \frac{\partial \det D^2 (u+t\varphi)}{\partial u_{ij}}\varphi_{ij}\mid_{t=0}\\
&=& \int_{\Omega} F^{'}(\det D^2 u) U^{ij}\varphi_{ij}.
\end{eqnarray*}
In the last equality, we used the following identity from the proof of Lemma \ref{divfreeU},
$$\frac{\partial \det D^2 u}{\partial u_{ij}}= U^{ji}= U^{ij}.$$
Let
$$W^{ij}= F^{'}(\det D^2 u) U^{ij} =\frac{1}{n+2} (\det D^2 u)^{-\frac{n+1}{n+2}} U^{ij}.$$
Then, since $\varphi$ and all of its derivatives vanish on $\partial\Omega$, we integrate by parts twice to obtain
\begin{eqnarray*}
0=\int_{\Omega} F^{'}(\det D^2 u) U^{ij}\varphi_{ij} = \int_{\Omega} W^{ij}\varphi_{ij}=\int_{\Omega}\partial_{i}\partial_{j} W^{ij}\varphi.
\end{eqnarray*}
This is true for all  $\varphi\in C^{\infty}_{0}(\Omega)$, hence for $w=(\det D^2 u)^{-\frac{n+1}{n+2}}$, we have
$$0=\partial_{i}\partial_{j} W^{ij} =\frac{1}{n+ 2} \partial_{i}\partial_{j} ( U^{ij} w).$$
\end{proof}

\begin{lem}
\label{A_inv}
 The functional
 $$\mathcal{A}(u,\Omega)= \int_{\Omega} G(\det D^2 u(x)) dx,$$
 where
 $$G(t) = t^{\frac{1}{n+2}},$$
 defined over smooth and convex functions $u$ on $\Omega\subset\R^n$, 
 is invariant under the unimodular affine transformations in $\R^{n+1}$.
\end{lem}
\begin{proof}
 Let us denote by $ASL(n+1)$ the group of unimodular affine transformations in $\R^{n+1}$. Note that $ASL(n+1)$ is generated by $ASO(n+1)$-the group of translations and 
 proper rotations
 in $\R^{n+1}$, and the linear transformation group $M$ of $\R^{n+1}$ mapping the point $(x_1,\cdots, x_n, x_{n+1})$ of $\R^{n+1}$ onto the one with
 coordinates $(\lambda_1 x_1, \cdots, \lambda_n x_n, (\lambda_1\cdots\lambda_n)^{-1} x_{n+1})$ for any $\lambda_1,\cdots,\lambda_n>0$.
 
 We first verify that $\mathcal{A}$ is invariant under $ASO(n+1)$. Let
 $$\mathcal{M}=\{(x, u(x))| x\in\Omega\}$$
 be the graph of $u$ over $\Omega$. Let $d\Sigma$ and $K$ be the volume element and Gauss curvature of $\mathcal{M}$, respectively. At $(x, u(x))\in\mathcal{M}$, we have
 $$K(x) = \frac{\det D^2 u(x)}{(1 + |Du(x)|^2)^{\frac{n+2}{2}}}.$$
 Then, for general $G$,
 \begin{eqnarray*}
  \mathcal{A}(u,\Omega)=\int_{\Omega} G(\det D^2 u(x)) dx &=& \int_{\Omega} G(K(x) (1+ |Du(x)|^2)^{\frac{n+2}{2}}) dx\\
  &=& \int_{\mathcal{M}} G(K(x) (1+ |Du(x)|^2)^{\frac{n+2}{2}}) (1+ |Du(x)|^2)^{-\frac{1}{2}} d\Sigma.
  \end{eqnarray*}
In the particular case of $G(t)=t^{\frac{1}{n+2}},$
$$\mathcal{A}(u,\Omega)=\int_{\mathcal{M}} K^{\frac{1}{n+2}} d\Sigma,$$
which is clearly invariant under $ASO(n+1)$.

Note that for $G(t)= t^{\alpha}$ where $\alpha\neq \frac{1}{n+2}$, $\mathcal{A}$ is not invariant under $ASO(n+1)$.

Finally, we verify that $\mathcal{A}$ is invariant under $M$. For any $\lambda=(\lambda_1,\cdots,\lambda_n)$ with $\lambda_1,\cdots,\lambda_n>0$, 
let $$\Omega_{\lambda}=\{(\lambda_1 x_1,\cdots,\lambda_n x_n)|x=(x_1,\cdots, x_n)\in\Omega\}.$$ The
image of the graph $\mathcal{M}$ under the mapping
$$(x_1,\cdots, x_n, x_{n+1})\longmapsto (\lambda_1 x_1, \cdots, \lambda_n x_n, (\lambda_1\cdots\lambda_n)^{-1} x_{n+1})$$
is the graph
$$\mathcal{M}_{\lambda} =\{(\lambda_1 x_1, \cdots, \lambda_n x_n, (\lambda_1\cdots\lambda_n)^{-1} u(x))| x\in\Omega\}:= \{(y, u_\lambda(y))| y\in\Omega_{\lambda}\}$$
of the function $u_{\lambda}$ defined over $\Omega_{\lambda}$ where
$$u_\lambda(y)= (\lambda_1\cdots\lambda_n)^{-1} u(\frac{y_1}{\lambda_1},\cdots,\frac{y_n}{\lambda_n}).$$
Clearly, by simple computations and changes of variables, we have
\begin{eqnarray*}
 \mathcal{A}(u_{\lambda},\Omega_{\lambda})=\int_{\Omega_{\lambda}} [\det D^2 u_\lambda (y)]^{\frac{1}{n+2}} dy &=&\int_{\Omega_{\lambda}} 
 (\lambda_1\cdots \lambda_n)^{-1} [\det D^2 u(\frac{y_1}{\lambda_1},\cdots,\frac{y_n}{\lambda_n})]^{\frac{1}{n+2}} dy
 \\ &=& \int_{\Omega} [\det D^2 u(x)]^{\frac{1}{n+2}} dx= \mathcal{A}(u,\Omega).
\end{eqnarray*}
Thus  $\mathcal{A}$ is invariant under the group of transformations $M$. The proof of our lemma is complete.
\end{proof}

\section{A heuristic explanation
of Trudinger-Wang's non-smooth example}
In this appendix, we provide a heuristic explanation
of Trudinger-Wang's example of non-smooth convex solutions to the affine maximal surface equation
\begin{equation}\sum_{i, j=1}^{n}\frac{\partial^2}{\partial x_i \partial x_j}(U^{ij} w)=0, w = [\det D^2 u]^{-\frac{n+1}{n +2}}; U= (U^{ij})= (\det D^2 u) (D^2 u)^{-1}.
 \label{AMSE2}
\end{equation}

Trudinger and Wang found in \cite{TW00} that the non-smooth convex function $u(x)= \sqrt{|x'|^9 + x_{10}^2}$, where
$x'= (x_1, \cdots, x_9)$, satisfies (\ref{AMSE2}) in $\R^{10}$ and is not differentiable at the origin.

This explanation is based on simple symmetry and scaling arguments and it is reminiscent of Pogorelov's singular solution (see \cite{Pogo})
of the form 
\begin{equation}u(x', x_n)=|x'|^{2-2/n} f(x_n)
 \label{pogo_ex}
\end{equation}
 to the Monge-Amp\`ere equation 
 \begin{equation}\det D^2 u=1.
  \label{det1}
 \end{equation}
Here and what follows, we denote a point $x\in\R^n$ by $$x=(x', x_n)~ \text{where}~ x'= (x_1,\cdots, x_{n-1}).$$
The equation (\ref{det1}) is invariant under the rescalings of $u$ given 
 by $$u_{\lambda} (x', x_n)= \lambda^{2/n-2} u(\lambda x', x_n).$$ 
The Pogorelov example in (\ref{pogo_ex}) is invariant under these rescalings.

Let us now return to (\ref{AMSE2}). We denote a point $x\in\R^n$ by
$$x= (y, z)~\text{where}~y\in \R^k,~z\in\R^l~\text{with}~k+ l=n.$$
We easily find that (\ref{AMSE2}), or more generally,
$$\sum_{i, j=1}^{n}\frac{\partial^2}{\partial x_i \partial x_j}(U^{ij} w)=\text{constant m}, w = [\det D^2 u]^{-\frac{n+1}{n +2}}; U= (U^{ij})= (\det D^2 u) (D^2 u)^{-1}$$
is invariant under rescalings
\begin{equation}u_\lambda (y, z)= \lambda^l u(y, \lambda z),~\lambda>0.
 \label{lambda_kl}
\end{equation}

One way to see this is the following. We note that the function $u$ satisfying (\ref{AMSE2}) is a critical point of the affine area functional
$$\mathcal{A}(u,\Omega)=\int_{\Omega}[\det D^2 u(x)]^{\frac{1}{n+2}} dx$$
and that this functional is invariant under the following one-parameter transformation group of $\R^{n+1}$ mapping the point $(y, z, x_{n+1})$ of $\R^{n+1}$ onto the one with
 coordinates $(y, \lambda^{-1} z, \lambda^{l} x_{n+1})$ for any $\lambda>0$. 
 
The function
\begin{equation}v(y, z)= \frac{h(y)}{|z|^l}
 \label{vform1}
\end{equation}
is invariant under the rescalings (\ref{lambda_kl}) of (\ref{AMSE2}). We can look for convex solutions to (\ref{AMSE2}) of the form (\ref{vform1}). Since the function 
$\frac{1}{|z|^l} (z\in \R^l)$ needs to be convex, we must have $l=1$. Another reason (without using the convexity of $v$) to consider the case $l=1$ only is the following. 
We might try to look for solutions to (\ref{AMSE2}) of the form (\ref{vform1}) with $2\leq l<n$ but in this case, it seems impossible to resolve the singularity of $v$ at the origin
 using the unimodular affine transformations in $\R^{n+1}$. 
 
To simplify computations, we only consider radial $h(y)$. Thus, with further simplifications, we are led to finding a solution to (\ref{AMSE2})
of the form
\begin{equation}v(y', y_n)= \frac{|y'|^{2\alpha}}{|y_n|}
 \label{vform2}
\end{equation}
where $y=(y', y_n)\in\R^n$ with $y'=(y_1,\cdots, y_{n-1})$. 
  
  Since we require $v$ to be convex, we can impose $\alpha\geq 1$. Computing as in \cite{TW00}, we find that $v$ solves (\ref{AMSE2}) (away from the origin)
  if
  $$8\alpha^2 -(n^2-4n+ 12)\alpha + 2(n-1)^2=0$$
  which is solvable for $n\geq 10$. When $n=10$, we have $\alpha=\frac{9}{2}$
  and hence
  \begin{equation}v(y)=\frac{|y'|^9}{2 |y_{10}|}
   \label{fin_v}
  \end{equation}
  is a solution to (\ref{AMSE2}) (away from the origin). We need some more work to show that it solves (\ref{AMSE2}) weakly on the whole space $\R^{10}$; see
  \cite{TW00} for more details.
  
Recall that (\ref{AMSE2}) and the affine area functional $\mathcal{A}$ are also invariant under the rotations in $\R^{n+1}$. Thus, the rotation  
\begin{equation*}
\left\{
 \begin{alignedat}{1}
   x' ~& = y', \\\
x_{10} ~&= \frac{1}{\sqrt{2}} (y_{10}-v),\\\
u~&= \frac{1}{\sqrt{2}} (y_{10} + v)
 \end{alignedat} 
  \right.
  \end{equation*}
 transforms $v$ in (\ref{fin_v})
 to the Trudinger-Wang singular function $$u(x)= \sqrt{|x'|^9 + x_{10}^2}.$$
 This $u$ is merely Lipschitz and solves the affine maximal surface equation (\ref{AMSE2}).

\end{document}